\documentclass[11pt,a4paper,oneside,english]{amsart}
\usepackage[T1]{fontenc}
\usepackage[latin9]{inputenc}
\usepackage{babel}
\usepackage{amsthm}
\usepackage{amstext}
\usepackage{amssymb}
\usepackage{graphicx}
\usepackage{esint}
\usepackage[all]{xy}
\usepackage[unicode=true,pdfusetitle,
 bookmarks=true,bookmarksnumbered=false,bookmarksopen=false,
 breaklinks=false,pdfborder={0 0 1},backref=false,colorlinks=false]
 {hyperref}

\makeatletter

\pdfpageheight\paperheight
\pdfpagewidth\paperwidth

\numberwithin{equation}{section}
\numberwithin{figure}{section}
\theoremstyle{plain}
\newtheorem{thm}{\protect\theoremname}[section]
  \theoremstyle{remark}
  \newtheorem{rem}[thm]{\protect\remarkname}
  \theoremstyle{definition}
  \newtheorem{defn}[thm]{\protect\definitionname}
  \theoremstyle{plain}
  \newtheorem{prop}[thm]{\protect\propositionname}
  \theoremstyle{remark}
  \newtheorem*{acknowledgement*}{\protect\acknowledgementname}
  \theoremstyle{plain}
  \newtheorem{lem}[thm]{\protect\lemmaname}
  \theoremstyle{plain}
  \newtheorem{cor}[thm]{\protect\corollaryname}

\advance\hoffset by -0.5 truecm

\usepackage{times}
\usepackage{mathrsfs}
\usepackage{bbold}

\makeatother

  \providecommand{\acknowledgementname}{Acknowledgement}
  \providecommand{\corollaryname}{Corollary}
  \providecommand{\definitionname}{Definition}
  \providecommand{\lemmaname}{Lemma}
  \providecommand{\propositionname}{Proposition}
  \providecommand{\remarkname}{Remark}
\providecommand{\theoremname}{Theorem}

\begin{document}

\title{On the growth rate of leaf-wise intersections}

\author{Leonardo Macarini, Will J. Merry, and Gabriel P. Paternain}

\address{(L. Macarini) Universidade Federal do Rio de Janeiro, Instituto de
Matem\'atica, Cidade Universit\'aria, CEP 21941-999, Rio de Janeiro,
Brazil }

\email{\texttt{leonardo@impa.br}}

\address{(W. J. Merry) Department of Pure Mathematics and Mathematical Statistics,
University of Cambridge, Cambridge CB3 0WB, England }

\email{\texttt{w.merry@dpmms.cam.ac.uk}}

\address{(G. P. Paternain) Department of Pure Mathematics and Mathematical
Statistics, University of Cambridge, Cambridge CB3 0WB, England }

\email{\texttt{g.p.paternain@dpmms.cam.ac.uk}}
\begin{abstract}
We define a new variant of Rabinowitz Floer homology that is particularly
well suited to studying the growth rate of leaf-wise intersections.
We prove that for closed manifolds $M$ whose loop space $\Lambda M$
is {}``complicated'', if $\Sigma\subseteq T^{*}M$ is a non-degenerate
fibrewise starshaped hypersurface and $\varphi\in\mbox{Ham}_{c}(T^{*}M,\omega)$
is a generic Hamiltonian diffeomorphism then the number of leaf-wise
intersection points of $\varphi$ in $\Sigma$ grows exponentially
in time. Concrete examples of such manifolds are $(S^{2}\times S^{2})\#(S^{2} \times S^{2})$,
$\mathbb{T}^{4}\#\mathbb{C}P^{2}$, or any surface of genus greater
than one.
\end{abstract}
\maketitle

\section{Introduction}

Let $M$ denote a closed connected orientable $n$-dimensional manifold
with cotangent bundle $\pi:T^{*}M\rightarrow M$. Let $\lambda=pdq$
and $Y=p\partial_{p}$ denote the Liouville $1$-form and Liouville
vector field on $T^{*}M$ respectively, and let $\omega=d\lambda$
denote the canonical symplectic structure. Note that $i_{Y}\omega=\lambda$.
Let $\mbox{Ham}_{c}(T^{*}M,\omega)$ denote the group of compactly
supported Hamiltonian diffeomorphisms of $T^{*}M$.

Recall that a \textbf{fibrewise starshaped} \textbf{hypersurface}
$\Sigma$\textbf{ }is a closed connected separating hypersurface in
$T^{*}M$ such that $Y$ is transverse to $\Sigma$ and points in
the outwards direction. This is equivalent to requiring that $\lambda_{\Sigma}:=\lambda|_{\Sigma}$
is a positive contact form on $\Sigma$. Given a fibrewise starshaped
hypersurface $\Sigma$, let $R_{\Sigma}$ denote the Reeb vector field\textbf{
}associated to the contact $1$-form $\lambda_{\Sigma}$. Let $\phi_{t}^{\Sigma}:\Sigma\rightarrow\Sigma$
denote the flow of $R_{\Sigma}$. We say that $\Sigma$ is a \textbf{non-degenerate
}hypersurface if all the closed orbits of $R_{\Sigma}$ are transversely
non-degenerate (see Definition \ref{def:non degen} below). Given
$p\in\Sigma$, let $\mathcal{L}_{p}$ denote the leaf of the characteristic
foliation of $\Sigma$ running through $p$. We can parametrize $\mathcal{L}_{p}$
via $\mathcal{L}_{p}:=\{\phi_{t}^{\Sigma}(p)\,:\, t\in\mathbb{R}\}$.
A \textbf{defining Hamiltonian }for $\Sigma$ is an autonomous Hamiltonian
$F\in C^{\infty}(T^{*}M,\mathbb{R})$ such that $\Sigma=F^{-1}(0)$
and such that the Hamiltonian vector field $X_{F}$ is compactly supported
and satisfies $X_{F}|_{\Sigma}=R_{\Sigma}$.\newline 

Given $\varphi\in\mbox{Ham}_{c}(T^{*}M,\omega)$, we say that a point
$p\in\Sigma$ is a \textbf{leaf-wise intersection point }for $\varphi$
if there exists a real number $\eta\in\mathbb{R}$ such that \begin{equation}
\varphi(\phi_{\eta}^{\Sigma}(p))=p.\label{eq:lwp}
\end{equation}
We say that $p$ is a \textbf{periodic }leaf-wise intersection point\textbf{
}if $\mathcal{L}_{p}$ is a closed leaf. In this paper we will only
be interested in leaf-wise intersection points that are \textbf{not}
periodic. This is not a major restriction, as Albers and Frauenfelder
(see \cite[Theorem 3.3]{AlbersFrauenfelder2008} or Proposition \ref{prop:-generically no plwip}
below) show that if $n=\dim\, M\geq2$ and $\Sigma\subseteq T^{*}M$
is a non-degenerate fibrewise starshaped hypersurface then a generic
Hamiltonian diffeomorphism has no periodic leaf-wise intersection
points in $\Sigma$. Thus for simplicity the term {}``leaf-wise interection
point'' should be understood as {}``non-periodic leaf-wise intersection
point'', unless explicitly stated otherwise. With this convention
in mind, the \textbf{time-shift }$\eta\in\mathbb{R}$ of a leaf-wise
intersection point $p$ is the unqiue%
\footnote{Of course, without the implicit {}``non-periodic'' in front of the
term {}``leaf-wise intersection point'' $\eta$ is not unique: if
$\phi_{T}^{\Sigma}(p)=p$ then $\varphi(\phi_{\eta+kT}^{\Sigma}(p))=p$
for all $k\in\mathbb{Z}$.%
} real number $\eta$ such that \eqref{eq:lwp} is satisfied. real
number $\eta\in\mathbb{R}$ such that \[
\varphi(\phi_{\eta}^{\Sigma}(p))=p.
\]
A leaf-wise intersection point has zero time-shift if and only if
it is a fixed point of $\varphi$. A leaf-wise intersection point
is called \textbf{positive }if its time-shift $\eta$ is strictly
positive, and \textbf{negative }if its time-shift is strictly negative.
In this paper we will only be interested in positive leaf-wise intersection
points. This is no great loss, as the negative leaf-wise intersection
points of $\varphi$ are precisely the positive leaf-wise intersection
points of $\varphi^{-1}$.
\begin{rem}
Our definition of a leaf-wise intersection point is slightly different
to the standard one, where rather than referring to $p$ as the leaf-wise
intersection point, instead the point $\bar{p}:=\phi_{\eta}^{\Sigma}(p)$
is called {}``the leaf-wise intersection point''. With this convention
a point $\bar{p}$ is a leaf-wise intersection point if $\varphi(\bar{p})\in\mathcal{L}_{\bar{p}}$,
which is perhaps a more natural definition. However using the standard
convention it would seem natural (see \cite[p1]{AlbersFrauenfelder2010})
to define the {}``time-shift'' of $\bar{p}$ to be $-\eta$ rather
then $\eta$, and as a result with the standard definition we would
end up counting negative leaf-wise intersection points, which is somehow
less aesthetically pleasing (see the statement of Theorem A below).
\end{rem}
The leaf-wise intersection problem asks whether a given Hamiltonian
diffeomorphism always has a leaf-wise intersection point in a given
fibrewise starshaped hypersurface, and if so, whether one can obtain
a lower bound on the number of such leaf-wise intersections. This
problem was introduced by Moser in \cite{Moser1978}, and since then
has been studied by a number of different authors \cite{Bangaya1980,EkelandHofer1989,Hofer1990,Ginzburg2007,Dragnev2008,Gurel2010,Ziltener2010,AlbersFrauenfelder2010,AlbersFrauenfelder2008,AlbersFrauenfelder2010a,AlbersFrauenfelder2010b,Kang2009,Kang2010,Kang2010b,Merry2010a}.
We refer to \cite{AlbersFrauenfelder2010c} for a brief history of
the problem and a discussion of the progress made so far. Here we
mention only one result that is particularly relevant to our paper:
in \cite{AlbersFrauenfelder2008} Albers and Frauenfelder establish
that if the homology of the free loop space is infinite dimensional,
then given a non-degenerate fibrewise starshaped hypersurface $\Sigma$,
a generic Hamiltonian diffeomorphism has infinitely many leaf-wise
intersection points in $\Sigma$. This appears to have been the first
result which asserts the existence of \textbf{infinitely} many leaf-wise
intersection points, instead of just a finite lower bound. In this
paper we extend this result to show that if the base manifold $M$
satisfies a certain topological condition (roughly that its loop space
homology is sufficiently {}``complicated'' - concrete examples of
such manifolds are $(S^{2}\times S^{2})\#(S^{2}\times S^{2})$, $\mathbb{T}^{4}\#\mathbb{C}P^{2}$
or any surface of genus greater than one), then not only do generic
Hamiltonian diffeomorphisms have infinitely many leaf-wise intersection
points in any non-degenerate fibrewise starshaped hypersurface, but
the number of such leaf-wise intersection points {}``grows'' exponentially
with time. The precise statements are given below in Theorem A and
Corollaries B and C. To the best of our knowledge this is the first
result which establishes the existence of {}``more'' than just infinitely
many leaf-wise intersection points.\newline

Let us fix $\varphi\in\mbox{Ham}_{c}(T^{*}M,\omega)$. Suppose $H\in C_{c}^{\infty}(S^{1}\times T^{*}M,\mathbb{R})$
is any Hamiltonian that \textbf{generates }$\varphi$, i.e. $\phi_{1}^{H}=\varphi$.
If $p$ is a positive leaf-wise intersection point of $\varphi$ with
time-shift $\eta$ then consider the (not necessarily smooth) loop
$x\in C^{0}(S^{1},T^{*}M)$ defined by \[
x(t):=\begin{cases}
\phi_{2t\eta}^{\Sigma}(p), & 0\leq t\leq1/2,\\
\phi_{2t-2}^{H}(p), & 1/2\leq t\leq1.
\end{cases}
\]
Obviously the curve $x$ depends on the choice of Hamiltonian $H$
generating $\varphi$, but asking which free homotopy class $\alpha\in[S^{1},M]$
the projection $\pi\circ x$ belongs to is independent of $H$ (see
Lemma \ref{lem:Independence lemma} below). Thus it makes sense to
speak of\textbf{ }leaf-wise intersection points \textbf{belonging}
to\textbf{ $\alpha$}. Given $T>0$ denote by by $n_{\Sigma,\alpha}(\varphi,T)$
the number of positive leaf-wise intersection points that belong to
$\alpha$ with time-shift $0<\eta<T$. As indicated above, in this
paper we study the \textbf{growth rate} of the function $n_{\Sigma,\alpha}(\varphi,\cdot)$
for a given $\varphi\in\mbox{Ham}_{c}(T^{*}M,\omega)$. In order to
state our results we first need to introduce several definitions.
Denote by $\Lambda T^{*}M$ the free loop space of $T^{*}M$. Given
$H\in C_{c}^{\infty}(S^{1}\times T^{*}M,\mathbb{R})$, denote by $A^{H}:\Lambda T^{*}M\rightarrow\mathbb{R}$
the standard Hamiltonian action functional \begin{equation}
A^{H}(x):=\int x^{*}\lambda-\int_{0}^{1}H(t,x)dt.\label{eq:standard action functional}
\end{equation}
Denote by $\mathcal{A}(A^{H})$ the \textbf{action spectrum }of $A^{H}$:\[
\mathcal{A}(A^{H}):=\left\{ A^{H}(x)\,:\, x\mbox{ is a critical point of }A^{H}\right\} .
\]
Now suppose $\varphi\in\mbox{Ham}_{c}(T^{*}M,\omega)$. A theorem
of Frauenfelder and Schlenk \cite[Corollary 6.2]{FrauenfelderSchlenk2007}
says that if $H,K\in C_{c}^{\infty}(S^{1}\times T^{*}M,\mathbb{R})$
both \textbf{generate} $\varphi$ then%
\footnote{Strictly speaking their result pertains only to the subset of the
action spectrum generated by \textbf{contractible }periodic points.
But they work only with a \textbf{weakly exact }symplectic manifold.
In our case the symplectic form is exact (instead of just being weakly
exact), and thus the same proof carries through for the entire action
spectrum. We also remark that the same result is also true for \textbf{closed
}symplectically aspherical manifolds (see \cite[Theorem 1.1]{Schwarz2000},
which builds on Seidel \cite{Seidel1997}), although this is considerably
deeper.%
} \[
\mathcal{A}(A^{H})=\mathcal{A}(A^{K}).
\]
Thus we may define the action spectrum $\mathcal{A}(\varphi)$ of
$\varphi$ to be $\mathcal{A}(A^{H})$ for any $H\in C_{c}^{\infty}(S^{1}\times T^{*}M,\mathbb{R})$
generating $\varphi$. Now define \[
\kappa:\mbox{Ham}_{c}(T^{*}M,\omega)\rightarrow[0,\infty)
\]
by \begin{equation}
\kappa(\varphi):=\sup\left\{ \left|\eta\right|\,:\,\eta\in\mathcal{A}(\varphi)\right\} .\label{eq:kappaphi}
\end{equation}
Another way of measuring the {}``size'' of an element $\varphi\in\mbox{Ham}_{c}(T^{*}M,\omega)$
is given by the \textbf{Hofer norm}. We recall the definition: given
$H\in C_{c}^{\infty}(S^{1}\times T^{*}M,\mathbb{R})$, define\[
\left\Vert H\right\Vert _{+}:=\int_{0}^{1}\max_{(q,p)\in T^{*}M}H(t,q,p)dt,\ \ \ \left\Vert H\right\Vert _{-}:=-\int_{0}^{1}\min_{(q,p)\in T^{*}M}H(t,q,p)dt;
\]
\[
\left\Vert H\right\Vert :=\left\Vert H\right\Vert _{+}+\left\Vert H\right\Vert _{-}.
\]
For $\varphi\in\mbox{Ham}_{c}(T^{*}M,\omega)$, the Hofer norm of
$\varphi$ is defined to be:\begin{equation}
\left\Vert \varphi\right\Vert :=\inf\left\{ \left\Vert H\right\Vert \,:\, H\mbox{ generates }\varphi\right\} .\label{eq:hofer norm}
\end{equation}
Let us combine these two measures together and define \begin{equation}
\mu(\varphi):=2\kappa(\varphi)+6\left\Vert \varphi\right\Vert .\label{eq:mu phi}
\end{equation}
Write $\Lambda M$ for the free loop space of $M$ and $\Lambda_{\alpha}M$
the subspace of loops belonging to the free homotopy class $\alpha$.
Given a metric $g$ on $M$ define the energy functional\[
\mathcal{E}_{g}:\Lambda M\rightarrow\mathbb{R};
\]
\[
\mathcal{E}_{g}(q):=\int_{0}^{1}\frac{1}{2}\left|\dot{q}\right|^{2}dt.
\]
Given $0<a<\infty$ and $\alpha\in[S^{1},M]$, denote by \[
\Lambda_{\alpha}^{a}(M,g):=\left\{ q\in\Lambda_{\alpha}M\,:\,\mathcal{E}_{g}(q)\leq\frac{1}{2}a^{2}\right\} .
\]
We will prove the following theorem.\newline\newline\textbf{Theorem
A. }\emph{Let $M$ be a closed connected orientable manifold of dimension
$n\geq2$. Let $\Sigma$ be a non-degenerate fibrewise starshaped
hypersurface. Let $g$ be a bumpy Riemannian metric on $M$ with $S_{g}^{*}M$
contained in the interior of the compact region bounded by $\Sigma$.
There exists a constant $c=c(\Sigma,g)>0$ such that the following
property holds: Suppose $\varphi\in\mbox{\emph{Ham}}_{c}(T^{*}M,\mathbb{R})$
is a generic Hamiltonian diffeomorphism (see Remark \ref{allowing plwip}
for the precise meaning of the word {}``generic'' in this context).
Then for all $T>0$ sufficiently large, it holds that \begin{equation}
n_{\Sigma,\alpha}(\varphi,T)\geq\begin{cases}
\mbox{\emph{rank}}\left\{ \iota:H(\Lambda_{\alpha}^{c(T-\left\Vert \varphi\right\Vert )}(M,g);\mathbb{Z}_{2})\rightarrow H(\Lambda_{\alpha}M,\Lambda_{\alpha}^{4\mu(\varphi)}(M,g);\mathbb{Z}_{2})\right\} , & \alpha\ne0,\\
\mbox{\emph{rank}}\left\{ \iota:H(\Lambda_{0}^{c(T-\left\Vert \varphi\right\Vert )}(M,g),M;\mathbb{Z}_{2})\rightarrow H(\Lambda_{0}M,\Lambda_{0}^{4\mu(\varphi)}(M,g);\mathbb{Z}_{2})\right\} , & \alpha=0.
\end{cases}\label{eq:theorem A}
\end{equation}
}
\begin{rem}
Theorem A is proved only for $\mathbb{Z}_{2}$ coefficients. This
is because so far there is no treatment of coherent orientations for
Rabinowitz Floer homology, but we certainly expect the theorem to
hold with any field of coefficients. Because of this however, for
the remainder of this paper the notation  $H(X,A)$ for the singular
homology of a pair $(X,A)$ should always be understood as shorthand
for $H(X,A;\mathbb{Z}_{2})$. 
\end{rem}

\begin{rem}
\label{allowing plwip}As mentioned above, a generic Hamiltonian diffeomorphism
has no periodic leaf-wise intersection points, and hence it is sufficient
to prove Theorem A for Hamiltonian diffeomorphisms with no periodic
leaf-wise intersection points. In fact, we prove Theorem A for Hamiltonian
diffeomorphisms that (a) have no periodic leaf-wise intersection points
and (b) are generated by Hamiltonians for which the corresponding
Rabinowitz action functional is Morse (this condition is also generic
- again due to Albers and Frauenfelder \cite[Proposition 3.9]{AlbersFrauenfelder2010}).
The precise definition for the subset of Hamiltonian diffeomorphisms
for which we prove Theorem A is given in Definition \ref{what we prove theorem a for}
below. 
\end{rem}

\begin{rem}
\label{finite dim}A well known result which is essentially due to
Morse \cite{Morse1932} says that for any Riemannian manifold $(M,g)$
and for any $a>0$ the space $\Lambda^{a}(M,g)$ is finite-dimensional.
For the case of based loops a proof of this can be found in Milnor's
book \cite{Milnor1963}. A complete proof for the free loop space
is given in \cite{GoreskyHingston2009}. Thus the growth rate of $n_{\Sigma,\alpha}(\varphi,T)$
is also bounded from below by the growth rate of the function \[
T\mapsto\mbox{rank}\left\{ \iota:H(\Lambda_{\alpha}^{c(T-\left\Vert \varphi\right\Vert )}(M,g))\rightarrow H(\Lambda_{\alpha}M)\right\} .
\]

\end{rem}
Under certain topological assumptions on $M$, the number on the right-hand
side of \eqref{eq:theorem A} grows exponentially with $T$. For instance,
if $M$ is simply connected then a classical theorem of Gromov \cite{Gromov1978}
implies that whenever the Betti numbers $(b_{i}(\Lambda_{\alpha}M))_{i\in\mathbb{Z}}$
grow exponentially with $i$, the right-hand side of \eqref{eq:theorem A}
grows exponentially with $T$. In the simply connected case, various
results giving exponential growth of the Betti numbers $(b_{i}(\Lambda_{0}M))_{i\in\mathbb{Z}}$
have been obtained by Lambrechts \cite{Lambrechts2001a,Lambrechts2001};
a concrete example is $(S^{2}\times S^{2})\#(S^{2}\times S^{2})$.
In the non-simply connected case there are also plenty of examples
where the right-hand side of \eqref{eq:theorem A} with $\alpha=0$
still grows exponentially with $T$; see for instance \cite{PaternainPetean2005}.
To encapsulate the situation where Theorem A gives exponential growth,
following \cite{FrauenfelderSchlenk2006} we make the following definition.
\begin{defn}
\label{energy hyp}Given a closed Riemannian manifold $(M,g)$ and
$\alpha\in[S^{1},M]$ we define \[
C_{\Lambda,\alpha}(M,g):=\liminf_{a\rightarrow\infty}\frac{\log\,\mbox{rank}\left\{ \iota:H(\Lambda_{\alpha}^{a}(M,g))\rightarrow H(\Lambda_{\alpha}M)\right\} }{a}\in[0,\infty].
\]
Whilst the constant $C_{\Lambda,\alpha}(M,g)$ depends on $g$, asking
whether $C_{\Lambda,\alpha}(M,g)$ is positive or not is a purely
topological question. Thus we say that $M$ is \textbf{$(\Lambda,\alpha)$-energy
hyperbolic }if $C_{\Lambda,\alpha}(M,g)>0$ for some (and hence any)
Riemannian metric $g$ on $M$. 
\end{defn}
The following result can be proved in exactly the same way as \cite[Theorem B]{PaternainPetean2005},
and gives a wide class of Riemannian manifolds $M$ which are $(\Lambda,0)$-energy
hyperbolic.
\begin{prop}
\label{prop:PPresult}Let $M$ be a closed manifold of dimension $n\geq3$.
Suppose that $M$ can be decomposed as $N_{1}\#N_{2}$, where $\pi_{1}(N_{1})$
has a subgroup of finite index $\geq3$, and $N_{2}$ is a simply
connected manifold that is not a homology $\mathbb{Z}_{2}$-sphere.
Then $M$ is $(\Lambda,0)$-energy hyperbolic\emph{. }
\end{prop}
Note that $M=\mathbb{T}^{4}\#\mathbb{C}P^{2}$ satisfies the hypotheses
of Proposition \ref{prop:PPresult}. An immediate corollary of Remark
\ref{finite dim} and Theorem A is the following result, which, as
far as we are aware, is new even in the case $\Sigma=S_{g}^{*}M$.\newline\newline\textbf{Corollary
B. }\emph{Let $M$ be a closed connected orientable manifold of dimension
$n\geq2$ and fix $\alpha\in[S^{1},M]$. Assume $M$ is $(\Lambda,\alpha)$-energy
hyperbolic. Let $\Sigma\subseteq T^{*}M$ be a non-degenerate fibrewise
starshaped hypersurface. If $\varphi\in\mbox{\emph{Ham}}_{c}(T^{*}M,\mathbb{R})$
is a generic Hamiltonian diffeomorphism then $n_{\Sigma,\alpha}(\varphi,T)$
grows exponentially with $T$.}\newline

If we don't fix the free homotopy class $\alpha\in[S^{1},M]$ then
another source of examples for which we obtain an exponential growth
rate of leaf-wise intersections occurs when the fundamental group
modulo conjugacy of $M$ has exponential growth. In order to explain
this more precisely, let us first say that a smooth manifold $M$
is \textbf{$\Lambda$-energy hyperbolic }if \[
C_{\Lambda}(M,g):=\liminf_{a\rightarrow\infty}\frac{\log\,\mbox{rank}\left\{ \iota:H(\Lambda^{a}(M,g))\rightarrow H(\Lambda M)\right\} }{a}>0
\]
for some (and hence any) Riemannian metric $g$ on $M$. Next, note
that the fundamental group of $M$ is necessarily finitely generated.
Denote by $\widetilde{\pi}_{1}(M)\cong[S^{1},M]$ the fundamental
group of $M$ modulo conjugacy classes. Given $s\in\pi_{1}(M)$, denote
by $\overline{s}$ the image of $s$ in $\widetilde{\pi}_{1}(M)$.
Given a finite set of generators $S\subseteq\pi_{1}(M)$, let $\gamma_{S}:\mathbb{N}\rightarrow\mathbb{N}$
denote the \textbf{growth function }of $S$, defined by\[
\gamma_{S}(k):=\#\left\{ \alpha\in\widetilde{\pi}_{1}(M)\,:\exists\, s_{1},\dots,s_{k}\in S\cup S^{-1},\ \alpha=\overline{s_{1}s_{2}\dots s_{k}}\right\} .
\]
We define the \textbf{growth rate }$\nu(S)$ of $S$\textbf{ }to be
the number\[
\nu(S):=\lim_{k\rightarrow\infty}\frac{\log\gamma_{S}(k)}{k}\in[0,\infty).
\]
We say that $\widetilde{\pi}_{1}(M)$ as \textbf{exponential growth
}if $\nu(S)>0$ for some (and hence any) finite set of generators
$S$. There are many examples of manifolds $M$ for which $\widetilde{\pi}_{1}(M)$
has exponential growth; for example any surface of genus greater than
one. One can show (see for instance \cite[Lemma 4.15]{McLean2010})
that if $\widetilde{\pi}_{1}(M)$ has exponential growth then $M$
is $\Lambda$-energy hyperbolic. Define \[
n_{\Sigma}(\varphi,T):=\sum_{\alpha\in[S^{1},M]}n_{\Sigma,\alpha}(\varphi,T).
\]
Then we have:\newline\newline\textbf{Corollary C. }\emph{Let $M$
be a closed connected orientable manifold of dimension $n\geq2$.
Assume $\widetilde{\pi}_{1}(M)$ has exponential growth. Let $\Sigma\subseteq T^{*}M$
be a non-degenerate fibrewise starshaped hypersurface. If $\varphi\in\mbox{\emph{Ham}}_{c}(T^{*}M,\mathbb{R})$
is a generic Hamiltonian diffeomorphism then $n_{\Sigma}(\varphi,T)$
grows exponentially with $T$.}\newline

As with Corollary B, we believe this result is also new even in the
case $\Sigma=S_{g}^{*}M$.\newline
\begin{rem}
Whilst in general our results are only valid for a generic Hamiltonian
diffeomorphism $\varphi$, it will be apparent in the proof below
that the case $\varphi\equiv\mathbb{1}$ \textbf{is }included%
\footnote{Indeed, we will consider the general case only after first proving
the special case $\varphi\equiv\mathbb{1}$%
}. Thus as a special case of our results we obtain the following fact:
for a non-degenerate fibrewise starshaped hypersurface $\Sigma\subseteq T^{*}M$,
where $M$ is a $(\Lambda,\alpha)$-energy hyperbolic manifold, the
number of closed Reeb orbits belonging to the free homotopy class
$\alpha$ grows exponentially with time. In fact, this even shows
that the number of \textbf{geometrically distinct }closed Reeb orbits
grows exponentially with time. This result however is not new; it
follows from an observation of Seidel \cite[Section 4a]{Seidel2006}
that the \textbf{growth rate }of \textbf{symplectic homology }is invariant
under \textbf{Liouville isomorphism}. We refer to \cite{Seidel2006}
for a definition of these terms, and for an explanation as to why
this yields a proof of the fact above. We emphasize however that whilst
the case $\varphi\equiv\mathbb{1}$ can be proved much more easily
using symplectic homology, it does not appear possible to attack the
leaf-wise intersection problem with symplectic homology; at the moment
Rabinowitz Floer homology seems to be the most effective way of dealing
with these types of problems.\end{rem}
\begin{acknowledgement*}
We are grateful to Peter Albers, Urs Frauenfelder and Alex Ritter
for several helpful discussions and suggestions. We are also grateful
to Irida Altman for help with constructing Figure \ref{fig:The-function}.
\end{acknowledgement*}

\section{Preliminaries}

\subsection{Sign conventions}

$\ $\vspace{6 pt}

For the convenience of the reader we begin by gathering together the
various sign conventions we use. Let $M$ denote a closed connected
orientable $n$-dimensional manifold. Let $\pi:T^{*}M\rightarrow M$
denote the foot point map.
\begin{itemize}
\item We use the symplectic form $\omega=d\lambda$ on $T^{*}M$, where
$\lambda=pdq$ is the \textbf{Liouville $1$-form}. We will denote
by $Y=p\partial_{p}$ the \textbf{Liouville vector field}, which is
the unique vector field satisfying $i_{Y}\omega=\lambda$.
\item We denote by $\Lambda M$ and $\Lambda T^{*}M$ the \textbf{free loop
spaces }on $M$ and $T^{*}M$ respectively: \[
\Lambda M:=C^{\infty}(S^{1},M),\ \ \ \Lambda T^{*}M:=C^{\infty}(S^{1},T^{*}M).
\]
We denote by $\widetilde{\Lambda}M$ and $\widetilde{\Lambda}T^{*}M$
the completions of these spaces with respect to the Sobolev $W^{1,2}$
norm. Given $\alpha\in[S^{1},M]$, we denote by \[
\Lambda_{\alpha}M:=\left\{ q\in\Lambda M\,:\,[q]=\alpha\right\} ;
\]
\[
\Lambda_{\alpha}T^{*}M:=\left\{ x\in\Lambda T^{*}M\,:\,[\pi\circ x]=\alpha\right\} .
\]

\item An almost complex structure $J$ on $T^{*}M$ is \textbf{compatible}\emph{
}with $\omega$ if $\omega(J\cdot,\cdot)$ defines a Riemannian metric
on $T^{*}M$. We denote by $\mathcal{J}$ the set of time-dependent
almost complex structures $J=(J_{t})_{t\in S^{1}}$ such that each
$J_{t}$ is compatible with $\omega$. 
\item Given $J\in\mathcal{J}$ we denote by $\left\langle \left\langle \cdot,\cdot\right\rangle \right\rangle _{J}$
the $L^{2}$ inner product on $\Lambda T^{*}M\times\mathbb{R}$ defined
by\begin{equation}
\left\langle \left\langle (\xi,b),(\xi',b')\right\rangle \right\rangle _{J}:=\int_{0}^{1}\omega(J\xi,\xi')dt+bb'.\label{eq:metric Jg}
\end{equation}

\item Given a Riemannian metric $g$ on $M$ we denote by $\left\langle \left\langle \cdot,\cdot\right\rangle \right\rangle _{g}$
the $W^{1,2}$ metric on $\widetilde{\Lambda}M\times\mathbb{R}$ defined
by \[
\left\langle \left\langle (\zeta,b),(\zeta',b')\right\rangle \right\rangle _{g}:=\left\langle \zeta(0),\zeta'(0)\right\rangle +\int_{0}^{1}\left\langle \nabla_{t}\zeta,\nabla_{t}\zeta'\right\rangle dt+bb'.
\]

\item In this paper $F$ will always denote an \textbf{autonomous }Hamiltonian
$F:T^{*}M\rightarrow\mathbb{R}$, whereas $H$ will always denote
a \textbf{time-dependent }Hamiltonian $H:S^{1}\times T^{*}M\rightarrow\mathbb{R}$.
\item The symplectic gradient $X_{F}$ of a smooth function $F:T^{*}M\rightarrow\mathbb{R}$
is defined by $i_{X_{F}}\omega=-dF$.
\item Floer homology is defined using \textbf{negative}\emph{ }gradient
flow lines of the Rabinowitz action functional $A_{\mathfrak{f}}$.
\item The notation $H(X,A)$ for the singular homology of a pair $(X,A)$
should always be understood as shorthand for $H(X,A;\mathbb{Z}_{2})$.
\item We denote by $\mathbb{R}^{+}:=\left\{ \eta\in\mathbb{R}\,:\,\eta>0\right\} $.
\item \textbf{All sign conventions in this paper agree with the ones in
\cite{AbbondandoloSchwarz2009}.}
\end{itemize}

\subsection{Preliminaries on fibrewise starshaped hypersurfaces}

$\ $\vspace{6 pt}

We begin by defining our central objects of interest.
\begin{defn}
A submanifold $\Sigma^{2n-1}\subseteq T^{*}M$ is called a \textbf{fibrewise
starshaped hypersurface }if $\Sigma$ is a closed connected separating
hypersurface with the property that the Liouville vector field $Y$
is transverse to $\Sigma$ and points in the outward direction. This
is equivalent to asking that $\lambda_{\Sigma}:=\lambda|_{\Sigma}$
is a positive contact form on $\Sigma$. Given a fibrewise starshaped
hypersurface $\Sigma$, we denote by $R_{\Sigma}$ the \textbf{Reeb
vector field}\emph{ }of the contact $1$-form $\lambda_{\Sigma}$,
that is, the unique vector field on $\Sigma$ defined by the equations
$\lambda_{\Sigma}(R_{\Sigma})=1$ and $i_{R_{\Sigma}}d\lambda_{\Sigma}=0$.
Denote by $D(\Sigma)$ the compact region of $T^{*}M$ bounded by
$\Sigma$, and $D^{\circ}(\Sigma):=\mbox{int}(D(\Sigma))$. 
\end{defn}
Another way to think about such hypersurfaces is the following. Fix
a metric $g$ on $M$, and denote by $S_{g}^{*}M$ the unit cotangent
bundle of $(M,g)$. Then a hypersurface $\Sigma\subseteq T^{*}M$
is fibrewise starshaped if and only if there exists a smooth function
$\sigma:S_{g}^{*}M\rightarrow\mathbb{R}^{+}$ such that \begin{equation}
\Sigma=\mbox{graph}(\sigma)=\left\{ (q,\sigma(q,p))\,:\,(q,p)\in S_{g}^{*}M\right\} .\label{eq:s star m sigma}
\end{equation}

\begin{defn}
Given a fibrewise starshaped hypersurface $\Sigma\subseteq T^{*}M$,
let $\mathcal{D}(\Sigma)\subseteq C^{\infty}(T^{*}M,\mathbb{R})$
denote the set of all autonomous Hamiltonians $F:T^{*}M\rightarrow\mathbb{R}$
such that $F^{-1}(0)=\Sigma$, $X_{F}$ is compactly supported, and
such that $X_{F}|_{\Sigma}=R_{\Sigma}$. We call such Hamiltonians
\textbf{defining Hamiltonians}\emph{ }for $\Sigma$. Let \[
\mathcal{D}:=\bigcup_{\Sigma}\mathcal{D}(\Sigma),
\]
where the union is over all fibrewise starshaped hypersurfaces $\Sigma\subseteq T^{*}M$.
\end{defn}
Given a fibrewise starshaped hypersurface $\Sigma\subseteq T^{*}M$,
denote by $\mathcal{P}(\Sigma)$ the set of \textbf{Reeb orbits}\emph{
}of $R_{\Sigma}$:\[
\mathcal{P}(\Sigma):=\left\{ (x,T)\in\Lambda T^{*}M\times\mathbb{R}^{+}\,:\, x(S^{1})\subseteq\Sigma,\,\dot{x}=TR_{\Sigma}(x)\right\} .
\]
Given $\alpha\in[S^{1},M]$ let \[
\mathcal{P}(\Sigma,\alpha):=\left\{ (x,T)\in\mathcal{P}(\Sigma)\,:\,[\pi\circ x]=\alpha\right\} .
\]
Denote by $\mathcal{A}(\Sigma)$ the \textbf{action spectrum}\emph{
}of $\Sigma$:\[
\mathcal{A}(\Sigma):=\left\{ T\in\mathbb{R}^{+}\,:\,\exists\,(x,T)\in\mathcal{P}(\Sigma)\right\} ;
\]
\[
\mathcal{A}(\Sigma,\alpha):=\left\{ T\in\mathbb{R}^{+}\,:\,\exists\,(x,T)\in\mathcal{P}(\Sigma,\alpha)\right\} ,
\]
and set\[
\ell(\Sigma):=\inf\,\mathcal{A}(\Sigma),\ \ \ \ell(\Sigma,\alpha):=\inf\,\mathcal{A}(\Sigma,\alpha).
\]
Note that $\ell(\Sigma)>0$ for any fibrewise starshaped hypersurface.
\begin{rem}
\label{rem:The-action-spectru}The action spectrum is a closed nowhere
dense subset of $\mathbb{R}$ \cite[Proposition 3.7]{Schwarz2000}.
Moreover it varies {}``lower-semicontinuously'' with respect to
$\Sigma$ in the following sense. Suppose $\Sigma$ is given by the
graph of a smooth function $\sigma:S_{g}^{*}M\rightarrow\mathbb{R}^{+}$,
where $S_{g}^{*}M$ is the unit cotangent bundle of $M$ with respect
to some metric $g$ on $M$ (see \eqref{eq:s star m sigma}). Then
given any neighborhood $\mathcal{V}\subseteq\mathbb{R}$ of $\mathcal{A}(\Sigma)$
there exists a neighborhood $\mathcal{U}\subseteq C^{\infty}(S_{g}^{*}M,\mathbb{R}^{+})$
of $\sigma$ (where the later space is equipped with the strong Whitney
$C^{\infty}$-topology) such that if $\widetilde{\sigma}\in\mathcal{U}$
then the fibrewise starshaped hypersurface $\widetilde{\Sigma}$ defined
as the graph of $\widetilde{\sigma}$ satisfies $\mathcal{A}(\widetilde{\Sigma})\subseteq\mathcal{V}$.
See \cite[Lemma 3.1]{CieliebakGinzburgKerman2004}.
\end{rem}
The non-degeneracy assumption we will make is the following:
\begin{defn}
\label{def:non degen}We say a pair $(x,T)\in\mathcal{P}(\Sigma)$
is \textbf{transversely non-degenerate}\emph{ }if $1$ is not an eigenvalue
of the restriction of $d_{x(0)}\phi_{T}^{R_{\Sigma}}$ to the contact
hyperplane $\ker(\lambda_{\Sigma}(x(0)))\subseteq T_{x(0)}\Sigma$.
We say that $\Sigma$ is \textbf{non-degenerate} if every element
of $\mathcal{P}(\Sigma)$ is transversely non-degenerate. 
\end{defn}
Non-degeneracy is a generic property, in the following sense.
\begin{thm}
\label{thm:nondeg is generic} Fix a metric $g$ on $M$, and let
$S_{g}^{*}M$ denote the unit cotangent bundle of $(M,g)$. The subset
of $C^{\infty}(S_{g}^{*}M,\mathbb{R}^{+})$ consisting of those smooth
functions $\sigma:S_{g}^{*}M\rightarrow\mathbb{R}^{+}$ with the property
that the corresponding fibrewise starshaped hypersurface $\Sigma$
defined by the graph of $\sigma$ (see \eqref{eq:s star m sigma})
is non-degenerate, is residual in $C^{\infty}(S_{g}^{*}M,\mathbb{R}^{+})$.
\end{thm}
See \cite[Proposition 6.1]{HoferWysockiZehnder1998}, \cite[Lemma 2.1]{Bourgeois2009}
or \cite[Appendix B]{CieliebakFrauenfelder2009} for a proof of Theorem
\ref{thm:nondeg is generic}.

\section{$\mathcal{F}$-Rabinowitz Floer homology}

\subsection{The Rabinowitz action functional}

$\ $\vspace{6 pt}

We now define the (variant of the) Rabinowitz action functional that
we will use. Before doing so, we introduce the following convention.
Given an autonomous Hamiltonian $F\in C^{\infty}(T^{*}M,\mathbb{R})$
and a function $\chi\in C^{\infty}(S^{1},[0,1])$, we define $F^{\chi}:S^{1}\times T^{*}M\rightarrow\mathbb{R}$
by \[
F^{\chi}(t,x):=\chi(t)F(x).
\]

\begin{defn}
Fix $F\in C^{\infty}(T^{*}M,\mathbb{R})$, $f\in C^{\infty}(\mathbb{R},\mathbb{R})$
and $\chi\in C^{\infty}(S^{1},[0,\infty))$. The \textbf{Rabinowitz
action functional }associated to the triple $(F,f,\chi)$ is the functional\textbf{
}\[
A_{F^{\chi},f}:\Lambda T^{*}M\times\mathbb{R}\rightarrow\mathbb{R}
\]
defined by \[
A_{F^{\chi},f}(x,\eta):=\int x^{*}\lambda-f(\eta)\int_{0}^{1}F^{\chi}(t,x)dt.
\]
Suppose now $H\in C^{\infty}(S^{1}\times T^{*}M,\mathbb{R})$. The
\textbf{perturbed Rabinowitz action functional }associated to the
quadruple $(F,f,\chi,H)$ is the functional\[
A_{F^{\chi},f}^{H}:\Lambda T^{*}M\times\mathbb{R}\rightarrow\mathbb{R}
\]
 defined by \[
A_{F^{\chi},f}^{H}(x,\eta):=\int x^{*}\lambda-f(\eta)\int_{0}^{1}F^{\chi}(t,x)dt-\int_{0}^{1}H(t,x)dt.
\]
Thus $A_{F^{\chi},f}$ corresponds to the trivial perturbation $H=0$. 
\end{defn}
Although in principle we could use any functions $F,f,\chi,H$ in
the definition above, the definition only becomes interesting when
we restrict the class of functions we consider. Firstly, we will only
ever use functions $F\in\mathcal{D}$; in particular they will always
be constant outside a compact set%
\footnote{At least until Section \ref{sec:The-convex-case}, that is.%
}. Here is the definition of the class of functions $f$ we will study.
\begin{defn}
Let $\mathcal{F}\subseteq C^{\infty}(\mathbb{R},\mathbb{R}^{+})$
denote the set of smooth strictly positive functions $f:\mathbb{R}\rightarrow\mathbb{R}^{+}$
that are strictly increasing, satisfy $\lim_{\eta\rightarrow-\infty}f(\eta)=0$,
and are such that the derivative $f'$ satisfies $0<f'(\eta)\leq1$
for all $\eta\in\mathbb{R}$.\end{defn}
\begin{rem}
The reason for considering functions $f$ of the following form is
to be able to define continuation maps in Rabinowitz Floer homology
for monotone homotopies. This will be explained in Section \ref{sub:Monotone-homotopies-and},
see Remark \ref{The-innocent-looking} in particular. The idea of
perturbing the Rabinowitz action functional with such an auxiliary
function is not new. For instance, in \cite{CieliebakFrauenfelderOancea2010}
a similar idea was used; there however they used functions $f\in C^{\infty}(\mathbb{R},\mathbb{R})$
that were of the form \[
f(\eta)=\begin{cases}
\eta & \left|\eta\right|\leq R-\varepsilon\\
R & \left|\eta\right|\geq R
\end{cases}
\]
for some $R>\varepsilon>0$. They used these (and other more general)
perturbations in order to find the link between Rabinowitz Floer homology
and symplectic homology.
\end{rem}
Next, we will only ever take $\chi$ to lie in a certain subset $\mathcal{X}$
of $C^{\infty}(S^{1},[0,\infty))$. In order to define $\mathcal{X}$,
let us first associate to any element $\chi\in C^{\infty}(S^{1},[0,\infty))$
the function $\bar{\chi}:[0,1]\rightarrow[0,\infty)$ defined by \[
\bar{\chi}(t):=\int_{0}^{t}\chi(\tau)d\tau.
\]
Let $\mathcal{X}\subseteq C(S^{1},[0,\infty))$ denote those functions
$\chi$ whose associated function $\bar{\chi}$ satisfies the following
conditions: 
\begin{enumerate}
\item There exists $t_{0}=t_{0}(\chi)\in(0,1]$ such that $\bar{\chi}(t)\equiv1$
on $[t_{0},1]$;
\item On $[0,t_{0}]$ the function $\bar{\chi}$ is strictly increasing.
\end{enumerate}
Note that the function $\chi\equiv1$ is an element of $\mathcal{X}$.
It will sometimes be useful to restrict to the following subset $\mathcal{X}_{0}\subseteq\mathcal{X}$:
\[
\mathcal{X}_{0}:=\left\{ \chi\in\mathcal{X}\,:\, t_{0}(\chi)<1/2\right\} .
\]

\begin{rem}
\label{nu}Note that if $\chi\in\mathcal{X}$ then there is a unique
function $\nu:[0,1)\rightarrow[0,t_{0})$ such that \[
\bar{\chi}(\nu(t))=t\ \ \ \mbox{for all }t\in[0,1).
\]
One can extend $\nu$ to a continuous function $\nu:[0,1]\rightarrow[0,t_{0}]$
by setting $\nu(1):=t_{0}$. 
\end{rem}
Finally, here is the definition of the class of functions $H$ we
will use.
\begin{defn}
Let $\mathcal{H}$ denote the set of compactly supported time-dependent
Hamiltonians $H\in C_{c}^{\infty}(S^{1}\times T^{*}M,\mathbb{R})$
which have the additional property that $H(t,\cdot)\equiv0$ for $t\in[0,1/2]$. 
\end{defn}
It is easy to see that given any $\varphi\in\mbox{Ham}_{c}(T^{*}M,\omega)$
we can find $H\in\mathcal{H}$ such that $\varphi=\phi_{1}^{H}$ \cite[Lemma 2.3]{AlbersFrauenfelder2010c}.
Note that the function $H\equiv0$ is in $\mathcal{H}$.\newline

In order to ease the notation, let us write \[
\mathfrak{F}:=\mathcal{D}\times\mathcal{F}\times\mathcal{X}\times\mathcal{H},
\]
and refer to elements of $\mathfrak{F}$ by the single letter $\mathfrak{f}$.
Given $\mathfrak{f}=(F,f,\chi,H)\in\mathfrak{F}$, we will often (but
not always) write $A_{\mathfrak{f}}$ as shorthand for the perturbed
Rabinowitz action functional $A_{F^{\chi},f}^{H}$. In fact, most
of the time we will work only with a subset $\mathfrak{F}_{0}\subseteq\mathfrak{F}$.
Let \[
\mathfrak{F}_{0}':=\mathcal{D}\times\mathcal{F}\times\mathcal{X}\times\{0\};
\]
\[
\mathfrak{F}_{0}'':=\mathcal{D}\times\mathcal{F}\times\mathcal{X}_{0}\times\mathcal{H};
\]
\[
\mathfrak{F}_{0}:=\mathfrak{F}_{0}'\cup\mathfrak{F}_{0}''.
\]
In other words, an element $\mathfrak{f}\in\mathfrak{F}$ lies in
$\mathfrak{F}_{0}$ if and only if either $H=0$ or $\chi\in\mathcal{X}_{0}$.\newline

Let $\mathfrak{f}\in\mathfrak{F}$. One readily checks that a pair
$(x,\eta)\in\Lambda T^{*}M\times\mathbb{R}$ is a critical point of
$A_{\mathfrak{f}}$ if and only if\begin{equation}
\begin{cases}
\dot{x}=f(\eta)\chi(t)X_{F}(x)+X_{H}(t,x),\\
f'(\eta)\int_{0}^{1}\chi(t)F(x)dt=0.
\end{cases}\label{eq:first eq}
\end{equation}
Since $f'>0$ everywhere, these equations are equivalent to\begin{equation}
\begin{cases}
\dot{x}=f(\eta)\chi(t)X_{F}(x)+X_{H}(t,x),\\
\int_{0}^{1}\chi(t)F(x)dt=0.
\end{cases}\label{eq:sec eq}
\end{equation}
In particular, if $H=0$ then since $F$ is autonomous, these equations
become:\begin{equation}
\begin{cases}
\dot{x}=f(\eta)\chi(t)R_{\Sigma}(x),\\
x(S^{1})\subseteq\Sigma,
\end{cases}\label{eq:third eq}
\end{equation}
where $F\in\mathcal{D}(\Sigma)$. Given $-\infty\leq a\leq b\leq\infty$,
denote by $\mbox{Crit}^{(a,b)}(A_{\mathfrak{f}})$ the set of critical
points $(x,\eta)\in\Lambda T^{*}M\times\mathbb{R}$ with $A_{\mathfrak{f}}(x,\eta)\in(a,b)$.
Write simply $\mbox{Crit}(A_{\mathfrak{f}})$ instead of $\mbox{Crit}^{(-\infty,\infty)}(A_{\mathfrak{f}})$.
Similarly denote by $\mathcal{A}(A_{\mathfrak{f}}):=A_{\mathfrak{f}}(\mbox{Crit}(A_{\mathfrak{f}}))$
the\emph{ }\textbf{action spectrum}\emph{ }of $A_{\mathfrak{f}}$.
Given $\alpha\in[S^{1},M]$, let $\mbox{Crit}^{(a,b)}(A_{\mathfrak{f}},\alpha):=\mbox{Crit}^{(a,b)}(A_{\mathfrak{f}})\cap(\Lambda_{\alpha}T^{*}M\times\mathbb{R})$
and $\mathcal{A}(A_{\mathfrak{f}},\alpha):=A_{\mathfrak{f}}(\mbox{Crit}(A_{\mathfrak{f}},\alpha))$.\newline

Given $\varphi\in\mbox{Ham}_{c}(T^{*}M,\omega)$ and a fibrewise starshaped
hypersurface $\Sigma$, let \[
\mbox{LW}_{+}(\Sigma,\varphi):=\left\{ p\in\Sigma\,:\, p\mbox{ is a positive leaf-wise intersection point for }\varphi\right\} .
\]
The following lemma explains the advantage of choosing $\mathfrak{f}\in\mathfrak{F}_{0}$. 
\begin{lem}
\label{lem:key lemma}\cite{CieliebakFrauenfelder2009,AlbersFrauenfelder2008}
\begin{enumerate}
\item Suppose $\mathfrak{f}=(F,f,\chi,0)\in\mathfrak{F}_{0}'$, with $F\in\mathcal{D}(\Sigma)$.
Let $\nu:[0,1]\rightarrow[0,t_{0}]$ denote the function defined in
Remark \ref{nu}. Then $(x,\eta)\in\mbox{\emph{Crit}}(A_{\mathfrak{f}})$
if and only if $(x\circ\nu,f(\eta))\in\mathcal{P}(\Sigma)$. Moreover
in this case \[
A_{\mathfrak{f}}(x,\eta)=f(\eta)>0.
\]

\item Now suppose $\mathfrak{f}=(F,f,\chi,H)\in\mathfrak{F}_{0}''$ with
$F\in\mathcal{D}(\Sigma)$. Let $\varphi:=\phi_{1}^{H}$. Then there
is a surjective map \[
e_{\mathfrak{f}}:\mbox{\emph{Crit}}(A_{\mathfrak{f}})\rightarrow\mbox{\emph{LW}}_{+}(\Sigma,\varphi)
\]
given by \[
e_{\mathfrak{f}}(x,\eta):=x(0).
\]
If the leaf $\mathcal{L}_{x(0)}$ is not closed then $x(0)$ has time-shift
$f(\eta)$. If there are no periodic leaf-wise intersection points
then $e_{\mathfrak{f}}$ is injective. Moreover if $(x,\eta)\in\mbox{\emph{Crit}}(A_{\mathfrak{f}})$
then: \begin{equation}
A_{\mathfrak{f}}(x,\eta)=f(\eta)-\int_{0}^{1}\{\lambda(X_{H}(t,x))-H(t,x)\}dt.\label{eq:value of perturbed functional at crit}
\end{equation}

\end{enumerate}
\end{lem}
Let $\mathfrak{f}=(F,f,\chi,H)$ be as in part (2) of the previous
lemma. As stated in the Introduction, we want to be able to associate
to a leaf-wise intersection point $p\in\mbox{LW}_{+}(\Sigma,\varphi)$
a free homotopy class $\alpha\in[S^{1},M]$. It is natural to define
\[
\mbox{LW}_{+}(\Sigma,\varphi,\alpha):=e_{\mathfrak{f}}(\mbox{Crit}(A_{\mathfrak{f}},\alpha)).
\]

The following lemma, based on a well known argument (see for example
\cite[Proposition 3.1]{Schwarz2000}) implies that $\mbox{LW}_{+}(\Sigma,\varphi,\alpha)$
is well defined. 
\begin{lem}
\label{lem:Independence lemma}Suppose $\Sigma$ is a fibrewise starshaped
hypersurface and $\varphi\in\mbox{\emph{Ham}}_{c}(T^{*}M,\omega)$.
Suppose $H_{0},H_{1}\in\mathcal{H}$ both generate $\varphi$. Let
$F\in\mathcal{D}(\Sigma)$, $f\in\mathcal{F}$ and $\chi\in\mathcal{X}_{0}$.
Set $\mathfrak{f}_{i}:=(F,f,\chi,H_{i})\in\mathfrak{F}_{0}''$ for
$i=0,1$. Fix $p\in\mbox{\emph{LW}}_{+}(\Sigma,\varphi)$ and $\alpha\in[S^{1},M]$.
Then there exists $(x_{0},\eta_{0})\in\mbox{\emph{Crit}}(A_{\mathfrak{f}_{0}},\alpha)$
such that $e_{\mathfrak{f}_{0}}(x_{0},\eta_{0})=p$ if and only if
there exists $(x_{1},\eta_{1})\in\mbox{\emph{Crit}}(A_{\mathfrak{f}_{1}},\alpha)$
such that $e_{\mathfrak{f}_{1}}(x_{1},\eta_{1})=p$. \end{lem}
\begin{proof}
Suppose $p\in\mbox{LW}_{+}(\Sigma,\varphi)$. Thus there exists $\eta\in\mathbb{R}$
such that $\varphi(\phi_{f(\eta)}^{F}(p))=(p)$. Set $K_{i}:=H_{i}+f(\eta)F^{\chi}$
for $i=0,1$. Then $p$ is a fixed point of $\phi_{1}^{K_{0}}$ and
$\phi_{1}^{K_{1}}$, and if $x_{i}(t):=\phi_{t}^{K_{i}}(p)$ then
$(x_{i},\eta)\in\mbox{Crit}(A_{\mathfrak{f}_{i}})$. Note that by
construction $K_{0}(1,\cdot)\equiv0\equiv K_{1}(1,\cdot)$. Thus we
may define a loop $(\varphi_{t})_{t\in S^{1}}\subseteq\mbox{Ham}_{c}(T^{*}M,\omega)$
by \[
\varphi_{t}:=\begin{cases}
\phi_{2t}^{K_{0}}, & 0\leq t\leq1/2,\\
\phi_{1-2t}^{K_{1}}, & 1/2\leq t\leq1.
\end{cases}
\]
The flow $\varphi_{t}$ is the flow associated to the Hamiltonian
$G\in C_{c}^{\infty}(S^{1}\times T^{*}M,\mathbb{R})$ defined by \[
G(t,\cdot):=\begin{cases}
K_{0}(2t,\cdot), & 0\leq t\leq1/2,\\
-K_{1}(1-2t,\cdot), & 1/2\leq t\leq1.
\end{cases}
\]
Now consider the map $e_{\varphi}:T^{*}M\rightarrow\Lambda T^{*}M$
which sends a point in $T^{*}M$ to its orbit under $(\varphi_{t})$.
Then $\mbox{im}(e_{\varphi})$ is contained in a connected component
of $\Lambda T^{*}M$ (as $M$ is connected). But from the proof of
the Arnold conjecture for cotangent bundles we know that for any 1-periodic
compactly supported Hamiltonian function there exists at least one
contractible 1-periodic solution of the associated Hamiltonian system.
Thus $\mbox{im}(e_{\varphi})\cap\Lambda_{0}T^{*}M\ne\emptyset$, and
hence every loop in the image of $e_{\varphi}$ is contractible; in
particular the loop $e_{\varphi}(p)$ is contractible. But $e_{\varphi}(p)$
is a reparametrization of the loop $x_{0}*x_{1}^{-1}$. Thus necessarily
$x_{0}$ and $x_{1}$ belong to the same component $\Lambda_{\alpha}T^{*}M$
of $\Lambda T^{*}M$.
\end{proof}
Next, we quote the following result due to Albers and Frauenfelder.
\begin{prop}
\label{prop:-generically no plwip}\cite[Theorem 3.3]{AlbersFrauenfelder2008}
Suppose $\dim\, M\geq2$. Then if $\Sigma$ is a non-degenerate fibrewise
starshaped hypersurface then there exists a generic set $\mathcal{G}(\Sigma)\subseteq\mbox{\emph{Ham}}_{c}(T^{*}M,\omega)$
such that if $\varphi\in\mathcal{G}(\Sigma)$ then there are no periodic
leaf-wise intersection points: \[
\mbox{\emph{LW}}_{+}(\Sigma,\varphi)\cap\left\{ x(t)\,:\,(x,T)\in\mathcal{P}(\Sigma),\, t\in S^{1}\right\} =\emptyset.
\]

\end{prop}
It will be important to be able to control the size of $\left|A_{\mathfrak{f}}(x,\eta)\right|$
in terms of the size of $\left|\eta\right|$ and vice versa for $(x,\eta)\in\mbox{Crit}(A_{\mathfrak{f}})$.
This leads to the following definition.
\begin{defn}
\label{Define-a-semi-norm}Define a semi-norm $\kappa:C_{c}^{\infty}(S^{1}\times T^{*}M,\mathbb{R})\rightarrow[0,\infty)$
by \[
\kappa(H):=\sup_{(t,x)\in S^{1}\times\Lambda T^{*}M}\left|\int_{0}^{1}\lambda(X_{H}(t,x))-H(t,x)dt\right|.
\]
Note that \[
\kappa(H)=\sup\left\{ \left|\eta\right|\,:\,\eta\in\mathcal{A}(A^{H})\right\} ,
\]
where $A^{H}$ is the standard action functional \eqref{eq:standard action functional}.
As remarked in the introduction, since $\mathcal{A}(A^{H})$ depends
only on the element $\phi_{1}^{H}\in\mbox{Ham}_{c}(T^{*}M,\omega)$,
we may regard $\kappa$ as being defined on $\mbox{Ham}_{c}(T^{*}M,\omega)$.
Given $a\geq0$ let $\mathcal{H}(a)\subseteq\mathcal{H}$ denote the
subset of elements $H\in\mathcal{H}$ with $\kappa(H)\leq a$. 
\end{defn}
The following lemma is immediate from \eqref{eq:value of perturbed functional at crit}.
\begin{lem}
\label{lem:action bounds lemma}Suppose $\mathfrak{f}=(F,f,\chi,H)\in\mathfrak{F}_{0}$
with $H\in\mathcal{H}(c)$ for some $c>0$. Then if $(x,\eta)\in\mbox{\emph{Crit}}(A_{\mathfrak{f}})$
and $-\infty<a<b<\infty$, \[
\eta\in(a,b)\ \ \ \Rightarrow\ \ \ f(a)-c<A_{\mathfrak{f}}(x,\eta)<f(b)+c.
\]
Now suppose that $a-c>0$. Then\[
A_{\mathfrak{f}}(x,\eta)\in(a,b)\ \ \ \Rightarrow\ \ \ f^{-1}(a-c)<\eta<f^{-1}(b+c).
\]
\end{lem}
\begin{cor}
\label{cor:cptness of critical set}Fix $0<c<a<b<\infty$. Suppose
$\mathfrak{f}=(F,f,\chi,H)\in\mathfrak{F}_{0}$ with $H\in\mathcal{H}(c)$.
Then the set $\mbox{\emph{Crit}}^{(a,b)}(A_{\mathfrak{f}})$ is compact.\end{cor}
\begin{proof}
Arguing similarly to Lemma \ref{lem:action bounds lemma}, we see
that if $(x,\eta)\in\mbox{Crit}^{(a,b)}(A_{\mathfrak{f}})$ then $\eta\in(f^{-1}(a-c),f^{-1}(b+c))$.
In particular, $\left|\eta\right|$ is bounded. Since $F$ and $H$
are compactly supported and $0$ is a regular value of $F$, there
exists a compact set $V\subseteq T^{*}M$ such that $x(S^{1})\subseteq V$
for all $(x,\eta)\in\mbox{Crit}(A_{\mathfrak{f}})$. Since $\left|\eta\right|$
is bounded, the Arzela-Ascoli theorem together with the first equation
in \eqref{eq:sec eq} then imply that $\mbox{Crit}^{(a,b)}(A_{\mathfrak{f}})$
is precompact, and hence compact.
\end{proof}
In fact, it will be most convenient to actually require $f(\eta)=\eta$
in the action interval we work with.
\begin{defn}
Given $a>0$ denote by $\mathcal{F}(a)\subseteq\mathcal{F}$ the subset
of functions $f\in\mathcal{F}$ that satisfy $f(\eta)=\eta$ for all
$\eta\in[a,\infty)$.
\end{defn}
We next address the non-degeneracy issue.
\begin{defn}
\label{def:non degen for H}An element $\mathfrak{f}\in\mathfrak{F}_{0}'$
is called \textbf{regular }if $A_{\mathfrak{f}}$ is a \textbf{Morse-Bott
}function, and $\mbox{Crit}(A_{\mathfrak{f}})$ is a discrete union
of circles. If $\mathfrak{f}=(F,f,\chi,0)$ with $F\in\mathcal{D}(\Sigma)$
then $\mathfrak{f}$ is regular if and only if $\Sigma$ is non-degenerate
in the sense of Definition \ref{def:non degen}. In particular, a
generic element of $\mathfrak{F}_{0}'$ is regular (cf. Theorem \ref{thm:nondeg is generic}).
An element $\mathfrak{f}\in\mathfrak{F}_{0}''$ is called \textbf{regular
}if $A_{\mathfrak{f}}$ is a \textbf{Morse }function. Given a fibrewise
starshaped hypersurface $\Sigma$, there is a residual subset $\mathcal{R}(\Sigma)\subseteq\mathcal{H}$
such that if $F\in\mathcal{D}(\Sigma)$ and $H\in\mathcal{R}(\Sigma)$
then for any $f\in\mathcal{F}$ and $\chi\in\mathcal{X}_{0}$ the
quadruple $(F,f,\chi,H)$ is regular. See \cite[Proposition 3.9]{AlbersFrauenfelder2010}.
We denote by \[
\mathfrak{F}_{0,\textrm{reg}}=\mathfrak{F}_{0,\textrm{reg}}'\cup\mathfrak{F}_{0,\textrm{reg}}''
\]
the set of regular elements of $\mathfrak{F}_{0}$.
\end{defn}
Given $J\in\mathcal{J}$ we denote by $\nabla_{J}A_{\mathfrak{f}}$
the gradient of $A_{\mathfrak{f}}$ with respect to the inner product
$\left\langle \left\langle \cdot,\cdot\right\rangle \right\rangle _{J}$
(see \eqref{eq:metric Jg}). A quick computation tells us\[
\nabla_{J}A_{\mathfrak{f}}(x,\eta)=\left(J_{t}(x)(\dot{x}-f(\eta)\chi(t)X_{F}(x)-X_{H}(t,x)),-f'(\eta)\int_{0}^{1}\chi(t)F(x)dt\right).
\]

\begin{defn}
A \textbf{gradient flow line}\emph{ }of $A_{\mathfrak{f}}$ (with
respect to $J\in\mathcal{J}$) is a map $u:\mathbb{R}\rightarrow\Lambda T^{*}M\times\mathbb{R}$
such that \begin{equation}
\partial_{s}u+\nabla_{J}A_{\mathfrak{f}}(u)=0.\label{eq:rf eq}
\end{equation}
In components $u=(x,\eta)$ this reads:\[
\partial_{s}x+J_{t}(x)(\partial_{t}x-f(\eta)\chi(t)X_{F}(x)-X_{H}(t,x))=0;
\]
\[
\partial_{s}\eta-f'(\eta)\int_{0}^{1}\chi(t)F(x)dt=0.
\]
Given $0<a<b<\infty$, denote by $\mathcal{M}^{(a,b)}(\nabla_{J}A_{\mathfrak{f}})$
the set of gradient flow lines $u:\mathbb{R}\rightarrow\Lambda T^{*}M\times\mathbb{R}$
of $A_{\mathfrak{f}}$ that satisfy $a<A_{\mathfrak{f}}(u(s))<b$
for all $s\in\mathbb{R}$. Given $\alpha\in[S^{1},M]$, let $\mathcal{M}^{(a,b)}(\nabla_{J}A_{\mathfrak{f}},\alpha)$
denote the subset of $\mathcal{M}^{(a,b)}(\nabla_{J}A_{\mathfrak{f}})$
consisting of those flow lines $u=(x,\eta)$ that satisfy $[\pi\circ x(s,\cdot)]=\alpha$
for all $s\in\mathbb{R}$. 
\end{defn}
Fix $\mathfrak{f}\in\mathfrak{F}_{0,\textrm{reg}}$. It is well known
that the non-degeneracy assumption that $A_{\mathfrak{f}}$ is Morse(-Bott)
implies that every element $u\in\mathcal{M}^{(a,b)}(\nabla_{J}A_{\mathfrak{f}})$
is asymptotically convergent at each end to elements of $\mbox{Crit}^{(a,b)}(A_{\mathfrak{f}})$.
That is, the limits \[
\lim_{s\rightarrow\pm\infty}u(s,t)=:(x_{\pm}(t),\eta_{\pm}),\ \ \ \lim_{s\rightarrow\infty}\partial_{t}u(s,t)=0,
\]
exist, and the convergence is uniform in $t$, and the limits $(x_{\pm},\eta_{\pm})$
belong to $\mbox{Crit}^{(a,b)}(A_{\mathfrak{f}})$ (see for instance
\cite{Salamon1999}). Moreover, if $E(u)$ denotes the \textbf{energy}\emph{
}of a gradient flow line:\[
E(u):=\int_{-\infty}^{\infty}\left\Vert \partial_{s}u(s)\right\Vert _{J}^{2}ds,
\]
then if $u\in\mathcal{M}^{(a,b)}(\nabla_{J}A_{\mathfrak{f}})$ is
asymptotically convergent to $(x_{\pm},\eta_{\pm})\in\mbox{Crit}^{(a,b)}(A_{\mathfrak{f}})$
it holds that \[
A_{\mathfrak{f}}(x_{-},\eta_{-})-A_{\mathfrak{f}}(x_{+},\eta_{+})=E(u)>0.
\]
Given $u\in\mathcal{M}^{(a,b)}(\nabla_{J}A_{\mathfrak{f}})$, the
linearization of the gradient flow equation gives rise to a Fredholm
operator $D_{u}$. There exists a residual subset $\mathcal{J}_{\textrm{reg}}(\mathfrak{f})$
such that if $J\in\mathcal{J}_{\textrm{reg}}(\mathfrak{f})$ then
for every $0<a<b<\infty$ and every $u\in\mathcal{M}^{(a,b)}(\nabla_{J}A_{\mathfrak{f}})$
the operator $D_{u}$ is surjective.
\begin{defn}
\label{def:acs def}Suppose $S$ is a fibrewise starshaped hypersurface.
An $\omega$-compatible almost complex structure $J$ on $T^{*}M$
is called \textbf{convex }on $T^{*}M\backslash D^{\circ}(S)$ if the
following three conditions hold:

\[
J(\xi_{S})=\xi_{S},\ \ \ \omega(J(p)Y(p),Y(p))=1,\ \ \ d_{p}\phi_{t}^{Y}\circ J(p)=J(p)\circ d_{p}\phi_{t}^{Y}\ \ \ \mbox{for all }p\in S.
\]

Here $\phi_{t}^{Y}$ is the semi-flow of $Y$ on $T^{*}M\backslash D^{\circ}(S)$.
Denote by $\mathcal{J}(S)\subseteq\mathcal{J}$ the set of all time
dependent almost complex structures $J=(J_{t})_{t\in S^{1}}$ such
that each $J_{t}$ is convex and independent of $t$ on $T^{*}M\backslash D^{\circ}(S)$. 
\end{defn}
Our motivation for studying such almost complex structures is the
following lemma, which is based on a well known argument using the
maximum principle.
\begin{lem}
\label{lem:why convex good}Suppose $\Sigma,S$ are fibrewise starshaped
hypersurfaces with $D(\Sigma)\subseteq D^{\circ}(S)$. Suppose $\mathfrak{f}=(F,f,H,\chi)\in\mathfrak{F}_{0,\textrm{\emph{reg}}}$,
where $F\in\mathcal{D}(\Sigma)$ is such that $\mbox{\emph{supp}}(X_{F})\subseteq D^{\circ}(S)$.
Fix $J\in\mathcal{J}(S)$. Then for any $0<a<b<\infty$ and any $u\in\mathcal{M}^{(a,b)}(\nabla_{J}A_{\mathfrak{f}})$
we have $\mbox{\emph{im}}(u)\subseteq D(S)$. 
\end{lem}

\subsection{\label{sub:Floer-homology-of}Floer homology of the Rabinowitz action
functional}

$\ $\vspace{6 pt}

We now define the Rabinowitz Floer chain complex associated to the
action functional $A_{\mathfrak{f}}$ for $\mathfrak{f}\in\mathfrak{F}_{0,\textrm{reg}}$.
The construction is slightly different depending as to whether $\mathfrak{f}\in\mathfrak{F}_{0,\textrm{reg}}'$
or $\mathfrak{f}\in\mathfrak{F}_{0,\textrm{reg}}''$. We begin with
the latter case, since this is somewhat easier. 

Fix $\mathfrak{f}=(F,f,\chi,H)\in\mathfrak{F}_{0,\textrm{reg}}''$,
$\alpha\in[S^{1},M]$ and $0<a<b<\infty$. Suppose $v:=(x,\eta)$
is a critical point of $A_{\mathfrak{f}}$. Then $x$ is a 1-periodic
orbit of the time-dependent Hamiltonian $G:=f(\eta)F^{\chi}+H$. Since
$A_{\mathfrak{f}}$ is Morse, $x$ is a non-degenerate orbit, and
hence the \textbf{Conley-Zehnder index }$\mu_{\textrm{CZ}}(x;G)$
of $x$ as an orbit of $G$ is a well defined integer. See for instance
\cite{SalamonZehnder1992} or \cite{AbbondandoloSchwarz2006} (the
latter in particular for non-contractible loops) for the definition
of the Conley-Zehnder index, although note that our sign conventions
match \cite{AbbondandoloPortaluriSchwarz2008} not \cite{SalamonZehnder1992}
or \cite{AbbondandoloSchwarz2006}. We define $\mu(v):=\mu_{\textrm{CZ}}(x;G)$.
Let $\mbox{Crit}_{k}(A_{\mathfrak{f}})$ denote those critical points
$v$ with index $\mu(v)=k$. Denote by $CF_{k}^{(a,b)}(A_{\mathfrak{f}},\alpha)$
the $\mathbb{Z}_{2}$-vector space\[
CF_{k}^{(a,b)}(A_{\mathfrak{f}},\alpha):=\mbox{Crit}_{k}^{(a,b)}(A_{\mathfrak{f}},\alpha)\otimes\mathbb{Z}_{2}.
\]
Choose $J\in\mathcal{J}_{\textrm{reg}}(\mathfrak{f})$. Given $v_{\pm}\in\mbox{Crit}^{(a,b)}(A_{\mathfrak{f}},\alpha)$
denote by $\mathcal{M}(v_{-},v_{+})$ the moduli space of maps $u\in\mathcal{M}^{(a,b)}(\nabla_{J}A_{\mathfrak{f}},\alpha)$
that are asymptotically convergent to $v_{\pm},$ divided out by the
translation $\mathbb{R}$-action. Then $\mathcal{M}(v_{-},v_{+})$
carries the structure of a smooth manifold of dimension $\mu(v_{-})-\mu(v_{+})-1$.
Under certain conditions (see Theorem \ref{thm:main theorem of rfh}
below) the manifolds $\mathcal{M}(v_{-},v_{+})$ are compact up to
breaking. Assuming this is the case, the boundary operator $\partial$
on $CF^{(a,b)}(A_{\mathfrak{f}},\alpha)$ is defined via:\[
\partial v:=\sum_{w\in\textrm{Crit}^{(a,b)}(A_{\mathfrak{f}},\alpha)}\#_{2}\mathcal{M}_{0}(v,w)w,
\]
where $\mathcal{M}_{0}(v,w)$ denotes the possibly empty zero-dimensional
component of $\mathcal{M}(v,w)$, and $\#_{2}$ denotes the cardinality
taken modulo $2$. It turns out that $\partial$ has degree $-1$
with respect to the grading $\mu$. We denote by $HF^{(a,b)}(A_{\mathfrak{f}},\alpha)$
the resulting homology, which is independent of the choice of almost
complex structure $J\in\mathcal{J}_{\textrm{reg}}(\mathfrak{f})$
we chose. 

Now let us consider the case $\mathfrak{f}=(F,f,\chi,0)\in\mathfrak{F}_{0,\textrm{reg}}'$.
Suppose $v:=(x,\eta)$ is a critical point of $A_{\mathfrak{f}}$.
Then $x$ is a 1-periodic orbit of the time-dependent Hamiltonian
$G:=f(\eta)F^{\chi}$. Since $A_{\mathfrak{f}}$ is Morse-Bott, $x$
is a transversely non-degenerate orbit, and hence the \textbf{transverse
Conley-Zehnder index }$\mu_{\textrm{CZ}}^{\tau}(x;G)$ of $x$ as
an orbit of $X_{G}$ is a well defined integer (see for instance \cite[Section 3]{AbbondandoloSchwarz2009}
for the definition of the transverse Conley-Zehnder index). 

Pick a Morse function $h:\mbox{Crit}(A_{\mathfrak{f}})\rightarrow\mathbb{R}$,
and denote by $\mbox{Crit}(h)\subseteq\mbox{Crit}(A_{\mathfrak{f}})$
the set of critical points of $h$. Define an augmented grading $\mu:\mbox{Crit}(h)\rightarrow\mathbb{Z}$
by \[
\mu(v):=\mu_{\textrm{CZ}}^{\tau}(x;G)+i_{h}(v),\ \ \ v=(x,\eta),
\]
where $i_{h}(v)\in\{0,1\}$ is the Morse index of $v$. Let $\mbox{Crit}_{k}^{(a,b)}(h,\alpha):=\{v\in\mbox{Crit}(h)\cap\mbox{Crit}^{(a,b)}(A_{\mathfrak{f}},\alpha)\,:\,\mu(v)=k\}$.
Given $k\in\mathbb{Z}$, define \[
CF_{k}^{(a,b)}(A_{\mathfrak{f}},\alpha):=\mbox{Crit}_{k}^{(a,b)}(h,\alpha)\otimes\mathbb{Z}_{2}.
\]
One now defines the boundary operator in much the same way as before,
only this time one must take $\mathcal{M}(v_{-},v_{+})$ to be the
moduli space of \textbf{gradient flow lines with cascades} of $h$.
We refer the reader to \cite[Appendix A]{Frauenfelder2004} for more
information. We emphasize once again that in order to be able to define
the Floer homology we need the manifolds $\mathcal{M}(v_{-},v_{+})$
to be compact up to breaking, which is \textbf{not} always the case.

\subsection{Admissible quadruples}

$\ $\vspace{6 pt}
\begin{defn}
\label{Fix-.-Ad}Fix $\alpha\in[S^{1},M]$. A quadruple $\mathfrak{q}=(\mathfrak{f},a,b,J)$
consisting of $\mathfrak{f}\in\mathfrak{F}_{0,\textrm{reg}}$, $J\in\mathcal{J_{\textrm{reg}}}(\mathfrak{f})$
and $0<a<b<\infty$ is called \textbf{$\alpha$-admissible }if the
following conditions are satisfied:
\begin{enumerate}
\item $\mathcal{A}(A_{\mathfrak{f}},\alpha)\cap\{a,b\}=\emptyset$;
\item The set $\mbox{Crit}^{(a,b)}(A_{\mathfrak{f}},\alpha)$ is compact;
\item There exist constants $C_{\textrm{loop}},C_{\textrm{mult}}>0$ such
that for all $u=(x,\eta)\in\mathcal{M}^{(a,b)}(\nabla_{J}A_{\mathfrak{f}},\alpha)$
it holds that $\left\Vert x\right\Vert _{L^{\infty}}<C_{\textrm{loop}}$
and $\left\Vert \eta\right\Vert _{L^{\infty}}<C_{\textrm{mult}}$.
\end{enumerate}
A quadruple $\mathfrak{q}$ is simply called \textbf{admissible }if
it is $\alpha$-admissible for all $\alpha\in[S^{1},M]$.
\end{defn}
The next result follows by standard arguments in Floer homology, see
for instance \cite{Salamon1990}.
\begin{thm}
\label{thm:main theorem of rfh}Fix $\alpha\in[S^{1},M]$. If $\mathfrak{q}=(\mathfrak{f},a,b,J)$
is an $\alpha$-admissible quadruple, then the Floer homology $HF^{(a,b)}(A_{\mathfrak{f}},\alpha)$
is well defined (that is, the manifolds $\mathcal{M}(v_{-},v_{+})$
are compact up to breaking, see above). 
\end{thm}
We will now find conditions under which a quadruple $\mathfrak{q}=(\mathfrak{f},a,b,J)$
is admissible. The first step is the following two preliminary lemmas,
which are minor modifications of the argument of \cite[Lemma 2.11]{AlbersFrauenfelder2010c}.
\begin{lem}
\label{lem:1}Suppose $\mathfrak{f}=(F,f,\chi,H)\in\mathfrak{F}_{0}$
and $J\in\mathcal{J}$. There exist constants $k,T>0$ depending only
on $F$ such that if $x\in\Lambda T^{*}M$ satisfies \[
x(\mbox{\emph{supp}}(\chi))\subseteq U_{k}(F):=F^{-1}(-k,k)
\]
then it holds that \[
\frac{2}{3}\left(A_{\mathfrak{f}}(x,\eta)-T\left\Vert \nabla_{J}A_{\mathfrak{f}}(x,\eta)\right\Vert _{J}-\kappa(H)\right)\leq f(\eta)\leq2\left(A_{\mathfrak{f}}(x,\eta)+T\left\Vert \nabla_{J}A_{\mathfrak{f}}(x,\eta)\right\Vert _{J}+\kappa(H)\right).
\]
\end{lem}
\begin{proof}
In this proof and the next we denote by $\left\Vert \cdot\right\Vert _{2}$
the norm\[
\left\Vert \xi\right\Vert _{2}:=\int_{0}^{1}\omega(J\xi,\xi)dt,
\]
so that \[
\left\Vert (\xi,b)\right\Vert _{J}=\sqrt{\left\Vert \xi\right\Vert _{2}^{2}+b^{2}}.
\]
There exists $k>0$ such that\[
\frac{1}{2}+k\leq\lambda(X_{F}(p))\leq\frac{3}{2}-k\ \ \ \mbox{for all }p\in U_{k}(F).
\]
Set \[
T=T(F):=\left\Vert \lambda|_{U_{k}(F)}\right\Vert _{\infty}.
\]
For any $(x,\eta)\in\Lambda T^{*}M\times\mathbb{R}$ with $x(\mbox{supp}(\chi))\subseteq U_{k}(F)$,
we have\begin{align*}
A_{\mathfrak{f}}(x,\eta) & =\int_{0}^{1}\lambda(\dot{x})dt-f(\eta)\int_{0}^{1}F^{\chi}(t,x)dt-\int_{0}^{1}H(t,x)dt\\
 & =f(\eta)\int_{0}^{1}\lambda(\chi(t)X_{F}(x))dt+\int_{0}^{1}\lambda(\dot{x}-f(\eta)\chi(t)X_{F}(x))dt\\
 & -f(\eta)\int_{0}^{1}\chi(t)F(x)dt-\int_{0}^{1}H(t,x)dt\\
 & \geq f(\eta)\int_{0}^{1}\chi(t)\lambda(X_{F}(x))dt-\left|\int_{0}^{1}\lambda(\dot{x}-f(\eta)\chi(t)X_{F}(x)-X_{H}(t,x))dt\right|\\
 & -f(\eta)\int_{0}^{1}\chi(t)F(x)dt-\kappa(H)\\
 & \geq\left(\frac{1}{2}+k\right)f(\eta)-T\left\Vert \dot{x}-f(\eta)\chi(t)X_{F}(x)-X_{H}(t,x)\right\Vert _{2}-f(\eta)k-\kappa(H)\\
 & \geq\frac{1}{2}f(\eta)-T\left\Vert \dot{x}-f(\eta)\chi(t)X_{F}(x)-X_{H}(t,x)\right\Vert _{2}-\kappa(H)\\
 & \geq\frac{1}{2}f(\eta)-T\left\Vert \nabla_{J}A_{\mathfrak{f}}(x,\eta)\right\Vert _{J}-\kappa(H),
\end{align*}
and similarly \begin{align*}
A_{\mathfrak{f}}(x,\eta) & =\int_{0}^{1}\lambda(\dot{x})dt-f(\eta)\int_{0}^{1}F^{\chi}(t,x)dt-\int_{0}^{1}H(t,x)dt\\
 & =f(\eta)\int_{0}^{1}\lambda(\chi(t)X_{F}(x))dt+\int_{0}^{1}\lambda(\dot{x}-f(\eta)\chi(t)X_{F}(t,x))dt\\
 & -f(\eta)\int_{0}^{1}\chi(t)F(x)dt-\int_{0}^{1}H(t,x)dt\\
 & \leq f(\eta)\int_{0}^{1}\chi(t)\lambda(X_{F}(x))dt+\left|\int_{0}^{1}\lambda(\dot{x}-f(\eta)\chi(t)X_{F}(x)-X_{H}(t,x))dt\right|\\
 & -f(\eta)\int_{0}^{1}\chi(t)F(x)dt+\kappa(H)\\
 & \leq\left(\frac{3}{2}-k\right)f(\eta)+T\left\Vert \dot{x}-f(\eta)\chi(t)X_{F}(x)-X_{H}(t,x)\right\Vert _{2}+f(\eta)k+\kappa(H)\\
 & \leq\frac{3}{2}f(\eta)+T\left\Vert \nabla_{J}A_{\mathfrak{f}}(x,\eta)\right\Vert _{J}+\kappa(H).
\end{align*}
\end{proof}
\begin{lem}
\label{lem:lemma 2} Suppose $\mathfrak{f}=(F,f,\chi,H)\in\mathfrak{F}_{0}$
and $J\in\mathcal{J}$. For every $k>0$ there exists $\rho=\rho(k,F)>0$
such that if $(x,\eta)\in\Lambda T^{*}M\times\mathbb{R}$ satisfies:
\[
\left\Vert \nabla_{J}A_{\mathfrak{f}}(x,\eta)\right\Vert _{J}\leq\rho f'(\eta),
\]
then $x(\mbox{\emph{supp}}(\chi))\subseteq U_{k}(F)$.\end{lem}
\begin{proof}
To begin with, arguing exactly as in \cite[Lemma 2.11, Claim 2]{AlbersFrauenfelder2010c}
(which only uses the loop component of the $\nabla_{J}A_{\mathfrak{f}}(x,\eta)$),
one sees that if $x(\mbox{supp}(\chi))\cap(T^{*}M\backslash U_{k}(F))\ne\emptyset$
and\emph{ }$x(\mbox{supp}(\chi))\cap U_{k/2}(F)\ne\emptyset$ then
\[
\left\Vert \nabla_{J}A_{\mathfrak{f}}(x,\eta)\right\Vert _{J}\geq\frac{k}{2\left\Vert \nabla F\right\Vert _{\infty}}.
\]
Next, if $x(\mbox{supp}(\chi))\subseteq T^{*}M\backslash U_{k/2}(F)$
then looking at the second component of the gradient equation, \[
\left\Vert \nabla_{J}A_{\mathfrak{f}}(x,\eta)\right\Vert _{J}\geq\left|f'(\eta)\int_{0}^{1}\chi(t)F(x)dt\right|\geq f'(\eta)\frac{k}{2}.
\]
 Thus if \[
\rho:=\rho(k,F):=\min\left\{ \frac{k}{2},\frac{k}{2\left\Vert \nabla F\right\Vert _{L^{\infty}}}\right\} ,
\]
then using the fact that $f'(\eta)\leq1$ for all $\eta\in\mathbb{R}$
as $f\in\mathcal{F}$, we see that if $\left\Vert \nabla_{J}A_{\mathfrak{f}}(x,\eta)\right\Vert _{J}\leq\rho f'(\eta)$
then both of the two previous options cannot happen, and hence we
must have $x(\mbox{supp}(\chi))\subseteq U_{k}(F)$.
\end{proof}
Putting these two results together we deduce:
\begin{cor}
\label{cor:the linfty cor}Suppose $\mathfrak{f}=(F,f,\chi,H)\in\mathfrak{F}_{0}$
and $J\in\mathcal{J}$. There exist constants $\rho,T>0$ depending
only on $F$ such that if $(x,\eta)\in\Lambda T^{*}M\times\mathbb{R}$
satisfies\[
\left\Vert \nabla_{J}A_{\mathfrak{f}}(x,\eta)\right\Vert _{J}<\rho f'(\eta)
\]
then\[
\frac{2}{3}\left(A_{\mathfrak{f}}(x,\eta)-T\left\Vert \nabla_{J}A_{\mathfrak{f}}(x,\eta)\right\Vert _{J}-\kappa(H)\right)\leq f(\eta)\leq2\left(A_{\mathfrak{f}}(x,\eta)+T\left\Vert \nabla_{J}A_{\mathfrak{f}}(x,\eta)\right\Vert _{J}+\kappa(H)\right).
\]
\end{cor}
\begin{rem}
\label{rem:Observe-the constants}The constants $\rho(F)$ and $T(F)$
depend continuously on $F$, and depend only on the behavior of $F$
close to $F^{-1}(0)$.
\end{rem}
We now further refine the class of functions $f$ that we consider.
\begin{defn}
\label{def:strange def}Given $a,r>0$ let $\mathcal{F}(a,r)\subseteq\mathcal{F}(a)$
denote those functions that satisfy the additional condition:
\begin{itemize}
\item There exists $A>0$ such that \begin{equation}
Af'(-A)>r.\label{eq:the constant A}
\end{equation}

\end{itemize}
\end{defn}
\begin{rem}
\label{messy f}Given $a>0$ it is possible to construct a function
$f\in\bigcap_{r>0}\mathcal{F}(a,r)$. To do this one first considers
a function $f_{1}\in\mathcal{F}(a)$ such that $f_{1}(\eta)=e^{\eta}$
for $\eta\leq\log\, a/2$. Then for each $n\in\mathbb{N}$, $n\geq\log\, a/2$
one can choose $\varepsilon_{n}>0$ with $\varepsilon_{n}\rightarrow0$
such that $f_{1}$ can be modified on each interval $(-n-1/2,-n+1/2)$
to a new function $f\in\mathcal{F}(a)$ with the property that $f'(\eta)=1$
for $\eta\in(-n-\varepsilon_{n},-n+\varepsilon_{n})$.

A rough construction of this is as follows: given $n>\log\, a/2$
let \[
\delta_{n}:=\frac{1}{2}\left(e^{-n+1/2}-e^{-n-1/2}\right).
\]
Let $f_{2}$ denote the (non-smooth) function such that $f_{2}=f_{1}$
on $\mathbb{R}\backslash\left(\bigcup_{n\geq\log\, a/2}(-n-1/2,-n+1/2)\right)$
and on each interval $(-n-1/2,-n+1/2)$ is the piecewise linear function
\[
f_{2}(\eta)=\begin{cases}
e^{-n-1/2}, & -n-1/2<\eta\leq-n-\delta_{n},\\
\eta+e^{-n-1/2}+n+\delta_{n}, & -n-\delta_{n}\leq\eta\leq-n+\delta_{n},\\
e^{-n+1/2}, & -n+\delta_{n}\leq\eta<-n+1/2.
\end{cases}
\]
Note that $f_{2}$ is continuous by the choice of $\delta_{n}$. Now
set $\varepsilon_{n}:=\frac{1}{2}\delta_{n}$. Then one can construct
a smooth function $f\in\mathcal{F}(a)$ such that $f=f_{2}$ on $\mathbb{R}\backslash\left(\bigcup_{n\geq\log\, a/2}(-n-1/2,-n-\varepsilon_{n})\cup(-n+\varepsilon_{n},-n+\frac{1}{2})\right)$.
See Figure \ref{fig:The-function} below. By construction $f'(-n)=1$
for each $n\geq\log\, a/2$, and hence $f\in\bigcap_{r>0}\mathcal{F}(a,r)$. 

\begin{figure}
\includegraphics[scale=0.6]{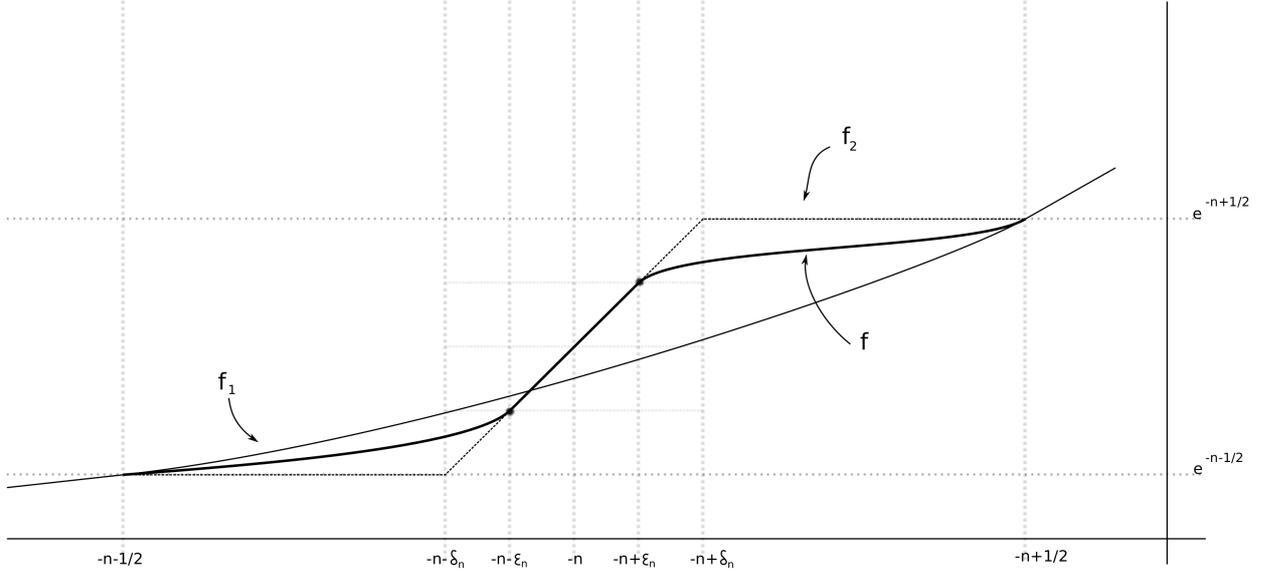}\caption{\label{fig:The-function}The function $f$}

\end{figure}

\end{rem}
The following lemma is elementary, but for the convenience of the
reader we include a proof. 
\begin{lem}
\label{lem:about the set F}For any $a,r>0$ the set $\mathcal{F}(a,r)$
is non-empty and path-connected. If $a'\leq a$ and $r'\geq r$ then
$\mathcal{F}(a',r')\subseteq\mathcal{F}(a,r)$.\end{lem}
\begin{proof}
We have already proved that $\mathcal{F}(a,r)$ is non-empty (see
Remark \ref{messy f} above). To show that $\mathcal{F}(a,r)$ is
path-connected, first observe that if $f_{0},f_{1}\in\mathcal{F}(a,r)$
both satisfy \eqref{eq:the constant A} with the \textbf{same}\emph{
}constant $A>0$ then the linear homotopy $f_{s}:=sf_{1}+(1-s)f_{0}$
is contained in $\mathcal{F}(a,r)$ for all $s\in[0,1]$. It therefore
suffices to show that if $f\in\mathcal{F}(a,r)$ satisfies \eqref{eq:the constant A}
with respect to some $A>0$, then given any $B>A$ we can find a new
function $f_{1}\in\mathcal{F}(a,r)$ that satisfies \eqref{eq:the constant A}
with respect to $B$, and such that we may find a homotopy $(f_{s})_{s\in[0,1]}\subseteq\mathcal{F}(a,r)$
with $f_{0}=f$. 

In order to do this, let $(\lambda_{s})_{s\in[0,1]}$ denote a family
of smooth functions $\lambda_{s}:\mathbb{R}\rightarrow\mathbb{R}$
such that::\[
\lambda_{s}(\eta)=\begin{cases}
\eta, & 0\leq\eta<\infty,\\
\eta+s(B-A), & -\infty\leq\eta\leq-sB;
\end{cases}\ \ \ 0<\lambda_{s}'\leq1
\]
(such functions $\lambda_{s}$ exist as $A<B$). Set $f_{s}:=f\circ\lambda_{s}$.
We claim that $f_{s}\in\mathcal{F}(a,r)$ for each $s\in[0,1]$. It
is clear that $f_{s}\in\mathcal{F}(a)$ for each $s\in[0,1]$. Moreover,
\[
(A+s(B-A))f_{s}'(-A-s(B-A))=(A+s(B-A))f'(-A)\geq Af'(-A)>r.
\]
Thus $f_{s}$ satisfies \eqref{eq:the constant A} with respect to
$A+s(B-A)$ for each $s\in[0,1]$. The last statement of the lemma
is immediate, and hence this completes the proof.
\end{proof}
The next result uses the same idea as \cite[Proposition 5.5]{CieliebakFrauenfelderOancea2010},
and shows that for a suitable choice of $f\in\mathcal{F}$ one can
bound the $\eta$ component of gradient flow lines with action in
a fixed interval.
\begin{prop}
\label{pro:useful dichotomy}Fix $F\in\mathcal{D}$ and $0<a<b<\infty$.
Let $\rho,T>0$ be the constants associated to $F$ from Corollary
\ref{cor:the linfty cor}. Let $f\in\mathcal{F}\left(\frac{a}{6},\frac{b-a}{\min\{\rho,a/4T\}}\right)$
and $H\in\mathcal{H}(a/2)$. Choose $\chi$ such that $\mathfrak{f}:=(F,f,\chi,H)\in\mathfrak{F}_{0}$
and choose $J\in\mathcal{J}$. There exists a constant $C_{\textrm{\emph{mult}}}>0$
depending only on $a,b,F$ and $f$, such that if $u=(x,\eta)\in\mathcal{M}^{(a,b)}(\nabla_{J}A_{\mathfrak{f}})$,
then $\left\Vert \eta\right\Vert _{L^{\infty}}\leq C_{\textrm{\emph{mult}}}$. \end{prop}
\begin{proof}
First note that \begin{equation}
\lim_{s\rightarrow\pm\infty}\eta(s)\geq\frac{a}{2}.\label{eq:limpos}
\end{equation}
Indeed, this follows from the fact that by \eqref{eq:value of perturbed functional at crit},
\[
A_{\mathfrak{f}}(x_{\pm},\eta_{\pm})=f(\eta_{\pm})-\int_{0}^{1}\{\lambda(X_{H}(t,x_{\pm}))-H(t,x_{\pm})\}dt,
\]
and hence \[
f(\eta_{\pm})\geq A_{\mathfrak{f}}(x_{\pm},\eta_{\pm})-\kappa(H)\geq\frac{a}{2}.
\]
Since $f\in\mathcal{F}(a/6)$ one therefore has $\eta_{\pm}\geq a/2$. 

It will be convenient to define \[
\rho_{1}:=\min\left\{ \rho,\frac{a}{4T}\right\} ,
\]
so that $f\in\mathcal{F}\left(\frac{a}{6},\frac{b-a}{\rho_{1}}\right)$.
By definition of the set $\mathcal{F}\left(\frac{a}{6},\frac{b-a}{\rho_{1}}\right)$,
there exists \textbf{$A>0$} such that\begin{equation}
f'(-A)A>\frac{b-a}{\rho_{1}}.\label{eq:B ass}
\end{equation}
Fix $u\in\mathcal{M}^{(a,b)}(\nabla_{J}A_{\mathfrak{f}})$. Define
a function $\sigma_{u}:\mathbb{R}\rightarrow[0,\infty)$ by \begin{equation}
\sigma_{u}(s):=\inf\left\{ \sigma\geq0\,:\,\left\Vert \nabla_{J}A_{\mathfrak{f}}(u(s+\sigma))\right\Vert _{J}\leq\rho_{1}f'(\eta(s+\sigma))\right\} ;\label{eq:sigma u of s}
\end{equation}
$\sigma_{u}$ is well defined as $\lim_{s\rightarrow\infty}f'(\eta(s))=1$
(from \eqref{eq:limpos} and the fact that $f\in\mathcal{F}(a/6)$),
and $\lim_{s\rightarrow\infty}\left\Vert \nabla_{J}A_{\mathfrak{f}}(u(s))\right\Vert _{J}=0$.
Next define

\[
i_{u}(s):=\inf_{s\leq r\leq s+\sigma_{u}(s)}f'(\eta(r)).
\]
Note that \[
E(u)=\int_{-\infty}^{\infty}\left\Vert \nabla_{J}A_{\mathfrak{f}}(u(s))\right\Vert _{J}^{2}ds=\lim_{s\rightarrow-\infty}A_{\mathfrak{f}}(u(s))-\lim_{s\rightarrow\infty}A_{\mathfrak{f}}(u(s))\leq b-a.
\]
Next, \begin{align*}
E(u) & \geq\int_{s}^{s+\sigma_{u}(s)}\left\Vert \nabla_{J}A_{\mathfrak{f}}(u(r))\right\Vert _{J}^{2}dr\\
 & \geq\int_{s}^{s+\sigma_{u}(s)}\rho_{1}^{2}f'(\eta(r))^{2}dr\\
 & \geq\rho_{1}^{2}i_{u}(s)^{2}\sigma_{u}(s),
\end{align*}
and hence \begin{equation}
\sigma_{u}(s)\leq\frac{b-a}{\rho_{1}^{2}i_{u}(s)^{2}}.\label{eq:sigma u}
\end{equation}
Now observe that \begin{align*}
\left|\eta(s)-\eta(s+\sigma_{u}(s))\right| & \leq\int_{s}^{s+\sigma_{u}(s)}\left|\partial_{r}\eta(r)\right|dr\\
 & \leq\left(\sigma_{u}(s)\int_{s}^{s+\sigma_{u}(s)}\left|\partial_{r}\eta(r)\right|^{2}dr\right)^{1/2}\\
 & \leq\left(\sigma_{u}(s)\int_{s}^{s+\sigma_{u}(s)}\left\Vert \nabla_{J}A_{\mathfrak{f}}(u(r))\right\Vert _{J}^{2}dr\right)^{1/2}\\
 & \leq(\sigma_{u}(s)E(u))^{1/2}\\
 & \leq\frac{b-a}{\rho_{1}i_{u}(s)},
\end{align*}
where the last line used \eqref{eq:sigma u}. Next, Corollary \ref{cor:the linfty cor}
implies that for any $s\in\mathbb{R}$, \begin{align*}
f(\eta(s+\sigma_{u}(s)) & \geq\frac{2}{3}\left(A_{\mathfrak{f}}(u(s+\sigma_{u}(s)))-T\left\Vert \nabla_{J}A_{\mathfrak{f}}(u(s+\sigma_{u}(s)))\right\Vert _{J}-\kappa(H)\right)\\
 & \geq\frac{2}{3}\left(a-T\rho_{1}f'(\eta(s+\sigma_{u}(s)))-\frac{a}{2}\right).
\end{align*}
Since $\rho_{1}\leq a/4T$ and $f'\leq1$, we deduce that \[
f(\eta(s+\sigma_{u}(s))\geq\frac{a}{6}.
\]
Since $f\in\mathcal{F}(a/6)$ we see that \[
\eta(s+\sigma_{u}(s))\geq\frac{a}{6}>0,
\]
and thus \[
\eta(s)\geq\frac{a}{6}-\frac{b-a}{\rho_{1}i_{u}(s)}>-\frac{b-a}{\rho_{1}i_{u}(s)}.
\]
In particular, \[
f'(\eta(s))\eta(s)\geq i_{u}(s)\eta(s)>-\frac{(b-a)}{\rho_{1}}.
\]
Using \eqref{eq:limpos}, if there exists some $s_{0}\in\mathbb{R}$
such that $\eta(s_{0})<-A$ then by continuity there exists $s_{1}\in\mathbb{R}$
such that $\eta(s_{1})=-A$. But then we obtain a contradiction via
\eqref{eq:B ass} \[
-\frac{b-a}{\rho_{1}}>-f'(-A)A=f'(\eta(s_{1}))\eta(s_{1})>-\frac{b-a}{\rho_{1}}.
\]
It follows that $\eta(s)>-A$ for all $s\in\mathbb{R}$.

Now we address the upper bound. Define a new function $\widetilde{\sigma}_{u}:\mathbb{R}\rightarrow[0,\infty)$
by \begin{equation}
\widetilde{\sigma}_{u}(s):=\inf\left\{ \sigma\geq0\,:\,\left\Vert \nabla_{J}A_{\mathfrak{f}}(u(s+\sigma))\right\Vert _{J}\leq\rho_{1}f'(-A)\right\} .\label{eq:new sigma}
\end{equation}
Arguing as above we see that for any $s\in\mathbb{R}$, \[
\widetilde{\sigma}_{u}(s)\leq\frac{b-a}{\rho_{1}^{2}f'(-A)^{2}},
\]
and hence\begin{equation}
\left|\eta(s)-\eta(s+\widetilde{\sigma}_{u}(s))\right|\leq\frac{b-a}{\rho_{1}f'(-A)}<A,\label{eq:sigma twiddle}
\end{equation}
where the last inequality used \eqref{eq:B ass} again. Then by Corollary
\ref{cor:the linfty cor} we see that for any $s\in\mathbb{R}$,\begin{align*}
f(\eta(s+\widetilde{\sigma}_{u}(s)) & \leq2\left(A_{\mathfrak{f}}(u(s+\widetilde{\sigma}_{u}(s)))+T\left\Vert \nabla_{J}A_{\mathfrak{f}}(u(s+\widetilde{\sigma}_{u}(s)))\right\Vert _{J}+\kappa(H)\right)\\
 & \leq2(b+T\rho_{1}f'(-A)+a/2)\leq2a+2b,
\end{align*}
and hence $\eta(s+\widetilde{\sigma}_{u}(s))\leq2a+2b$. Thus by \eqref{eq:sigma twiddle},\[
\eta(s)<2a+2b+A.
\]
We conclude that \[
\sup_{s\in\mathbb{R}}\left|\eta(s)\right|<C_{\textrm{mult}}=C_{\textrm{mult}}(a,b,f):=2a+2b+A.
\]

\end{proof}
Proposition \ref{pro:useful dichotomy} prompts the following definition.
\begin{defn}
\label{def:strange def 2}Given $F\in\mathcal{D}$ and $0<a<b<\infty$,
let \[
\mathcal{F}(F,a,b):=\mathcal{F}\left(\frac{a}{6},\frac{b-a}{\rho_{1}}\right),
\]
where $\rho_{1}=\min\{\rho,a/4T\}$ and $\rho=\rho(F)$ and $T=T(F)$
are the constants from Corollary \ref{cor:the linfty cor}. 
\end{defn}
The following result is the main one of this section. 
\begin{thm}
\label{thm:admissable yes}Fix $\alpha\in[S^{1},M]$. Suppose $\mathfrak{f}=(F,f,\chi,H)\in\mathfrak{F}_{0,\textrm{\emph{reg}}}$
and $0<a<b<\infty$ are such that $a,b\notin\mathcal{A}(A_{\mathfrak{f}},\alpha)$.
Suppose also that $f\in\mathcal{F}(F,a,b)$ where $F\in\mathcal{D}(\Sigma)$,
and $H\in\mathcal{H}(a/2)$. Let $S$ denote a fibrewise starshaped
hypersurface such that $D(\Sigma)\subseteq D^{\circ}(S)$ and such
that $\mbox{\emph{supp}}(X_{F})\subseteq D^{\circ}(S)$. Choose $J\in\mathcal{J}_{\textrm{\emph{reg}}}(\mathfrak{f})\cap\mathcal{J}(S)$.
Then the quadruple $\mathfrak{q}:=(\mathfrak{f},a,b,J)$ is $\alpha$-admissible.\end{thm}
\begin{proof}
Immediate from Corollary \ref{cor:cptness of critical set}, Lemma
\ref{lem:why convex good} and Proposition \ref{pro:useful dichotomy}.\end{proof}
\begin{rem}
\label{independent of f}In fact, one can show using an argument based
on Floer's \textbf{bifurcation method }(see \cite[Proposition 4.11]{CieliebakFrauenfelderOancea2010})
that in the situation above, $HF^{(a,b)}(A_{F^{\chi},f}^{H},\alpha)$
is actually \textbf{independent }of the choice of $f\in\mathcal{F}(F,a,b)$.
Nevertheless, for the purposes of the present paper we do not need
this observation, and we will make no use of it. 
\end{rem}

\subsection{\label{sub:Truncating-the-function}Truncating the function $f$}

$\ $\vspace{6 pt}

\textbf{A posteriori}, we discover that one can truncate the function
$f$ at infinity without affecting the Floer homology. Indeed, fix
$\alpha\in[S^{1},M]$ and a non-degenerate fibrewise starshaped hypersurface
$\Sigma$ and $0<a<b<\infty$ such that $a,b\notin\mathcal{A}(\Sigma,\alpha)$.
Choose $F\in\mathcal{D}(\Sigma)$, $f\in\mathcal{F}(F,a,b)$ and $\chi\in\mathcal{X}$.
Set $\mathfrak{f}:=(F,f,\chi,0)\in\mathfrak{F}_{0,\textrm{reg}}'$.
Suppose $R>2a+2b+A+1$, where $A>0$ is the constant associated to
$f$ from \eqref{eq:the constant A}. Let $\bar{f}:\mathbb{R}\rightarrow\mathbb{R}^{+}$
denote a smooth function such that $\bar{f}\equiv f$ on $(-\infty,R-1]$
and such that $\bar{f}(\eta)=R$ for $\eta\in[R+1,\infty)$, with
$0\leq\bar{f}'(\eta)\leq1$ on all of $\mathbb{R}$. We will call
such a function $\bar{f}$ an \textbf{$R$-truncation}\emph{ }of $f$.
Let $\bar{\mathfrak{f}}:=(F,\bar{f},\chi,0)$.

Consider the Rabinowitz action functional $A_{\bar{\mathfrak{f}}}$.
This functional will have many more critical points than $A_{\mathfrak{f}}$,
as $\bar{f}'$ is no longer strictly positive everywhere (i.e. one
can no longer deduce \eqref{eq:sec eq} from \eqref{eq:first eq}).
However if $(x,\eta)$ is a critical point of $A_{\bar{\mathfrak{f}}}$
with $\bar{f}'(\eta)=0$ then we necessarily have $\eta\geq R-1$,
and hence $A_{\bar{\mathfrak{f}}}(x,\eta)=\bar{f}(\eta)\geq R-1$
by Lemma \ref{lem:key lemma}.1. In particular, $(x,\eta)\notin\mbox{Crit}^{(a,b)}(A_{\bar{\mathfrak{f}}})$.
We conclude that \[
\mbox{Crit}^{(a,b)}(A_{\bar{\mathfrak{f}}},\alpha)=\mbox{Crit}^{(a,b)}(A_{\mathfrak{f}},\alpha).
\]
In particular, this implies the Rabinowitz Floer complexes $CF^{(a,b)}(A_{\bar{\mathfrak{f}}},\alpha)$
and $CF^{(a,b)}(A_{\mathfrak{f}},\alpha)$ coincide as \textbf{groups}.\emph{
}Moreover the proof of Proposition \ref{pro:useful dichotomy} shows
that the $\eta$-component of a gradient flow line $u=(x,\eta)\in\mathcal{M}^{(a,b)}(\nabla_{J}A_{\bar{\mathfrak{f}}})$
never escapes the interval $(-A,2a+2b+A]$ (for any $J\in\mathcal{J}$).
In particular, $\eta$ never escapes $(-\infty,R-1)$. Since $f\equiv\bar{f}$
on $(-\infty,R-1]$, it follows that the differential of the two Floer
complexes (with respect to a suitably chosen almost complex structure)
is also the same, whence it follows that \[
HF^{(a,b)}(A_{\bar{\mathfrak{f}}},\alpha)\cong HF^{(a,b)}(A_{\mathfrak{f}},\alpha).
\]
We will use this observation in the proof of Lemma \ref{lem:its an iso}
below.

\subsection{\label{sub:Inclusion/Quotient-maps}Inclusion/Quotient maps}

$\ $\vspace{6 pt}

Let us make the following observation. Suppose we are given $a,b,c,d>0$
such that $a<\min\{b,c\}$ and $d>\max\{b,c\}$. Fix $\alpha\in[S^{1},M]$.
Suppose $\Sigma$ is a non-degenerate fibrewise starshaped hypersurface
and $F\in\mathcal{D}(\Sigma)$, and suppose that $a,b,c,d\notin\mathcal{A}(\Sigma,\alpha)$
and $f\in\mathcal{F}(F,a,d)$. Choose $\chi\in\mathcal{X}$ and $H\in\mathcal{H}(a/2)$
such that $\mathfrak{f}=(F,f,\chi,H)\in\mathfrak{F}_{0,\textrm{reg}}$.
Fix an almost complex structure $J\in\mathcal{J}_{\textrm{ref}}(\mathfrak{f})\cap\mathcal{J}(S)$,
where $S$ is a fibrewise starshaped hypersurface such that $D(\Sigma)\subseteq D^{\circ}(S)$.
Our hypotheses imply that the three Floer homology groups \[
HF^{(a,b)}(A_{\mathfrak{f}},\alpha),\ HF^{(a,d)}(A_{\mathfrak{f}},\alpha)\ \mbox{and }HF^{(c,d)}(A_{\mathfrak{f}},\alpha)
\]
are all well defined. There are natural chain maps between the three
groups given by \[
CF^{(a,b)}(A_{\mathfrak{f}_{i}},\alpha)\overset{\textrm{inclusion}}{\rightarrow}CF^{(a,d)}(A_{\mathfrak{f}_{i}},\alpha)
\]
and \[
CF^{(a,d)}(A_{\mathfrak{f}_{i}},\alpha)\overset{\textrm{quotient}}{\rightarrow}CF^{(a,d)}(A_{\mathfrak{f}_{i}},\alpha)/CF^{(a,c)}(A_{\mathfrak{f}_{i}},\alpha)=CF^{(c,d)}(A_{\mathfrak{f}_{i}},\alpha).
\]
We denote by \begin{equation}
i:HF^{(a,b)}(A_{\mathfrak{f}},\alpha)\rightarrow HF^{(c,d)}(A_{\mathfrak{f}},\alpha)\label{eq:inclusionquotientmap}
\end{equation}
the induced map on homology given by the composition of these two
maps. It is clear that if \[
\mathcal{A}(\Sigma,\alpha)\cap[a,c]=\mathcal{A}(\Sigma,\alpha)\cap[b,d]=\emptyset
\]
then $i$ is an isomorphism.

\subsection{\label{sub:The-Floer-homology}The Floer homology groups $HF^{(a,\infty)}(A_{F,f},\alpha)$}

$\ $\vspace{6 pt}

In this section we extend the definition of $HF^{(a,b)}$ to cover
the case $b=\infty$. Suppose $\Sigma\subseteq T^{*}M$ is a non-degenerate
fibrewise starshaped hypersurface. From this moment on it will be
convenient to work with just \textbf{one} function $f$, instead of
picking a function $f$ for each action interval $(a,b)$. For this
purpose, set $\ell:=\ell(\Sigma)$ and choose \begin{equation}
f\in\bigcap_{r>0}\mathcal{F}(\ell/12,r)\label{eq:one f choice}
\end{equation}
(such functions exist by Remark \ref{messy f}). This function $f$
has the desirable property%
\footnote{As a result, from now on we will abandon the notation $\mathcal{F}(F,a,b)$
and solely work with functions $f$ satisfying \eqref{eq:one f choice}
instead.%
} that given any $F\in\mathcal{D}(\Sigma)$ and any $\ell/2<a<b<\infty$
we have $f\in\mathcal{F}(F,a,b)$. 

Fix $F\in\mathcal{D}(\Sigma)$ and $\alpha\in[S^{1},M]$. Then for
any $\ell/2<a<b<\infty$ such that $a,b\notin\mathcal{A}(\Sigma,\alpha)$,
the Floer homology $HF^{(a,b)}(A_{F,f},\alpha)$ is defined. Moreover
if $c>b$ also satisfies $c\notin\mathcal{A}(\Sigma,\alpha)$, then
from Section \ref{sub:Inclusion/Quotient-maps} there is a natural
map $HF^{(a,b)}(A_{F,f},\alpha)\rightarrow HF^{(a,c)}(A_{F,f},\alpha)$.
These maps form a directed system, and hence we can define\[
HF^{(a,\infty)}(A_{F,f},\alpha):=\underset{b\rightarrow\infty}{\underrightarrow{\lim}}HF^{(a,b)}(A_{F,f},\alpha).
\]
We denote by \begin{equation}
\iota_{a}^{b}:HF^{(a,b)}(A_{F,f},\alpha)\rightarrow HF^{(a,\infty)}(A_{F,f},\alpha)\label{eq:iota}
\end{equation}
the induced map. Since we also have natural maps $HF^{(a,c)}(A_{F,f},\alpha)\rightarrow HF^{(b,c)}(A_{F,f},\alpha)$,
there is an induced map \[
\pi_{a}^{b}:HF^{(a,\infty)}(A_{F,f},\alpha)\rightarrow HF^{(a,b)}(A_{F,f},\alpha).
\]
For future use, given $3\ell/4<a<b<\infty$ with $a,b\notin\mathcal{A}(\Sigma,\alpha)$
let us denote by \[
Z(a,b):=\pi_{3\ell/4}^{b}\circ\iota_{3\ell/4}^{a},
\]
 so that $Z(a,b)$ is a map \begin{equation}
Z(a,b):HF^{(3\ell/4,a)}(A_{F,f},\alpha)\rightarrow HF^{(b,\infty)}(A_{F,f},\alpha).\label{eq:the map Z-1}
\end{equation}

Note that $Z(a,b)=0$ if $a<b$.

\section{\label{sec:Continuation-homomorphisms}Continuation homomorphisms}

In this section we develop the theory of \textbf{continuation homomorphisms
}for the Rabinowitz action functional $A_{\mathfrak{f}}$. Continuation
homomorphisms in Floer theory were introduced originally by Floer
in \cite{Floer1989}, and are a powerful tool for proving invariance
results for Floer homology. The main reason for introducing the function
$f$ is that, as we will see below, these Floer homology groups behave
well with respect to monotone homotopies. This is in contrast to the
usual Rabinowitz Floer homology groups (see for instance \cite{CieliebakFrauenfelder2009}),
for which it is not known whether they behave well with respect to
monotone homotopies, see Remark \ref{The-innocent-looking} below.

\subsection{Continuation maps}

$\ $\vspace{6 pt}

We begin with a discussion of continuation maps in the most general
form that we will need. From now on we will be somewhat sloppy in
our treatment of almost complex structures; wherever possible we will
suppress them from the notation and from our discussion. Sometimes
however we will be forced to include them in our notation (see for
instance \eqref{continuation eq-1} below). In general the reader
should think of $(J_{s})_{s\in[0,1]}$ as a generically chosen family
of almost complex structures that all lie in $\mathcal{J}(S)$ for
some fixed large fibrewise starshaped hypersurface $S$. We will not
specify precisely what conditions $(J_{s})$ must satisfy, and will
content ourselves with merely stating that these conditions are generically
satisfied. In keeping with our new policy of supressing the mention
of $J$, from now on we will refer to a triple $(\mathfrak{f},a,b)$
as being \textbf{admissible} if $(\mathfrak{f},a,b,J)$ is admissible
(in the sense of Definition \ref{Fix-.-Ad}).\newline

Suppose we are given a smooth family $\mathfrak{f}_{s}=(F_{s},f_{s},\chi_{s},H_{s})\subseteq\mathfrak{F}_{0}$
for $s\in[0,1]$. Assume that $\mathfrak{f}_{0}$ and $\mathfrak{f}_{1}$
lie in $\mathfrak{F}_{0,\textrm{reg}}$. Let us fix once and for all
a smooth cut-off function $\beta:\mathbb{R}\rightarrow[0,1]$ such
that $\beta(s)=0$ for $s\leq0$ and $\beta(s)=1$ for $s\geq1$,
and $0\leq\beta'(s)\leq2$ for all $s\in\mathbb{R}$. Let $\mathcal{N}(\nabla A_{\mathfrak{f}_{s}},\alpha)$
denote the set of maps $u=(x,\eta):\mathbb{R}\rightarrow\Lambda_{\alpha}T^{*}M\times\mathbb{R}$
that satisfy\[
\partial_{s}u+\nabla_{J_{\beta(s)}}A_{\mathfrak{f}_{\beta(s)}}(u)=0.
\]
It would be more accurate to write $\mathcal{N}(\nabla_{J_{\beta(s)}}A_{\mathfrak{f}_{\beta(s)}},\alpha)$,
but we omit the {}``$J_{s}$'' and the {}``$\beta$'' in order
to make the notation slightly less cumbersome. Thus $\mathcal{N}(\nabla A_{\mathfrak{f}_{s}},\alpha)$
is the set of maps $u=(x,\eta):\mathbb{R}\rightarrow\Lambda_{\alpha}T^{*}M\times\mathbb{R}$
that satisfy:\begin{equation}
\begin{cases}
\partial_{s}x+J_{\beta(s),t}(x)(\partial_{t}x-f_{\beta(s)}(\eta)\chi_{\beta(s)}(t)X_{F_{\beta(s)}}(x)+X_{H_{\beta(s)}}(t,x))=0\\
\partial_{s}\eta-f_{\beta(s)}'(\eta(s))\int_{0}^{1}F_{\beta(s)}(x)dt=0.
\end{cases}\label{continuation eq-1}
\end{equation}
If $u=(x,\eta)$ satisfies \eqref{continuation eq-1} and has finite
energy $E(u)<\infty$ then as before the limits\begin{equation}
\lim_{s\rightarrow\pm\infty}u(s,t)=:v_{\pm}(t)=(x_{\pm}(t),\eta_{\pm}),\ \ \ \lim_{s\rightarrow\pm\infty}\partial_{s}u(s,t)=0\label{cont 2 eq-1}
\end{equation}
exist and are uniform in the $t$-variable. Moreover $v_{-}\in\mbox{Crit}(A_{\mathfrak{f}_{0}},\alpha)$
and $v_{+}\in\mbox{Crit}(A_{\mathfrak{f}_{1}},\alpha)$. 

Given $u\in\mathcal{N}(\nabla A_{\mathfrak{f}_{s}},\alpha)$ and $-\infty\leq s_{0}\leq s_{1}\leq\infty$,
set \[
\Delta_{s_{0}}^{s_{1}}(u):=\int_{s_{0}}^{s_{1}}\left(\frac{\partial}{\partial s}A_{\mathfrak{f}_{\beta(s)}}\right)(u(s))ds.
\]
Write $\Delta(u):=\Delta_{-\infty}^{\infty}(u)$. Following Ginzburg
\cite{Ginzburg2007}, given $C\geq0$ let us say the family $(\mathfrak{f}_{s})$
is\textbf{ $C$-bounded }if for every $u\in\mathcal{N}(\nabla A_{\mathfrak{f}_{s}},\alpha)$
and every $-\infty\leq s_{0}\leq s_{1}\leq\infty$ it holds that\[
\Delta_{s_{0}}^{s_{1}}(u)\leq C.
\]
In order to explain the relevance of the term $\Delta_{s_{0}}^{s_{1}}(u)$,
given $a,b>0$ denote by $\mathcal{N}_{a}^{b}(\nabla A_{\mathfrak{f}_{s}},\alpha)$
the subset of $\mathcal{N}(\nabla A_{\mathfrak{f}_{s}},\alpha)$ consisting
of those maps $u$ that satisfy\[
\lim_{s\rightarrow-\infty}A_{\mathfrak{f}_{\beta(s)}}(u(s))\leq b,\ \ \ \lim_{s\rightarrow\infty}A_{\mathfrak{f}_{\beta(s)}}(u(s))\geq a.
\]
Then if $u\in\mathcal{N}_{a}^{b}(\nabla A_{\mathfrak{f}_{s}},\alpha)$
one readily checks that\begin{equation}
E(u)\leq b-a+\Delta(u);\label{eq:en eq}
\end{equation}
\begin{equation}
\sup_{s\in\mathbb{R}}A_{\mathfrak{f}_{\beta(s)}}(u(s))\leq b+\sup_{s\in\mathbb{R}}\Delta_{-\infty}^{s}(u);\label{eq:sup eq}
\end{equation}

\begin{equation}
\inf_{s\in\mathbb{R}}A_{\mathfrak{f}_{\beta(s)}}(u(s))\geq a-\sup_{s\in\mathbb{R}}\Delta_{s}^{\infty}(u).\label{eq:inf eq}
\end{equation}
In particular,\[
\lim_{s\rightarrow\infty}A_{\mathfrak{f}_{\beta(s)}}(u(s))\leq b+\Delta(u);
\]
\[
\lim_{s\rightarrow\infty}A_{\mathfrak{f}_{\beta(s))}}(u(s))\geq a-\Delta(u).
\]

\begin{defn}
\label{def:admissible families}Fix a family $(\mathfrak{f}_{s})_{s\in[0,1]}$
as above, and fix $a,b>0$ and $C\geq0$. We say that $\{(\mathfrak{f}_{s}),a,b,C\}$
is an \textbf{$\alpha$-admissible family }if 
\begin{enumerate}
\item The triples $(\mathfrak{f}_{0},a,b)$ and $(\mathfrak{f}_{1},a+C,b+C)$
are $\alpha$-admissible. Thus $HF^{(a,b)}(A_{\mathfrak{f}_{0}},\alpha)$
and $HF^{(a+C,b+C)}(A_{\mathfrak{f}_{1}},\alpha)$ are well defined.
\item The family $(\mathfrak{f}_{s})$ is $C$-bounded.
\item There exist constants $C_{\textrm{loop}},C_{\textrm{mult}}>0$ such
that if $u=(x,\eta)\in\mathcal{N}_{a}^{b}(\nabla A_{\mathfrak{f}_{s}},\alpha)$
then it holds that $\left\Vert x\right\Vert _{L^{\infty}}<C_{\textrm{loop}}$
and $\left\Vert \eta\right\Vert _{L^{\infty}}<C_{\textrm{mult}}$.
\end{enumerate}
\end{defn}
The following basic theorem follows from standard Floer homological
methods; see for instance \cite[Section 4.4]{BiranPolterovichSalamon2003}
or \cite[Section 3.2.3]{Ginzburg2007}.
\begin{thm}
\textbf{\emph{\label{thm:(Continuity-property-of}(Continuity properties
of filtered Floer homology)}}
\begin{enumerate}
\item Suppose $\{(\mathfrak{f}_{s}),a,b,C\}$ is an $\alpha$-admissible
family. Then there exists a chain map \[
\Psi:CF^{(a,b)}(A_{\mathfrak{f}_{0}},\alpha)\rightarrow CF^{(a+C,b+C)}(A_{\mathfrak{f}_{1}},\alpha)
\]
which induces a homomorphism \[
\psi:HF^{(a,b)}(A_{\mathfrak{f}_{0}},\alpha)\rightarrow HF^{(a+C,b+C)}(A_{\mathfrak{f}_{1}},\alpha).
\]

\item Suppose $c,d>0$ are such that $a\leq c$ and $b\leq d$. Suppose
in addition that $\{(\mathfrak{f}_{s}),c,d,C\}$ is $\alpha$-admissible.
Then the following diagram commutes:\[
\xymatrix{HF^{(a,b)}(A_{\mathfrak{f}_{0}},\alpha)\ar[r]^{\psi}\ar[d]_{i} & HF^{(a+C,b+C)}(A_{\mathfrak{f}_{1}},\alpha)\ar[d]^{i}\\
HF^{(c,d)}(A_{\mathfrak{f}_{0}},\alpha)\ar[r]_{\psi} & HF^{(c+C,d+C)}(A_{\mathfrak{f}_{1}},\alpha)
}
\]
Here the vertical maps are the maps from \eqref{eq:inclusionquotientmap}.
\end{enumerate}
\end{thm}

\subsection{\label{sub:Monotone-homotopies-and}Monotone homotopies}

$\ $\vspace{6 pt}

In this section we suppose we are given two non-degenerate fibrewise
starshaped hypersurfaces $\Sigma$ and $\Sigma'$ with the property
that \[
D(\Sigma')\subseteq D(\Sigma).
\]

Let us first fix a smooth family $(\Sigma_{s})_{s\in[0,1]}$ of fibrewise
starshaped hypersurfaces such that:
\begin{enumerate}
\item $\Sigma_{0}=\Sigma$ and $\Sigma_{1}=\Sigma'$;
\item For generic $s\in[0,1]$, $\Sigma_{s}$ is non-degenerate;
\item for any $0\leq s_{0}\leq s_{1}\leq1$ one has $D(\Sigma_{s_{1}})\subseteq D(\Sigma_{s_{0}})$.
\end{enumerate}
We will call such a family a \textbf{concentric }family of fibrewise
starshaped hypersurfaces. Given such a family $(\Sigma_{s})$ it is
possible%
\footnote{For example, one could first let $\widetilde{F}_{s}$ denote the Hamiltonian
constructed at the start of Section \ref{sub:-estimates-for} (see
\eqref{eq:nice H}) below for $\Sigma=\Sigma_{s}$, and then set $F_{s}:=(\widetilde{F}_{s})_{R}$
as in Section \ref{sub:-estimates-for} for some $R>1$.%
} to choose a smooth family $(F_{s})_{s\in[0,1]}\subseteq\mathcal{D}$
of Hamiltonians such that $F_{s}\in\mathcal{D}(\Sigma_{s})$ and such
that $\partial_{s}F_{s}(q,p)\geq0$ for all $(q,p)\in T^{*}M$. For
the remainder of this section we fix such a family $(F_{s})$.

Set \[
\ell:=\min_{s\in[0,1]}\ell(\Sigma_{s})>0,
\]
and fix once and for all a function $f\in\bigcap_{r>0}\mathcal{F}\left(\ell/12,r\right)$.
By construction, given any $s\in[0,1]$ such that $\Sigma_{s}$ is
non-degenerate, and any $\alpha\in[S^{1},M]$, $\chi\in\mathcal{X}$,
and $\ell/2<a<b<\infty$ such that $\mathcal{A}(\Sigma_{s},\alpha)\cap\{a,b\}=\emptyset$,
the Floer homology $HF^{(a,b)}(A_{F_{s}^{\chi},f},\alpha)$ is well
defined. In this section we will only ever use $\chi\equiv1$, so
let us set \[
\mathfrak{f}_{s}:=(F_{s},f,1,0).
\]
Now let us fix $\alpha\in[S^{1},M]$. Suppose we are given $\ell/2<a<b<\infty$
such that $a,b\notin\mathcal{A}(\Sigma,\alpha)\cup\mathcal{A}(\Sigma',\alpha)$.
Then we claim there exists a chain map \[
\Psi_{0}^{1}:CF^{(a,b)}(A_{\mathfrak{f}_{0}},\alpha)\rightarrow CF^{(a,b)}(A_{\mathfrak{f}_{1}},\alpha)
\]
inducing a homomorphism \[
\psi_{0}^{1}:HF^{(a,b)}(A_{\mathfrak{f}_{0}},\alpha)\rightarrow HF^{(a,b)}(A_{\mathfrak{f}_{1}},\alpha).
\]

This follows readily from our discussion above. Indeed, we claim that
$\{(\mathfrak{f}_{s}),a,b,0\}$ is an $\alpha$-admissible family.
Condition (1) of Definition \ref{def:admissible families} is satisfied
by assumption, and since $\partial_{s}F_{s}\geq0$ we have $\Delta_{s_{0}}^{s_{1}}(u)\leq0$
for all $u\in\mathcal{N}(\nabla A_{\mathfrak{f}_{s}},\alpha)$ and
$-\infty\leq s_{0}\leq s_{1}\leq\infty$, which shows Condition (2)
is satisfied. 
\begin{rem}
\label{The-innocent-looking}The innocent looking fact that $\partial_{s}F_{s}\geq0$
implies $\Delta_{s_{0}}^{s_{1}}(u)\leq0$ is in fact the key point
of the present paper, and the whole point of perturbing the Rabinowitz
action functional with a \textbf{positive }function $f$. In the setting
of 'standard' Rabinowitz Floer homology, the corresponding expression
for $\Delta_{s_{0}}^{s_{1}}(u)$ is given by $-\int_{s_{0}}^{s_{1}}\eta(s)\int_{0}^{1}\partial_{s}F_{\beta(s)}(x(s))ds$
instead of $-\int_{s_{0}}^{s_{1}}f(\eta(s))\int_{0}^{1}\partial_{s}F_{\beta(s)}(x(s))ds$.
Since the Lagrange multiplier $\eta(s)$ could very well become negative,
one cannot conclude from $\partial_{s}F_{s}\geq0$ that $\Delta_{s_{0}}^{s_{1}}(u)\leq0$
in the standard case.
\end{rem}
The existence of a constant $C_{\textrm{loop}}>0$ satisfying the
requirements of Condition (3) follows from the choice of a correct
almost complex structure, and we will say nothing about this (see
the opening paragraph of Section \ref{sec:Continuation-homomorphisms}).
Equations \eqref{eq:en eq}, \eqref{eq:sup eq} and \eqref{eq:inf eq}
show that the proof of Proposition \ref{pro:useful dichotomy} goes
through without change to establish the existence of a constant $C_{\textrm{mult}}>0$
such that Condition (3) of Definition \ref{def:admissible families}
is satisfied. Thus theorem \ref{thm:(Continuity-property-of} proves
the claim. 

Note that there is nothing special about $s=0$ and $s=1$; in general
given any $0\leq s_{0}\leq s_{1}\leq1$ such that $\Sigma_{s_{0}}$
and $\Sigma_{s_{1}}$ are non-degenerate and $a,b\notin\mathcal{A}(\Sigma_{s_{0}},\alpha)\cup\mathcal{A}(\Sigma_{s_{1}},\alpha)$,
this construction gives a map \[
\Psi_{s_{0}}^{s_{1}}:CF^{(a,b)}(A_{\mathfrak{f}_{s_{0}}},\alpha)\rightarrow CF^{(a,b)}(A_{\mathfrak{f}_{s_{1}}},\alpha)
\]
inducing a map\[
\psi_{s_{0}}^{s_{1}}:HF^{(a,b)}(A_{\mathfrak{f}_{s_{0}}},\alpha)\rightarrow HF^{(a,b)}(A_{\mathfrak{f}_{s_{1}}},\alpha)
\]

In fact, we can say rather more about the homomorphisms $(\psi_{s_{0}}^{s_{1}})$.
As with Theorem \ref{thm:(Continuity-property-of} itself, these two
properties follow from standard Floer homological methods. See for
instance \cite[Section 4.4]{BiranPolterovichSalamon2003} or \cite[Section 3.2.3]{Ginzburg2007}. 
\begin{enumerate}
\item Firstly, the maps $(\psi_{s_{0}}^{s_{1}})$ are actually \textbf{independent}
of choice of $(\Sigma_{s})$ in the following sense. Suppose $(\widetilde{\Sigma}_{s})_{s\in[0,1]}$
is another family of fibrewise starshaped hypersurfaces satisfying
the three conditions above, with corresponding definining Hamiltonians
$(\widetilde{F}_{s})$. Let \[
\widetilde{\ell}:=\min_{s\in[0,1]}\ell(\widetilde{\Sigma}_{s}).
\]
Suppose that%
\footnote{This caveat is added solely to ensure that the relevant homology groups
are well defined.%
} $f\in\bigcap_{r>0}\mathcal{F}(\widetilde{\ell}/12,r)$. Set $\widetilde{\mathfrak{f}}_{s}:=(\widetilde{F}_{s},f,1,0)$.
Then $\{(\widetilde{\mathfrak{f}}_{s}),a,b,0\}$ is also an $\alpha$-admissible
family, and hence gives rise to another family of chain maps $(\widetilde{\Psi}_{s_{0}}^{s_{1}})$.
These chain maps are chain homotopic to the original chain maps, and
hence they induce the same map on homology. 
\item The induced maps $(\psi_{s_{0}}^{s_{1}})$ enjoy the following \textbf{functorial
properties} whenever they are defined:\[
\psi_{s_{0}}^{s_{2}}=\psi_{s_{1}}^{s_{2}}\circ\psi_{s_{0}}^{s_{1}}\ \ \ \mbox{whenever}\ \ \ 0\leq s_{0}\leq s_{1}\leq s_{2}\leq1;
\]
\[
\psi_{s_{0}}^{s_{0}}=\mathbb{1}.
\]

\end{enumerate}
The proof of the next lemma requires a little more work, but is by
now standard.
\begin{lem}
\label{lem:its an iso}Suppose in addition that $a,b\notin\mathcal{A}(\Sigma_{s},\alpha)$
for \textbf{\emph{all}} $s\in[0,1]$. Then the homomorphism $\psi_{0}^{1}$
is actually an isomorphism. \end{lem}
\begin{proof}
Let $\rho_{s},T_{s}>0$ be the constants for $F_{s}$ from Corollary
\ref{cor:the linfty cor} (note $\rho_{s}$ and $T_{s}$ depend continuously
on $s$, cf. Remark \ref{rem:Observe-the constants}). Let\[
\rho_{1}:=\min\left\{ \min_{s}\rho_{s},\frac{a}{4\max_{s}T_{s}}\right\} .
\]
Our choice of $f$ guarantees that there exists $0<\varepsilon<a/4$
and $A>0$ such that \begin{equation}
f(\eta)=\eta\ \ \ \mbox{for all }\eta\geq\frac{a-4\varepsilon}{6};\label{eq:ass 1}
\end{equation}
\begin{equation}
Af'(-A)>\frac{b-a+\varepsilon}{\rho_{1}}.\label{eq:ass 2}
\end{equation}
Shrinking $\varepsilon$ if necessary, we may assume in addition that
\begin{equation}
\mathcal{A}(\Sigma_{s},\alpha)\cap[a,a+\varepsilon]=\mathcal{A}(\Sigma_{s},\alpha)\cap[b,b+\varepsilon]=\emptyset\label{eq:spectrum}
\end{equation}
for all $s\in[0,1]$. Now choose $R>2a+2b+A+1$ and let $\bar{f}$
denote an $R$-truncation of $f$ (see Section \ref{sub:Truncating-the-function}).
Set $\mathfrak{g}_{s}:=(F_{1-s},\bar{f},1,0)$. Our choice of $R$
implies that for every $s\in[0,1]$ such that $\Sigma_{s}$ is non-degenerate,
the Floer homology $HF^{(a,b)}(A_{\mathfrak{g}_{s}},\alpha)$ is well
defined, and moreover \[
HF^{(a,b)}(A_{\mathfrak{f}_{0}},\alpha)\cong HF^{(a,b)}(A_{\mathfrak{g}_{1}},\alpha);
\]
\[
HF^{(a,b)}(A_{\mathfrak{f}_{1}},\alpha)\cong HF^{(a,b)}(A_{\mathfrak{g}_{0}},\alpha).
\]
Now we compute that for $u\in\mathcal{N}(\nabla A_{\mathfrak{g}_{s}},\alpha)$
and $-\infty\leq s_{0}\leq s_{1}\leq\infty$, \begin{align*}
\Delta_{s_{0}}^{s_{1}}(u)= & -\int_{s_{0}}^{s_{1}}\bar{f}(\eta)\int_{0}^{1}\left(\partial_{s}F_{\beta(1-s)}(x)\right)dtds\\
 & \leq2R\sup_{s\in[0,1]}\left\Vert \partial_{s}F_{s}\right\Vert _{\infty}.
\end{align*}
Choose $N\in\mathbb{N}$ and a subdivision $0<i_{0}<i_{1}<\dots<i_{N}=1$
such that \begin{equation}
\max_{0\leq p\leq N-1}\left|i_{p+1}-i_{p}\right|\leq\frac{\varepsilon}{2R\sup_{s\in[0,1]}\left\Vert \partial_{s}F_{s}\right\Vert _{\infty}},\label{eq:choice of subdivision}
\end{equation}
and such that $\Sigma_{i_{p}}$ is non-degenerate for each $p=0,1,\dots,N$.
Now set\[
\mathfrak{g}_{s}^{p}:=\left(F_{(1-s)i_{p+1}-si_{p}},\bar{f},1,0\right).
\]
We claim that $\{(\mathfrak{g}_{s}^{p}),a,b,\varepsilon\}$ is an
$\alpha$-admissible family for each $p=0,1,\dots,N-1$. Indeed, Condition
(1) of Definition \ref{def:admissible families} is obviously satisfied,
and Condition (2) is satisfied by \eqref{eq:choice of subdivision}.
Finally, the reader is invited to check that our two assumptions \eqref{eq:ass 1}
and \eqref{eq:ass 2} together with the equations \eqref{eq:en eq},
\eqref{eq:sup eq} and \eqref{eq:inf eq} mean that the proof of Proposition
\ref{pro:useful dichotomy} goes through to ensure that Condition
(3) is satisfied for each $p=0,1,\dots,N-1$. 

As a result, Theorem \ref{thm:(Continuity-property-of} implies that
for each $p=0,1,\dots,N-1$ there exists a chain map \[
\Phi_{i_{p+1}}^{i_{p}}:CF^{(a,b)}(A_{\mathfrak{f}_{i_{p+1}}},\alpha)\rightarrow CF^{(a+\varepsilon,b+\varepsilon)}(A_{\mathfrak{f}_{i_{p}}},\alpha)
\]
inducing a homomorphism\[
\phi_{i_{p+1}}^{i_{p}}:HF^{(a,b)}(A_{\mathfrak{f}_{i_{p+1}}},\alpha)\rightarrow HF^{(a+\varepsilon,b+\varepsilon)}(A_{\mathfrak{f}_{i_{p}}},\alpha).
\]
Next, note that \eqref{eq:spectrum} and \eqref{eq:inclusionquotientmap}
imply that \[
i:HF^{(a,b)}(A_{\mathfrak{f}_{i_{p}}},\alpha)\cong HF^{(a+\varepsilon,b+\varepsilon)}(A_{\mathfrak{f}_{i_{p}}},\alpha)\ \ \ \mbox{for all }p=0,1,\dots,N,
\]
and consequently we may think of $\phi_{i_{p+1}}^{i_{p}}$ as a map
\[
\phi_{i_{p+1}}^{i_{p}}:HF^{(a,b)}(A_{\mathfrak{f}_{i_{p+1}}},\alpha)\rightarrow HF^{(a,b)}(A_{\mathfrak{f}_{i_{p}}},\alpha).
\]
It is now easy to see from the two properties about the continuation
maps given just before the statement of the lemma that $\phi_{i_{p+1}}^{i_{p}}$
is an isomorphism with inverse given by $\psi_{i_{p}}^{i_{p+1}}$.
It thus follows that if\[
\phi_{1}^{0}:=\phi_{i_{N}}^{i_{N-1}}\circ\dots\circ\phi_{i_{2}}^{i_{1}}\circ\phi_{i_{1}}^{i_{0}},
\]
then $\phi_{1}^{0}$ is the desired inverse to $\psi_{0}^{1}$. 
\end{proof}
We will be interested in a slight generalization of this. 
\begin{prop}
\label{prop:staircase}Suppose we are given two smooth strictly decreasing
families $(a_{s})_{s\in[0,1]}$ and $(b_{s})_{s\in[0,1]}$ such that
$\ell/2<a_{s}<b_{s}<\infty$ for all $s\in[0,1]$ and such that $a_{s},b_{s}\notin\mathcal{A}(\Sigma_{s},\alpha)$
for all $s\in[0,1]$. Then there exists a chain map \[
\Theta_{0}^{1}:CF^{(a_{0},b_{0})}(A_{\mathfrak{f}_{0}},\alpha)\rightarrow CF^{(a_{1},b_{1})}(A_{\mathfrak{f}_{1}},\alpha)
\]
 inducing an isomorphism \[
\theta_{0}^{1}:HF^{(a_{0},b_{0})}(A_{\mathfrak{f}_{0}},\alpha)\rightarrow HF^{(a_{1},b_{1})}(A_{\mathfrak{f}_{1}},\alpha).
\]
Moreover the following diagram commutes:\[
\xymatrix{HF^{(a_{0},b_{0})}(A_{\mathfrak{f}_{0}},\alpha)\ar[rr]^{\psi_{0}^{1}}\ar[dr]_{\theta_{0}^{1}} &  & HF^{(a_{0},b_{0})}(A_{\mathfrak{f}_{1}},\alpha)\\
 & HF^{(a_{1},b_{1})}(A_{\mathfrak{f}_{1}},\alpha)\ar[ur]_{i}
}
\]
where \[
i:HF^{(a_{1},b_{1})}(A_{\mathfrak{f}_{1}},\alpha)\rightarrow HF^{(a_{0},b_{0})}(A_{\mathfrak{f}_{`}},\alpha)
\]
is the map from \eqref{eq:inclusionquotientmap}.\end{prop}
\begin{proof}
The trick here is to use a {}``staircase'' method to deal with the
fact that the endpoints are changing. This is explained in detail
in \cite[p18]{MacariniSchlenk2010}, but the general idea is the following.
There exists $N\in\mathbb{N}$ and sequences \[
0=i_{0}<i_{1}<\dots<i_{N}=1;
\]
\[
0=j_{0}<j_{1}<\dots<j_{N}=1;
\]
\[
0=k_{0}<k_{1}<\dots<k_{N}=1
\]
such that for all $p\in\{0,1,\dots,N-1\}$, $\Sigma_{k_{p}}$ is non-degenerate
and \begin{equation}
\mathcal{A}(\Sigma_{s},\alpha)\cap[a_{i_{p+1}},a_{i_{p}}]=\mathcal{A}(\Sigma_{s},\alpha)\cap[b_{j_{p+1}},b_{j_{p}}]=\emptyset\ \ \ \mbox{for all }s\in[k_{p},k_{p+1}].\label{eq:staircase}
\end{equation}
We already know from the previous lemma how to build isomorphisms\[
\psi_{p}:HF^{(a_{i_{p}},b_{j_{p}})}(\Sigma_{k_{p}},\alpha)\rightarrow HF_{\mathcal{}}^{(a_{i_{p}},b_{j_{p}})}(\Sigma_{k_{p+1}},\alpha),
\]
and \eqref{eq:staircase} implies that\[
HF^{(a_{i_{p}},b_{i_{p}})}(\Sigma_{k_{p+1}},\alpha)\cong HF_{\mathcal{}}^{(a_{i_{p+1}},b_{j_{p+1}})}(\Sigma_{k_{p+1}},\alpha).
\]
Thus we obtain isomorphisms \[
\theta_{p}:HF_{\mathcal{}}^{(a_{i_{p}},b_{j_{p}})}(\Sigma_{k_{p}},\alpha)\rightarrow HF^{(a_{i_{p+1}},b_{j_{p+1}})}(\Sigma_{k_{p+1}},\alpha),
\]
and the proposition follows with \[
\theta_{0}^{1}:=\theta_{N-1}\circ\dots\circ\theta_{1}\circ\theta_{0}.
\]

\end{proof}

\subsection{Leaf-wise intersections }

$\ $\vspace{6 pt}

In this section we start with a single non-degenerate fibrewise starshaped
hypersurface $\Sigma$. As before, set $\ell:=\ell(\Sigma)$ and fix
once and for all a function $f\in\bigcap_{r>0}\mathcal{F}\left(\ell/12,r\right)$
and a defining Hamiltonian $F\in\mathcal{D}(\Sigma)$. Given a class
$\alpha\in[S^{1},M]$ and a map $\varphi\in\mbox{Ham}_{c}(T^{*}M,\omega)$,
let us denote by $n_{\Sigma,\alpha}(\varphi,(a,b))$ the number of
positive leaf-wise intersections points of $\varphi$ in $\Sigma$
that belong to $\alpha$ and have time-shift $T\in(a,b)$. The following
lemma explains the link between the Floer homology of a suitable perturbed
Rabinowitz action functional $A_{F^{\chi},f}^{H}$ and the number
of positive leaf-wise intersections of $\varphi$. The proof is immediate
from Lemma \ref{lem:key lemma}.2 and Theorem \ref{thm:admissable yes}.
\begin{lem}
\label{lem:lw intersections and frfh}Suppose $\varphi\in\mbox{\emph{Ham}}_{c}(T^{*}M,\omega)$
is generated by $H\in\mathcal{H}$. Choose $\chi\in\mathcal{X}_{0}$
and set $\mathfrak{f}:=(F,f,\chi,H)$. Fix $\ell/2<a<b<\infty$ such
that $H\in\mathcal{H}(a/2)$ and $a,b\notin\mathcal{A}(A_{\mathfrak{f}},\alpha)$.
If $\mathfrak{f}\in\mathfrak{F}_{0,\textrm{\emph{reg}}}''$ (which
we can assume is the case for a generic $\varphi$) then $HF^{(a,b)}(A_{\mathfrak{f}},\alpha)$
is well defined. Moreover, provided $\varphi$ has no periodic leaf-wise
intersection points (which again, we may assume is the case for a
generic $\varphi$ by Proposition \ref{prop:-generically no plwip})
one has \[
n_{\Sigma,\alpha}(\varphi,(a,b))\geq\dim\, HF^{(a,b)}(A_{\mathfrak{f}},\alpha).
\]

\end{lem}
Now set $\mathfrak{g}:=(F,f,\chi,0)$, and note that $\mathfrak{g}\in\mathfrak{F}_{0,\textrm{reg}}'$.
Our next application of continuation homomorphisms is to interpolate
between the Floer homology of the perturbed Rabinowitz action functional
$A_{\mathfrak{f}}$ and the non-perturbed one $A_{\mathfrak{g}}$.
This lemma is a simple consequence of Theorem \ref{thm:(Continuity-property-of}.
\begin{lem}
\label{lem:continutations with H}Assume in addition that \[
a-\left\Vert H\right\Vert _{-},a+\left\Vert H\right\Vert _{+},b-\left\Vert H\right\Vert _{-},b+\left\Vert H\right\Vert _{+}\notin\mathcal{A}(\Sigma,\alpha),
\]
Thus both $HF^{(a-\left\Vert H\right\Vert _{-},b-\left\Vert H\right\Vert _{-})}(A_{\mathfrak{g}},\alpha)$
and $HF^{(a+\left\Vert H\right\Vert _{+},b+\left\Vert H\right\Vert _{+})}(A_{\mathfrak{g}},\alpha)$
are well defined. Assume moreover that not only is $H\in\mathcal{H}(a/2)$
but actually \begin{equation}
2\left\Vert H\right\Vert +\kappa(H)\leq\frac{a}{2}.\label{eq:assumption}
\end{equation}
Then there exists a commutative diagram\[
\xymatrix{HF^{(a-\left\Vert H\right\Vert _{-},b-\left\Vert H\right\Vert _{-})}(A_{\mathfrak{g}},\alpha)\ar[rr]\ar[dr] &  & HF^{(a+\left\Vert H\right\Vert _{+},b+\left\Vert H\right\Vert _{+})}(A_{\mathfrak{g}},\alpha)\\
 & HF^{(a,b)}(A_{\mathfrak{f}},\alpha)\ar[ur]
}
\]
\end{lem}
\begin{proof}
Let us first build the continuation map \begin{equation}
HF^{(a-\left\Vert H\right\Vert _{-},b-\left\Vert H\right\Vert _{-})}(A_{\mathfrak{g}},\alpha)\rightarrow HF^{(a,b)}(A_{\mathfrak{f}},\alpha).\label{eq:going one way}
\end{equation}
Set \[
\mathfrak{f}_{s}:=(F,f,\chi,sH)\ \ \ \mbox{for }s\in[0,1].
\]
We will verify that $\{(\mathfrak{f}_{s}),a-\left\Vert H\right\Vert _{-},b-\left\Vert H\right\Vert _{-},\left\Vert H\right\Vert _{-}\}$
forms an $\alpha$-admissible family in the sense of Definition \ref{def:admissible families}.
Condition (1) is satisfied by assumption. Suppose $u\in\mathcal{N}(\nabla A_{\mathfrak{f}_{s}},\alpha)$
and $-\infty\leq s_{0}\leq s_{1}\leq\infty$. This time we have \begin{align*}
\Delta_{s_{0}}^{s_{1}}(u) & =-\int_{s_{0}}^{s_{1}}\beta'(s)\int_{0}^{1}H(t,x)dtds\\
 & \leq\int_{0}^{1}\beta'(s)\left\Vert H\right\Vert _{-}ds\\
 & =\left\Vert H\right\Vert _{-}.
\end{align*}
Thus Condition (2) is satisfied. The reader may check that the stronger
assumption \eqref{eq:assumption} implies that the proof of Proposition
\ref{pro:useful dichotomy} goes through to provide the necessary
constant $C_{\textrm{mult}}>0$ to satisfy Condition (3). The existence
of the map \eqref{eq:going one way} now follows from Theorem \ref{thm:(Continuity-property-of}.
The second map is defined similarly. 
\end{proof}
Now set $\mathfrak{h}:=(F,f,1,0)$. We now want to interpolate between
the Floer homology of $A_{\mathfrak{g}}$ and the Floer homology of
$A_{\mathfrak{h}}$. 
\begin{lem}
\label{lem:chi}Suppose $\ell/2<a<b<\infty$ satisfy $a,b\notin\mathcal{A}(\Sigma,\alpha)$.
Then there exists an isomorphsim \[
HF^{(a,b)}(A_{\mathfrak{g}},\alpha)\rightarrow HF^{(a,b)}(A_{\mathfrak{h}},\alpha).
\]

\end{lem}
This lemma is proved in a similar fashion (in fact it's slightly easier)
to the proof of Lemma \ref{lem:its an iso}, and as such we omit the
proof. Putting the results of this section together we deduce:
\begin{cor}
\label{cor:lw bound}Assume the hypotheses of Lemma \ref{lem:continutations with H}.
Then \[
n_{\Sigma,\alpha}(\varphi,(a,b))\geq\mbox{\emph{rank}}\left\{ i:HF^{(a-\left\Vert H\right\Vert _{-},b-\left\Vert H\right\Vert _{-})}(A_{\mathfrak{h}},\alpha)\rightarrow HF^{(a+\left\Vert H\right\Vert _{+},b+\left\Vert H\right\Vert _{+})}(A_{\mathfrak{h}},\alpha)\right\} .
\]

\end{cor}

\section{\label{sec:The-convex-case}The convex case}

\subsection{\label{sub:-estimates-for}$L^{\infty}$ estimates for Hamiltonians
that are not constant outside a compact set}

$\ $\vspace{6 pt}

Throughout this section, assume $\Sigma\subseteq T^{*}M$ is a\textbf{
strictly fibrewise convex }non-degenerate fibrewise starshaped hypersurface.
By this we mean a fibrewise starshaped hypersurface with the additional
property that for each $q\in M$ the hypersurface $\Sigma\cap T_{q}^{*}M$
in $T_{q}^{*}M$ has a positive definite second fundamental form.
Let $\ell:=\ell(\Sigma)$, and fix once and for all a function $f\in\bigcap_{r>0}\mathcal{F}(\ell/12,r)$. 

For each $q\in M$, let $r_{q}:T_{q}^{*}M\rightarrow\mathbb{R}$ denote
the function that is homogeneous of degree $2$ and satisfies $r_{q}|_{\Sigma\cap T_{q}^{*}M}\equiv1$.
The function $(q,p)\mapsto r_{q}(p)$ is $C^{1}$ on all of $T^{*}M$,
but not necessarily smooth at the zero section. In order to correct
this, let $\rho:\mathbb{R}\rightarrow\mathbb{R}$ denote a smooth
function such that $\rho(s)=0$ for $s\leq\varepsilon$, and $\rho'(s)>0$
for $s>\varepsilon$, and $\rho(s)=s$ for $s\geq2\varepsilon$, where
$\varepsilon$ is some sufficiently small positive number. Then define
$F:T^{*}M\rightarrow\mathbb{R}$ by \begin{equation}
F(q,p):=\frac{1}{2}(\rho(r_{q}(p))-1).\label{eq:nice H}
\end{equation}
If $(q,p)\in\Sigma$ then \[
\lambda(X_{F}(q,p))=\omega(Y(q,p),X_{F}(q,p))=d_{(q,p)}F(Y(q,p))=r_{q}(p)=1.
\]
Of course, the function $F$ is not compactly supported, and thus
$F\notin\mathcal{D}(\Sigma)$, and hence $F$ cannot a priori be used
to compute the $\mathcal{F}$-Rabinowitz Floer homology of $\Sigma$.
In order to make it compactly supported, we truncate it at infinity.
Given $R>1$, let $F_{R}:T^{*}M\rightarrow\mathbb{R}$ denote a function
such that $F_{R}(q,p)=F(q,p)$ on $\{F\leq R-1\}$ and such that $F_{R}(q,p)=R$
on $\{F\geq R+1\}$. Then $F_{R}\in\mathcal{D}(\Sigma)$, and the
aim of this section is to compute $HF^{(3\ell/4,\infty)}(A_{F_{R},f},\alpha)$
for each $\alpha\in[S^{1},M]$.

The following result is highly non-trivial, and is taken from \cite[Section 3]{AbbondandoloSchwarz2009}
(the function $f$ makes no difference here, given that we know a
priori\emph{ }that the $\eta$-component of elements $u\in\mathcal{M}^{(a,b)}(\nabla_{J}A_{F_{R_{1}},f})$
are uniformly bounded). 
\begin{thm}
\label{thm:quadratic is compact}Let $S$ denote a fibrewise starshaped
hypersurface such that $D(\Sigma)\subseteq D^{\circ}(S)$ and such
that $\mbox{\emph{supp}}(X_{F})\subseteq D^{\circ}(S)$. Choose $J\in\mathcal{J}(S)$.
Choose $0<a<b<\infty$ such that $a,b\notin\mathcal{A}(\Sigma,\alpha)$.
Then there exists $R_{1},R_{0}>1$ with $R_{1}>R_{0}+1$ such that
if $u=(x,\eta)\in\mathcal{M}^{(a,b)}(\nabla_{J}A_{F_{R_{1}},f})$
then $x(\mathbb{R}\times S^{1})\subseteq\{F\leq R_{0}\}$. 
\end{thm}
In other words, as far as the gradient flow lines $u\in\mathcal{M}^{(a,b)}(\nabla_{J}A_{F,f})$
are concerned, we might as well not have truncated $F$ at all. This
result is not obvious; although Lemma \ref{lem:why convex good} implies
that there certainly exists $R_{2}>0$ such that if $u=(x,\eta)\in\mathcal{M}^{(a,b)}(\nabla_{J}A_{F_{R_{1}},f})$
then $x(\mathbb{R}\times S^{1})\subseteq\{F\leq R_{2}\}$, there is
absolutely no reason at all why we should have $R_{1}>R_{2}+1$. In
order to prove this result, one must first show that one can obtain
$L^{\infty}$ bounds for the Hamiltonian $F$ \textbf{without}\emph{
}first truncating it at infinity, and then show that these bounds
are unaffected if we then subsequently truncate $F$ at some sufficiently
large $R>0$. This last statement is only true because we restrict
to the action interval $(a,b)$. In other words, this proves the Floer
homology $HF^{(a,b)}(A_{F,f},\alpha)$ is well defined if we use the
Hamiltonian $F$, and moreover with a little more work this shows
that the Floer homology $HF^{(a,b)}(A_{F,f},\alpha)$ is isomorphic
to Floer homology $HF^{(a,b)}(A_{F_{R},f},\alpha)$. An alternative
proof of Theorem \ref{thm:quadratic is compact} is given in \cite[Section 6]{Merry2010a}.
Anyway, because of Theorem \ref{thm:quadratic is compact}, we may
as well work directly with the Hamiltonian $F$, rather than truncating
it at infinity. This is crucial for Theorem \ref{thm:precise version of theorem e}
below.

\subsection{The $f$-free time action functional}

$\ $\vspace{6 pt}

The Hamiltonian $F$ has positive definite fibrewise second differential,
and thus the \textbf{Fenchel transform} $L:TM\rightarrow\mathbb{R}$
is well defined. Explicitly, $L$ is the unique Lagrangian on $TM$
defined by \[
L(q,v):=\max_{p\in T_{q}M}\left\{ p(v)-F(q,p)\right\} .
\]
The \textbf{Legendre transformation} associated to $L$ is the diffeomorphism
$TM\cong T^{*}M$ given by $(q,v)\mapsto\left(q,\frac{\partial L}{\partial v}(q,v)\right)$.
One can recover $F$ from $L$ via\[
F(q,p)=\frac{\partial L}{\partial v}(q,v)(v)-L(q,v)\ \ \ \mbox{where }\frac{\partial L}{\partial v}(q,v)=p.
\]
Fix a Riemannian metric $g$ on $M$ for the remainder of this section.
There exist constants $c_{0},c_{1}>0$ such that for all $(q,v)\in TM$,
\begin{equation}
d_{v}^{2}(L|_{T_{q}M})\geq c_{0}\mathbb{1};\label{eq:aseq1}
\end{equation}
\begin{equation}
\left|\nabla_{vv}L(q,v)\right|\leq c_{1},\ \ \ \left|\nabla_{vq}L(q,v)\right|\leq c_{1}(1+\left|v\right|),\ \ \ \left|\nabla_{qq}L(q,v)\right|\leq c_{1}(1+\left|v\right|^{2}),\label{eq:aseq2}
\end{equation}
where $\nabla_{vv}$, $\nabla_{vq}$ and $\nabla_{qq}$ denote the
components of the Hessian of $L$ associated to the horizontal-vertical
splitting of $TTM$ induced by $g$. See \cite[Section 10]{AbbondandoloSchwarz2009}.

Define the \textbf{$f$-free time action functional} $S_{L,f}:\Lambda M\times\mathbb{R}\rightarrow\mathbb{R}$
by \[
S_{L,f}(q,\eta):=f(\eta)\int_{0}^{1}L\left(q,\frac{\dot{q}}{f(\eta)}\right)dt.
\]
Denote by $\mbox{Crit}(S_{L,f})$ the set of critical points of $S_{L,f}$.
We wish to do Morse theory with $S_{L,f}$, and as such we will work
with the completion $\widetilde{\Lambda}M$ of $\Lambda M$ in the
Sobolev $W^{1,2}$-norm. Given $a>0$ and $\alpha\in[S^{1},M]$ let
us abbreviate \begin{equation}
\mathbb{S}_{\alpha}^{a}:=\{(q,\eta)\in\widetilde{\Lambda}_{\alpha}M\times\mathbb{R}\,:\, S_{L,f}(q,\eta)<a\}.\label{eq:Saalpha}
\end{equation}

It is convenient to define\[
E_{L}(q,v):=\frac{\partial L}{\partial v}(q,v)(v)-L(q,v);
\]
one calls $E_{L}$ the \textbf{energy} of $L$. If $\frac{\partial L}{\partial v}(q,v)=p$
then $F(q,p)=E_{L}(q,v)$. 

Here is another way to interpret the elements of $\mbox{Crit}(S_{L,f})$.
Given $(q,\eta)\in\Lambda M\times\mathbb{R}$, let $\gamma:\mathbb{R}/f(\eta)\mathbb{Z}\rightarrow M$
denote the curve \[
\gamma(t):=q(t/f(\eta)).
\]
Then $(q,\eta)\in\mbox{Crit}(S_{L,f})$ if and only if $\gamma$ satisfies
the \textbf{Euler-Lagrange equations}\emph{ }for $L$:\begin{equation}
\frac{d}{dt}\frac{\partial L}{\partial v}(\gamma(t),\dot{\gamma}(t))=\frac{\partial L}{\partial q}(\gamma(t),\dot{\gamma}(t)),\label{eq:EL eq}
\end{equation}
and has energy equal to 0:\[
E_{L}(\gamma(t),\dot{\gamma}(t))\equiv0.
\]
The condition that $\Sigma$ is non-degenerate translates to the following
statement about the critical points of $S_{L.f}$:
\begin{lem}
Every critical point $(q,\eta)$ of $S_{L,f}$ is \textbf{\emph{non-degenerate}}\emph{
}in the sense that the space of non-zero Jacobi vector fields along
the corresponding solution $\gamma$ of \eqref{eq:EL eq} is one dimensional,
spanned by $(\dot{\gamma},0)$.
\end{lem}
Let us denote by $i_{S_{L,f}}(q,\eta)$ the \textbf{Morse index} of
a critical point $(q,\eta)$ of $S_{L,f}$. Since $L$ is a fibrewise
strictly convex superlinear Lagrangian, the Morse index $i_{S_{L.f}}(q,\eta)$
is finite for every $(q,\eta)\in\mbox{Crit}(S_{L,f})$ \cite{Duistermaat1976}.
The following lemma clarifies the relationship between the functionals
$S_{L,f}$ and $A_{F,f}$.
\begin{lem}
\label{lem:key lemma-1}There exists a map $\mathcal{Z}=\mathcal{Z}(L,f):\widetilde{\Lambda}M\times\mathbb{R}\rightarrow\widetilde{\Lambda}T^{*}M\times\mathbb{R}$
such that $(\pi_{*}\times\mathbb{1})\circ\mathcal{Z}=\mathbb{1}$
(where $\pi_{*}:\widetilde{\Lambda}T^{*}M\rightarrow\widetilde{\Lambda}M$
is the induced map $(\pi_{*}(x))(t):=\pi(x(t))$), and such that $\mathcal{Z}$
restricts to define a bijection $\mbox{\emph{Crit}}(S_{L.f})\rightarrow\mbox{\emph{Crit}}(A_{F,f})$.
Moreover, given any $(x,\eta)\in\widetilde{\Lambda}T^{*}M\times\mathbb{R}$,
we have \[
A_{F,f}(x,\eta)\leq S_{L,f}(\pi\circ x,\eta))
\]
 with equality if and only if $(x,\eta)=\mathcal{Z}(\pi\circ x,\eta)$. 

Finally, the map $\mathcal{Z}$ preserves the grading: for any $(q,\eta)\in\mbox{\emph{Crit}}(S_{L,f})$,
if $\mathcal{Z}(q,\eta)=:(x,\eta)$ then\[
i_{S_{L,f}}(q,\eta)=\mu_{\textrm{\emph{CZ}}}^{\tau}(x;f(\eta)F).
\]
\end{lem}
\begin{proof}
The map $\mathcal{Z}$ is defined by \[
\mathcal{Z}(q,\eta):=\left(\left(q,\frac{\partial L}{\partial v}\left(q,\dot{q}\right)\right),\eta\right).
\]
See \cite[Lemma 5.1]{AbbondandoloSchwarz2009} or \cite[Lemma 4.1]{Merry2010a}.
The last statement follows from \cite[Section 1.3]{MerryPaternain2010}.
The key ingredient is Duistermaat's \textbf{Morse index theorem} \cite{Duistermaat1976}.
\end{proof}
As mentioned above, one would like to do Morse theory with $S_{L,f}$.
There are two issues that need to be sorted before we can proceed.
The first problem is that in general the functional $S_{L,f}$ is
\textbf{not}\emph{ }of class $C^{2}$ on $\widetilde{\Lambda}M\times\mathbb{R}$.
Nevertheless, one has the following result, which is due to Abbondandolo
and Schwarz \cite[Theorem 4.1]{AbbondandoloSchwarz2009a} (see also
the discussion before Proposition $11.2$ in \cite{AbbondandoloSchwarz2009}).
\begin{prop}
Let $f\in\mathcal{F}$. Then there exists a \textbf{\emph{smooth pseudo-gradient}}\emph{
}for $S_{L,f}$ on $\widetilde{\Lambda}M\times\mathbb{R}$. In other
words, there exists a smooth vector field $V$ on $\widetilde{\Lambda}M\times\mathbb{R}$
such that:
\begin{enumerate}
\item $V$ is bounded;
\item $d_{(q,\eta)}S_{L,f}(V(q,\eta))\geq\frac{1}{2}\min\left\{ 1,\left\Vert d_{(q,\eta)}S_{L,f}\right\Vert _{g}\right\} $
for all $(q,\eta)\in\widetilde{\Lambda}M\times\mathbb{R}$;
\item the set $\mbox{\emph{Crit}}(V)$ of zeros of $V$ coincides with $\mbox{\emph{Crit}}(S_{L,f})$,
and the linearization of $V$ at a rest point $(q,\eta)$ of $V$
agrees with the Hessian of $S_{L,f}$ at $(q,\eta)$.
\end{enumerate}
\end{prop}
Secondly, we need to verify that we can choose a pseudo-gradient $V$
such that the pair $(S_{L,f},V)$ satisfies the \textbf{Palais-Smale
condition}. Recall that we say that the pair $(S_{L,f},V)$ satisfies
the Palais-Smale condition at the level \emph{$T\in\mathbb{R}$ }if
every sequence $(q_{i},\eta_{i})\subseteq\widetilde{\Lambda}M\times\mathbb{R}$
such that $S_{L,f}(q_{i},\eta_{i})\rightarrow T$ and $d_{(q_{i},\eta_{i})}S_{L,f}(V(q_{i},\eta_{i}))\rightarrow0$
admits a convergent subsequence. The fact that $(S_{L,f},V)$ satisfies
the Palais-Smale condition at any $T>0$ is essentially a consequence
of the fact that the \textbf{Ma\~n\'e critical value}\emph{ }$c(L)$
of $L$ is negative. Let us first recall the definition of the Ma\~n\'e
critical value. 
\begin{defn}
Let $K:TN\rightarrow\mathbb{R}$ denote a fibrewise strictly convex
and superlinear Lagrangian. Define the \textbf{action}\emph{ }$\mathbb{A}_{K}$
of $K$ to be the functional \[
\mathbb{A}_{K}:\left\{ \gamma:[0,T]\rightarrow N,\ \gamma\ \mbox{absolutely continuous,}\, T>0\right\} \rightarrow\mathbb{R};
\]
\[
\mathbb{A}_{K}(\gamma):=\int_{0}^{T}K(\gamma(t),\dot{\gamma}(t))dt.
\]
The\emph{ }\textbf{Ma\~n\'e critical value}\emph{ }$c(K)$ of $K$
is the real number defined by \[
c(K):=\inf\left\{ k\in\mathbb{R}\,:\,\mathbb{A}_{K+k}(\gamma)\geq0\ \forall\,\mbox{a.c. closed curves }\gamma\mbox{ defined on }[0,T],\ \forall\, T>0\right\} .
\]

\end{defn}
The next lemma follows straight from the definition.
\begin{lem}
\label{lem:bounded below}Suppose $c(K)\leq0$. Then for any $f\in\mathcal{F}$
it holds that \[
\inf_{(q,\eta)\in\widetilde{\Lambda}M\times\mathbb{R}}S_{K,f}(q,\eta)>-\infty.
\]

\end{lem}
In our case, the Ma\~n\'e critical value is strictly negative.
\begin{lem}
The Ma\~n\'e critical value of $L$ is strictly negative.\end{lem}
\begin{proof}
The proof is based on the following alternative characterization of
the critical value, which is due to Contreras, Iturriaga, Paternain
and Paternain \cite{ContrerasIturriagaPaternainPaternain1998}. Suppose
$K:TM\rightarrow\mathbb{R}$ is a fibrewise strictly convex superlinear
Lagrangian. Then $K$ is the Fenchel transform of a unique Hamiltonian
$P:T^{*}M\rightarrow\mathbb{R}$. Then \[
c(K)=\inf_{u\in C^{\infty}(M)}\sup_{q\in M}P(q,d_{q}u).
\]
In our case since $D(\Sigma)=D(F^{-1}(0))$ contains the zero section,
taking $u$ to be a constant function we have \[
c(L)\leq\inf_{u\in C^{\infty}(M)}\sup_{q\in M}F(q,d_{q}u)\leq\sup_{q\in M}F(q,0_{q})<0.
\]

\end{proof}
The following theorem is essentially taken from \cite[Proposition 3.8 and 3.12]{ContrerasIturriagaPaternainPaternain2000,Contreras2006};
see also \cite{Benci1986}.
\begin{thm}
\label{thm:Let--forPS}Let $V$ denote a smooth pseudo-gradient for
$S_{L,f}$. Then the pair $(S_{L,f},V)$ satisfies the Palais-Smale
condition at the level $T$ on $\widetilde{\Lambda}M\times\mathbb{R}$
for any $T>0$. \end{thm}
\begin{rem}
In fact, if $\alpha\ne0$ then the pair $(S_{L,f},V)$ satisfies the
Palais-Smale condition on $\widetilde{\Lambda}_{\alpha}M\times\mathbb{R}$
even at the level $T=0$.\end{rem}
\begin{proof}
\emph{(of Theorem \ref{thm:Let--forPS})}

Suppose we are given a sequence $(q_{i},\eta_{i})\subseteq\widetilde{\Lambda}M\times\mathbb{R}$
such that $S_{L,f}(q_{i},\eta_{i})\rightarrow T$ for some $T>0$
and $d_{(q_{i},\eta_{i})}S_{L,f}(V(q_{i},\eta_{i}))\rightarrow0$.
Passing to a subsequence we may assume that \begin{equation}
0\leq S_{L,f}(q_{i},\eta_{i})\leq C,\ \ \ \left\Vert d_{(q_{i},\eta_{i})}S_{L,f}\right\Vert _{g}\leq\frac{1}{i}\ \ \ \mbox{for all }i\in\mathbb{N},\label{eq:ass}
\end{equation}
where $C$ is some positive constant. We first check that $(\eta_{i})$
is uniformly bounded below. Equations \eqref{eq:aseq1} and \eqref{eq:aseq2}
imply that there exist constants $d_{0},d_{1},d_{2},d_{3}>0$ such
that\[
d_{0}\left|v\right|^{2}-d_{1}\leq L(q,v)\leq d_{2}\left|v\right|^{2}+d_{3}\ \ \ \mbox{for all }(q,v)\in TM.
\]
Compactness of $M$ implies, up to passing to a subsequence, that
$\lim_{i}q_{i}(0)=q_{0}$ for some $q_{0}\in M$. Write $\gamma_{i}(t):=q_{i}(t/f(\eta_{i}))$,
so that $\gamma_{i}:\mathbb{R}/f(\eta_{i})\mathbb{Z}\rightarrow M$.
We will write $l_{i}$ and $e_{i}$ for the \textbf{length }and \textbf{energy
}of the curves $\gamma_{i}$, given by \[
l_{i}:=\int_{0}^{f(\eta_{i})}\left|\dot{\gamma}_{i}(t)\right|dt,\ \ \ e_{i}:=\int_{0}^{f(\eta_{i})}\frac{1}{2}\left|\dot{\gamma}_{i}(t)\right|^{2}dt.
\]
The Cauchy-Schwarz inequality implies that \begin{equation}
l_{i}^{2}\leq2f(\eta_{i})e_{i}.\label{eq:CS}
\end{equation}
Note that\begin{equation}
2d_{2}e_{i}+d_{3}f(\eta_{i})\geq S_{L,f}(q_{i},\eta_{i})=\int_{0}^{f(\eta_{i})}L(\gamma_{i},\dot{\gamma}_{i})dt\geq2d_{0}e_{i}-d_{1}f(\eta_{i}).\label{eq:sandwich}
\end{equation}
Assume for contradiction that $(\eta_{i})$ is not uniformly bounded
below. Up to passing to a subsequence, we may assume that $\eta_{i}\rightarrow-\infty$.
We will now prove that after passing to a further subsequence if necessary,
$e_{i}\rightarrow0$. Then \eqref{eq:sandwich} implies that $S_{L,f}(q_{i},\eta_{i})\rightarrow0$,
which contradicts the fact that $T>0$.

To see this we argue as follows. Firstly, \eqref{eq:ass} implies
that $(e_{i})$ is bounded. Since $(e_{i})$ is bounded, \eqref{eq:CS}
implies that $l_{i}\rightarrow0$, and thus up to passing to a subsequence,
we may assume that $q_{i}(S^{1})\subseteq U\cong\mathbb{R}^{n}$ (where
$n=\dim\, M$) for all $i$. Thus for the remainder of the proof we
work on $\mathbb{R}^{n}$. We can therefore speak of the partial derivatives
$L_{q}=\frac{\partial L}{\partial q}$ and $L_{v}=\frac{\partial L}{\partial v}$.
The assumptions \eqref{eq:aseq1} and\eqref{eq:aseq2} imply that
there exist constants $c_{2},c_{3},c_{4}>0$ such that in the coordinates
on $U$, \begin{equation}
c_{2}:=\sup_{q\in U,v\in T_{q}M}\frac{\left|L_{q}(q,v)\right|}{1+\left|v\right|^{2}}<\infty;\label{eq:b3}
\end{equation}
\[
c_{3}:=\sup_{q\in U,v\in T_{q}M}\frac{\left|L_{vq}(q,v)\right|}{1+\left|v\right|^{2}}<\infty;
\]
\[
c_{4}:=\inf_{q\in U,v\in T_{q}M}\frac{v\cdot L_{vv}(q,v)\cdot v}{\left|v\right|^{2}}>0.
\]
Arguing as in \cite[Lemma 3.2(ii)]{Contreras2006}, we have for any
two points $q,q'\in U$ and any $v\in T_{q}M$ that \begin{equation}
L_{v}(q,v)\cdot v\geq L_{v}(q',0)\cdot v-c_{3}\left|v\right|\left|q-q'\right|-c_{3}\left|v\right|^{2}\left|q-q'\right|+c_{4}\left|v\right|^{2}.\label{eq:gonzalo}
\end{equation}
Let $\xi_{i}(t):=q_{i}(t)-q_{i}(0)$, so that $(\xi_{i},0)\in T_{(q_{i},\eta_{i})}(\widetilde{\Lambda}\mathbb{R}^{n}\times\mathbb{R})$.
Put $\zeta_{i}(t):=\xi_{i}(t/f(\eta_{i}))$, so that $\dot{\zeta}_{i}(t)=\dot{\gamma}_{i}(t)$.
Then \eqref{eq:ass} implies that \begin{equation}
\left|d_{(q_{i},\eta_{i})}S_{L,f}(\xi_{i},0)\right|\leq\frac{1}{i}\left\Vert (\xi_{i},0)\right\Vert _{g}\leq\frac{1}{i}\sqrt{2f(\eta_{i})e_{i}}.\label{eq:bound above}
\end{equation}
Next, a straightforward computation (see \cite[p331]{Contreras2006})
tells us that\[
d_{(q_{i},\eta_{i})}S_{L,f}(\xi_{i},0)=\int_{0}^{f(\eta_{i})}\left(L_{q}(\gamma_{i},\dot{\gamma}_{i})\zeta_{i}+L_{v}(\gamma_{i},\dot{\gamma}_{i})\dot{\zeta}_{i}\right)dt.
\]
We apply \eqref{eq:b3} and \eqref{eq:gonzalo} with $(q,v)=(\gamma_{i},\dot{\gamma}_{i})$
and $q'=\gamma_{i}(0)$ to obtain: \begin{align*}
d_{(q_{i},\eta_{i})}S_{L,f}(\xi_{i},0) & \geq-c_{2}\int_{0}^{f(\eta_{i})}\left(1+\left|\dot{\gamma}_{i}\right|^{2}\right)\left|\gamma_{i}-\gamma_{i}(0)\right|dt+\left(\int_{0}^{f(\eta_{i})}L_{q}(\gamma_{i}(0),0)\cdot\dot{\gamma}_{i}dt\right)\\
 & -c_{3}\int_{0}^{f(\eta_{i})}\left|\dot{\gamma}_{i}\right|\left|\gamma_{i}-\gamma_{i}(0)\right|dt-c_{3}\int_{0}^{f(\eta_{i})}\left|\dot{\gamma}_{i}\right|^{2}\left|\gamma_{i}-\gamma_{i}(0)\right|dt+2c_{4}e_{i}\\
 & \geq-c_{2}l_{i}f(\eta_{i})+0-c_{3}l_{i}^{2}-2(c_{2}+c_{3})l_{i}e_{i}+2c_{4}e_{i}.
\end{align*}
Combining this last equation with \eqref{eq:bound above} and dividing
through by $\sqrt{f(\eta_{i})}$, we see that \[
-c_{2}l_{i}\sqrt{f(\eta_{i})}-c_{3}\frac{l_{i}^{2}}{\sqrt{f(\eta_{i})}}-2(c_{2}+c_{3})\frac{l_{i}e_{i}}{\sqrt{f(\eta_{i})}}+2c_{4}\frac{e_{i}}{\sqrt{f(\eta_{i})}}\leq\frac{1}{i}\sqrt{2e_{i}}.
\]
Equation \eqref{eq:CS} implies the first three terms on the left-hand
side are bounded. Since the right-hand side is also bounded, we see
that \[
\frac{e_{i}}{\sqrt{f(\eta_{i})}}
\]
 is bounded, and thus $e_{i}\rightarrow0$ as claimed.

We have now proved that $(\eta_{i})$ is bounded below. Next, we check
that $(\eta_{i})$ is bounded above. Indeed, we have\[
S_{L,f}(q_{i},\eta_{i})=S_{L+c(L),f}(q_{i},\eta_{i})-c(L)f(\eta_{i}).
\]
Since $f(\eta)\equiv\eta$ on $[a,\infty)$, and since $S_{L+c(L),f}$
is bounded below (Lemma \ref{lem:bounded below}) and $c(L)<0$, we
must have $(\eta_{i})$ bounded above.

Thus $(\eta_{i})$ is a bounded sequence, and thus up to passing to
a subsequence, we may assume that $\eta_{i}\rightarrow\eta$ for some
$\eta\in\mathbb{R}$. From this point on the proof is essentially
identical to \cite[Proposition 3.12]{Contreras2006}, and thus we
will omit further details.
\end{proof}
Note that Lemma \ref{lem:key lemma-1} implies that $\mbox{Crit}(S_{L,f},\alpha)\cap\mathbb{S}_{\alpha}^{3\ell/4}=\emptyset$.
Using this observation together with Theorem \ref{thm:Let--forPS},
and arguing as in \cite[Proposition 11.3]{AbbondandoloSchwarz2009}
we conclude:
\begin{cor}
\label{cor:def retract}The pair $(\widetilde{\Lambda}_{\alpha}M\times\mathbb{R},\mathbb{S}_{\alpha}^{3\ell/4})$
is homotopy equivalent to $(\Lambda_{\alpha}M,\emptyset)$ if $\alpha\ne0$,
and to $(\Lambda_{0}M,M)$ if $\alpha=0$, where we view $M\subseteq\Lambda_{0}M$
as the constant loops.
\end{cor}
We are now in a position of being able to define the Morse homology
of $S_{L,f}$. Suppose $\ell/2<a<b<\infty$ and $\alpha\in[S^{1},M]$.
The \textbf{relative Morse homology }of $(S_{L,f},\alpha)$ on the
action interval $(a,b)$ will be well defined whenever $a,b\notin\mathcal{A}(\Sigma,\alpha)$.
Fix a smooth pseudo-gradient $V$ of $S_{L,f}$. Pick a Morse function
$m:\mbox{Crit}(S_{L,f})\rightarrow\mathbb{R}$, and denote by $\mbox{Crit}(m)\subseteq\mbox{Crit}(S_{L,f})$
the set of critical points of $m$. Define an augmented grading $i:\mbox{Crit}(m)\rightarrow\mathbb{Z}$
by \[
i(w):=i_{S_{L,f}}(w)+i_{m}(w),\ \ \ w=(q,\eta),
\]
where $i_{m}(w)\in\{0,1\}$ is the Morse index of $w$, seen as a
critical point of $m$. Let $\mbox{Crit}_{k}^{(a,b)}(m,\alpha):=\{w\in\mbox{Crit}(m)\cap\mbox{Crit}^{(a,b)}(S_{L,f},\alpha)\,:\, i(w)=k\}$.
Given $k\in\mathbb{Z}$, let \[
CM_{k}^{(a,b)}(S_{L,f},\alpha):=\mbox{Crit}_{k}^{(a,b)}(m,\alpha)\otimes\mathbb{Z}_{2}.
\]
Fix a Riemannian metric $g_{0}$ on $\mbox{Crit}(\ell)$ for which
the negative gradient flow $\phi_{t}^{-\nabla m}$ of $m$ is Morse-Smale.
Then up to a perturbation of the pseudo-gradient vector field $V$
and the metric $g_{0}$, we obtain a boundary operator \[
\partial:CM_{k}^{(a,b)}(S_{L,f},\alpha)\rightarrow CM_{k-1}^{(a,b)}(S_{L,f},\alpha)
\]
satisfying $\partial^{2}=0$. We denote by $HM^{(a,b)}(S_{L,f},\alpha)$
the homology of this chain complex. As our notation suggests, the
homology is independent of the auxiliary choices we made when defining
the chain complex and its boundary operator. The \textbf{Morse homology
theorem} tells us that there exists an isomorphism \begin{equation}
\theta^{(a,b)}:HM^{(a,b)}(S_{L,f},\alpha)\rightarrow H(\mathbb{S}_{\alpha}^{b},\mathbb{S}_{\alpha}^{a}).\label{eq:morse homology iso}
\end{equation}
See \cite{AbbondandoloMajer2006,AbbondandoloSchwarz2009} for more
details.

\subsection{The Abbondandolo-Schwarz isomorphism}

$\ $\vspace{6 pt}

Fix $\ell/2<a<b<\infty$ and $\alpha\in[S^{1},M]$ such that $a,b\notin\mathcal{A}(\Sigma,\alpha)$.
Both the Morse homology $HM^{(a,b)}(S_{L,f},\alpha)$ and the Floer
homology $HF^{(a,b)}(A_{F,f},\alpha)$ are defined. We now relate
the two chain complexes via an {}``Abbondandolo-Schwarz'' chain
map%
\footnote{The reader may wonder why we defy the standard alphabetical ordering
naming convention here. This chain map is called {}``$\Phi_{\textrm{SA}}$''
because it goes from the chain complex of the {}``$S$'' functional
to the chain complex of the {}``$A$'' functional. There is another
chain map that goes the other way round; this one is denoted by {}``$\Phi_{\textrm{AS}}$''
See \cite[Section 7]{AbbondandoloSchwarz2009} or \cite[Theorem 5.1]{Merry2010a}.
The chain map $\Phi_{\textrm{AS}}$ is not used in the present paper
however.  %
} $\Phi_{\textrm{SA}}^{(a,b)}:CM^{(a,b)}(S_{L,f},\alpha)\rightarrow CF^{(a,b)}(A_{F,f},\alpha)$. 
\begin{rem}
In the discussion that follows for simplicity we will suppress the
fact that we are in a Morse-Bott situation. In reality we need to
consider flow lines with cascades in the construction below, and we
need to choose the Morse functions $m$ on $\mbox{Crit}(S_{L,f})$
and $h$ on $\mbox{Crit}(A_{F,f})$ to satisfy certain compatibility
conditions. This extra subtlety is dealt with fully in \cite{AbbondandoloSchwarz2009},
and there are no changes whatsoever in the present situation. 

This chain map $\Phi_{\textrm{SA}}^{(a,b)}$ is defined by counting
solutions of the following mixed problem: Given a critical point $w$
of $\ell$ and a critical point $v$ of $h$, we consider the moduli
space of maps $u=(x,\eta):[0,\infty)\rightarrow\widetilde{\Lambda}T^{*}M\times\mathbb{R}$
that solve the Rabinowitz Floer equation $\partial_{s}u+\nabla A_{F,f}(u)=0$
on $(0,\infty)$ and satisfy the boundary conditions (a) $\lim_{s\rightarrow\infty}u(s)=v$
and (b) $(\pi\circ x(0),\eta(0))\in W^{u}(w;-V)$. Lemma \ref{lem:key lemma-1},
together with its differential version allows one to prove the necessary
compactness for such solutions. This method was invented by Abbondandolo
and Schwarz in \cite{AbbondandoloSchwarz2006}, and extended to Rabinowitz
Floer homology by the same authors in \cite{AbbondandoloSchwarz2009}.
The upshot is the following theorem, whose proof involves no ideas
not already present in either of the two aforementioned references,
and thus will be omitted.\end{rem}
\begin{thm}
\label{thm:precise version of theorem e}There exists a chain complex
isomorphism \[
\Phi_{\textrm{\emph{SA}}}^{(a,b)}:CM^{(a,b)}(S_{L,f},\alpha)\rightarrow CF^{(a,b)}(A_{F,f},\alpha)
\]
of the form \[
\Phi_{\textrm{\emph{SA}}}^{(a,b)}w=\sum_{v\in\textrm{\emph{Crit}}^{(a,b)}(h,\alpha)}n_{\textrm{\emph{SA}}}(w,v)v\ \ \ \forall v\in\mbox{\emph{Crit}}^{(a,b)}(h,\alpha),
\]
where $n_{\textrm{\emph{SA}}}(w,v)\in\mathbb{Z}_{2}$ is zero if $i(w)\ne\mu(v)$
or if $S_{L,f}(w)\leq A_{F,f}(v)$, unless $v=\mathcal{Z}(w)$, in
which case $n_{\textrm{\emph{SA}}}(w,\mathcal{Z}(w))=1$.
\end{thm}
Denote by $\phi_{\textrm{SA}}^{(a,b)}=[\Phi_{\textrm{SA}}^{(a,b)}]$
the induced map on homology. The Abbondandolo-Schwarz map is functorial
in the following sense. Fix $\ell/2<a<b<\infty$ and $\ell/2<c<d<\infty$,
such that $a\leq c$, $b\leq d$, and $a,b,c,d\notin\mathcal{A}(\Sigma,\alpha)$.
Then the following diagram commutes, where the horizontal maps are
all induced by inclusion, and $\theta^{(a,b)}$ denotes the isomorphism
\eqref{eq:morse homology iso}, \[
\xymatrix{HF{}^{(a,b)}(A_{F,f},\alpha)\ar[r] & HF^{(c,d)}(A_{F,f},\alpha)\\
HM^{(a,b)}(S_{L,f},\alpha)\ar[r]\ar[d]_{\theta^{(a,b)}}\ar[u]^{\phi_{\textrm{SA}}^{(a,b)}} & HM^{(c,d)}(S_{L,f},\alpha)\ar[d]^{\theta^{(c,d)}}\ar[u]_{\phi_{\textrm{SA}}^{(c,d)}}\\
H(\mathbb{S}_{\alpha}^{b},\mathbb{S}_{\alpha}^{a})\ar[r] & H(\mathbb{S}_{\alpha}^{d},\mathbb{S}_{\alpha}^{c})
}
\]

In order to fit in with our earlier notation \eqref{eq:the map Z-1},
let us denote by \begin{equation}
\widetilde{Z}(a,b):H(\mathbb{S}_{\alpha}^{a},\mathbb{S}_{\alpha}^{3\ell/4})\rightarrow H(\Lambda_{\alpha}M\times\mathbb{R},\mathbb{S}_{\alpha}^{b}),\label{eq:the map Z}
\end{equation}
the map on singular homology induced from inclusion. As before $\widetilde{Z}(a,b)=0$
if $a<b$.

Anyway, passing to the direct limit, we conclude the following result,
which is the main one of this section.
\begin{thm}
\label{thm:computation in the convex case}In the situation above
one has:
\begin{enumerate}
\item $HF^{(3\ell/4,\infty)}(A_{F,f},\alpha)\cong\begin{cases}
H(\Lambda_{\alpha}M) & \alpha\ne0,\\
H(\Lambda_{0}M,M) & \alpha=0.
\end{cases}$
\item Suppose $a,b>3\ell/4$ with $a,b\notin\mathcal{A}(\Sigma,\alpha)$.
Then it holds that\begin{multline*}
\mbox{\emph{rank}}\left\{ Z(a,b):HF^{(3\ell/4,a)}(A_{F,f},\alpha)\rightarrow HF_{\mathcal{}}^{(b,\infty)}(A_{F,f},\alpha)\right\} =\\
\mbox{\emph{rank}}\left\{ \widetilde{Z}(a,b):H(\mathbb{S}_{\alpha}^{a},\mathbb{S}_{\alpha}^{3\ell/4})\rightarrow H(\Lambda_{\alpha}M\times\mathbb{R},\mathbb{S}_{\alpha}^{b})\right\} .
\end{multline*}

\end{enumerate}
\end{thm}

\section{\label{sec:Proof-of-Theorem A}Proof of Theorem A}

In this section we complete the proof of Theorem A from the Introduction.
To begin with however we introduce the following definition.
\begin{defn}
Given a starshaped hypersurface $\Sigma\subseteq T^{*}M$, $T>0$
and $\alpha\in[S^{1},M]$, define\[
\delta(\Sigma,T,\alpha):=\inf\left\{ \left|T'-T''\right|\,:\, T'\ne T'',\ T',T''\in\mathcal{A}(\Sigma,\alpha)\cap[0,T]\right\} .
\]
If $\Sigma$ is non-degenerate then $\delta(\Sigma,T,\alpha)>0$ for
every (finite) $T>0$ and $\alpha\in[S^{1},M]$.
\end{defn}
We now proceed with the proof of Theorem A. Let $\Sigma$ denote a
non-degenerate fibrewise starshaped hypersurface. Let $g$ denote
a \textbf{bumpy }Riemannian metric on $M$ such that the unit disc
bundle $D(S_{g}^{*}M)$ is contained in $D^{\circ}(\Sigma)$, and
let us denote by $F_{g}:T^{*}M\rightarrow\mathbb{R}$ the Hamiltonian
\begin{equation}
F_{g}(q,p):=\frac{1}{2}\left(\left|p\right|_{g}^{2}-1\right).\label{eq:the ham fg}
\end{equation}
Asking $g$ to be bumpy is equivalent to asking that $S_{g}^{*}M=F_{g}^{-1}(0)$
is non-degenerate in the sense of Definition \ref{def:non degen}.
A theorem of Abraham \cite{Abraham1970} (first properly proved by
Anosov in \cite{Anosov1982}) states that the set $\mathcal{R}_{\textrm{bumpy}}(M)$
of all bumpy Riemannian metrics on $M$ is a residual subset of the
set $\mathcal{R}(M)$ of all Riemannian metrics on $M$, so such metrics
certainly exist%
\footnote{Note that this result does \textbf{not }follow from Theorem \ref{thm:nondeg is generic}
stated above.%
}. 
\begin{rem}
The point of choosing such a metric $g$ comes down to the fact that
we can compute the Floer homology $HF^{(a,\infty)}(A_{F_{g},f},\alpha)$
(see Theorem \ref{thm:computation in the convex case} above). Since
we proved Theorem \ref{thm:computation in the convex case} for any
strictly fibrewise convex\textbf{ }non-degenerate fibrewise starshaped
hypersurface $S$, we could equally well work with such any such hypersurface
$S$ satisfying $D(S)\subseteq D^{\circ}(\Sigma)$ rather than a unit
cotangent bundle. However for aesthetic reasons we prefer to work
with a unit contangent bundle, even if it means quoting the bumpy
metric theorem.
\end{rem}
Recall that $\mathcal{G}(\Sigma)\subseteq\mbox{Ham}_{c}(T^{*}M,\omega)$
denotes the generic subset of Hamiltonian diffeomorphisms $\varphi$
with no periodic leaf-wise intersection points (cf. Proposition \ref{prop:-generically no plwip}). 
\begin{defn}
\label{what we prove theorem a for}Let $\mathcal{O}(\Sigma)\subseteq\mathcal{G}(\Sigma)$
denote the set of Hamiltonian diffeomorphisms $\varphi\ne\mathbb{1}$
such that there exists $H\in\mathcal{R}(\Sigma)\subseteq\mathcal{H}$
(cf. Definition \ref{def:non degen for H}) that generates $\varphi$.
Since $\mathcal{G}(\Sigma)$ is a generic subset of $\mbox{Ham}_{c}(T^{*}M,\omega)$
and $\mathcal{R}(\Sigma)$ is a generic subset of $\mathcal{H}$,
$\mathcal{O}(\Sigma)$ is a generic subset of $\mbox{Ham}_{c}(T^{*}M,\omega)$. 
\end{defn}
We will prove Theorem A for Hamiltonian diffeomorphisms $\varphi\in\mathcal{O}(\Sigma)$.

\subsubsection*{Proof of Theorem A}

$\ $\vspace{6 pt}

Let $\phi_{t}^{Y}$ denote the flow of the Liouville vector field
$Y$. Given $t>0$ let $(S_{g}^{*}M)_{t}:=\phi_{t}^{Y}(S_{g}^{*}M)$,
so that $((S_{g}^{*}M)_{t})_{t\geq0}$ forms%
\footnote{Technically speaking this not quite the same as a concentric family
as defined in Section \ref{sub:Monotone-homotopies-and}, as the hypersurfaces
get larger as $t$ increases, not smaller.%
} a concentric family in the language of Section \ref{sub:Monotone-homotopies-and}.
Note that if \[
F_{g}^{t}(q,p):=\frac{1}{2}\left(\left|p\right|^{2}-e^{2t}\right),
\]
then $F_{g}^{t}\in\mathcal{D}((S_{g}^{*}M)_{t})$, and $(F_{g}^{t})_{t\geq0}$
satisfies $\partial_{t}F_{g}^{t}\leq0$. 

Let us now fix $\varphi\in\mathcal{O}(\Sigma)$ and $\alpha\in[S^{1},M]$.
Choose $s>0$ such that $D(\Sigma)\subseteq D^{\circ}((S_{g}^{*}M)_{s})$
and such that \[
0<e^{-s}(\mu(\varphi)-\left\Vert \varphi\right\Vert )<\frac{1}{2}\ell(S_{g}^{*}M,\alpha).
\]
Recall we defined the quantity $\mu(\varphi)=2\kappa(\varphi)+6\left\Vert \varphi\right\Vert $
in \eqref{eq:mu phi}, where $\kappa(\varphi)$ was defined in \eqref{eq:kappaphi},
and the Hofer norm $\left\Vert \varphi\right\Vert $ was defined in
\eqref{eq:hofer norm}. Recall also from the Introduction that for
any $H\in C_{c}^{\infty}(S^{1}\times T^{*}M,\mathbb{R})$, the value
of $\kappa(H)$ (cf. Definition \ref{Define-a-semi-norm}) depends
only on $\phi_{1}^{H}\in\mbox{Ham}_{c}(T^{*}M,\omega)$. 

Now fix $T>0$ such that \[
e^{-s}(T-\left\Vert \varphi\right\Vert )>2\mu(\varphi).
\]
Next we will some choose $H\in\mathcal{R}(\Sigma)$ generating $\varphi$
with $\left\Vert H\right\Vert -\left\Vert \varphi\right\Vert $ sufficiently
small. More precisely, we first ask that $\left\Vert \varphi\right\Vert \geq\frac{5}{6}\left\Vert H\right\Vert $,
and then in addition that\begin{equation}
0\leq e^{-s}(\left\Vert H\right\Vert -\left\Vert \varphi\right\Vert )\leq\min\left\{ \frac{1}{2}\ell(S_{g}^{*}M,\alpha),\delta(S_{g}^{*}M,\alpha,e^{-s}T),\mu(\varphi)\right\} .\label{eq:smallness condition}
\end{equation}
Set \[
\ell:=\min\left\{ \ell(\Sigma),e^{-s}\ell(S_{g}^{*}M)\right\} 
\]
and choose\[
f\in\bigcap_{r>0}\mathcal{F}\left(\frac{\ell}{12},r\right).
\]
Finally choose $F\in\mathcal{D}(\Sigma)$ and $\chi\in\mathcal{X}_{0}$.
Set \[
\mathfrak{f}:=(F,f,\chi,H),\ \ \ \mathfrak{g}:=(F,f,\chi,0),\ \ \ \mathfrak{h}:=(F,f,1,0),
\]
\[
\mathfrak{i}:=(F_{g},f,1,0),\ \ \ \mathfrak{j}:=(F_{g}^{s},f,1,0).
\]
We will tacitly assume that all the action values $\mu(\varphi)-\left\Vert H\right\Vert _{+},T-\left\Vert H\right\Vert _{+},\mu(\varphi),T$
that appear in the diagram below do not lie in the relevant action
spectrums, so that all the Floer homology groups are well defined.
We now splice together the various commutative diagrams from Section
\ref{sec:Continuation-homomorphisms} to create one big commutative
diagram (we omit all the $\alpha$'s for clarity):\[
\xymatrix@C=6pt{ &  & HF^{(\mu(\varphi)-\left\Vert H\right\Vert _{+},T-\left\Vert H\right\Vert _{+})}(A_{\mathfrak{f}})\ar[d]\\
 & HF_{\mathcal{}}^{(\mu(\varphi)-\left\Vert H\right\Vert ,T-\left\Vert H\right\Vert )}(A_{\mathfrak{g}})\ar[r]\ar[ur]\ar[d]_{\cong} & HF^{(\mu(\varphi),T)}(A_{\mathfrak{g}})\ar[d]^{\cong}\\
HF^{(\mu(\varphi)-\left\Vert H\right\Vert ,\mu(\varphi)-\left\Vert H\right\Vert )}(A_{\mathfrak{j}})\ar[r]\ar[d]_{\theta} & HF_{\mathcal{}}^{(\mu(\varphi)-\left\Vert H\right\Vert ,T-\left\Vert H\right\Vert )}(A_{\mathfrak{h}})\ar[r]\ar[d] & HF^{(\mu(\varphi),T)}(A_{\mathfrak{h}})\ar[d]\\
HF^{(e^{-s}(\mu(\varphi)-\left\Vert H\right\Vert ),e^{-s}(T-\left\Vert H\right\Vert ))}(A_{\mathfrak{i}})\ar[r]\ar[drr]_{Z} & HF^{(\mu(\varphi)-\left\Vert H\right\Vert ,T-\left\Vert H\right\Vert )}(A_{\mathfrak{i}})\ar[r] & HF_{\mathcal{}}^{(\mu(\varphi),T)}(A_{\mathfrak{i}})\ar[d]^{\iota}\\
 &  & HF_{\mathcal{}}^{(\mu(\varphi),\infty)}(A_{\mathfrak{i}})
}
\]

\noindent \begin{flushleft}
Here the top right-hand triangle is the commutative diagram from Lemma
\ref{lem:continutations with H}. For this to be well defined recall
we needed \[
2\left\Vert H\right\Vert +\kappa(H)\leq\mu(\varphi)-\left\Vert H\right\Vert _{+},
\]
and this is guaranteed by our requirement that $\left\Vert \varphi\right\Vert \geq\frac{5}{6}\left\Vert H\right\Vert $.
The square below comes from Lemma \ref{lem:chi}; the vertical maps
are isomorphisms. The map $\theta$ on the left-hand side is the map
from Proposition \ref{prop:staircase}; thus $\theta$ is an isomorphism.
The map $\iota$ in the bottom right is the map \eqref{eq:iota}.
Note that by \eqref{eq:smallness condition} one has \begin{align*}
HF^{(e^{-s}(\mu(\varphi)-\left\Vert H\right\Vert ),e^{-s}(T-\left\Vert H\right\Vert ))}(A_{\mathfrak{i}},\alpha) & \cong HF^{(3\ell/4,e^{-s}(T-\left\Vert H\right\Vert ))}(A_{\mathfrak{i}},\alpha)\\
 & \cong HF_{\mathcal{}}^{(3\ell/4,e^{-s}(T-\left\Vert \varphi\right\Vert ))}(A_{\mathfrak{i}},\alpha).
\end{align*}
Thus the diagonal map $Z$ at the bottom of the diagram is the map
\[
Z(e^{-s}(T-\left\Vert \varphi\right\Vert ),\mu(\varphi)):HF_{\mathcal{}}^{(3\ell/4,e^{-s}(T-\left\Vert \varphi\right\Vert ))}(A_{\mathfrak{i}},\alpha)\rightarrow HF^{(\mu(\varphi),\infty)}(A_{\mathfrak{i}},\alpha)
\]
 from \eqref{eq:the map Z-1}. Since $\theta$ and the two vertical
maps in the top-most square are isomorphisms, we can read off from
the diagram (see Corollary \ref{cor:lw bound}) that \begin{align*}
n_{\Sigma,\alpha}(\varphi,T) & \geq n_{\Sigma,\alpha}(\varphi,(\mu(\varphi)-\left\Vert H\right\Vert _{+},T-\left\Vert H\right\Vert _{+}))\\
 & \geq\mbox{rank}\left\{ HF_{\mathcal{}}^{(\mu(\varphi)-\left\Vert H\right\Vert ,T-\left\Vert H\right\Vert )}(A_{\mathfrak{g}},\alpha)\rightarrow HF^{(\mu(\varphi),T)}(A_{\mathfrak{g}},\alpha)\right\} \\
 & \geq\mbox{rank}\, Z(e^{-s}(T-\left\Vert \varphi\right\Vert ),\mu(\varphi)).
\end{align*}
By Theorem \ref{thm:computation in the convex case}, we have \[
\mbox{rank}\, Z(e^{-s}(T-\left\Vert \varphi\right\Vert ),\mu(\varphi))=\mbox{rank}\,\widetilde{Z}(e^{-s}(T-\left\Vert \varphi\right\Vert ),\mu(\varphi)),
\]
where $\widetilde{Z}(e^{-s}(T-\left\Vert \varphi\right\Vert ),\mu(\varphi))$
is the map \eqref{eq:the map Z}. Here the relevant free time action
functional is defined using the Lagrangian $L_{g}:TM\rightarrow\mathbb{R}$,
which by definition is the Fenchel transform of the Hamiltonian $F_{g}$
from \eqref{eq:the ham fg}, and is given by $L_{g}(q,v):=\frac{1}{2}\left(\left|v\right|_{g}^{2}+1\right)$.
\par\end{flushleft}

\noindent \begin{flushleft}
Recall we denote by $\mathcal{E}_{g}:\Lambda M\rightarrow\mathbb{R}$
the functional $\mathcal{E}_{g}(q):=\int_{0}^{1}\frac{1}{2}\left|\dot{q}\right|_{g}^{2}dt$,
and that we use the special notation \[
\Lambda_{\alpha}^{a}(M,g):=\left\{ q\in\Lambda_{\alpha}M\,:\,\mathcal{E}_{g}(q)\leq\frac{1}{2}a^{2}\right\} .
\]
Denote by $\mbox{pr}_{1}:\Lambda M\times\mathbb{R}\rightarrow\Lambda M$
the first projection, and given $a\in\mathbb{R}$ denote by $i_{a}:\Lambda M\rightarrow\Lambda M\times\mathbb{R}$
the map $i_{a}(q):=(q,a)$. We complete the proof of Theorem A with
the following elementary observation. 
\par\end{flushleft}
\begin{lem}
\noindent \begin{flushleft}
Suppose $a,b>3\ell/4$. \textup{\emph{Then if $\widetilde{Z}(a,b)$
denotes the map \eqref{eq:the map Z} then it holds that }}\begin{multline*}
\mbox{\emph{rank}}\left\{ \widetilde{Z}(a,b):H(\mathbb{S}_{\alpha}^{a},\mathbb{S}_{\alpha}^{3\ell/4})\rightarrow H(\Lambda_{\alpha}M\times\mathbb{R},\mathbb{S}_{\alpha}^{b})\right\} \geq\\
\mbox{\emph{rank}}\left\{ \iota:H(\Lambda_{\alpha}^{a}(M,g),\Lambda_{\alpha}^{3\ell/4}(M,g))\rightarrow H(\Lambda_{\alpha}M,\Lambda_{\alpha}^{2b}(M,g))\right\} .
\end{multline*}

\par\end{flushleft}\end{lem}
\begin{proof}
\noindent \begin{flushleft}
We first show that for any $c\geq3\ell/4$ (we will apply this with
$c=3\ell/4$ and $c=a$) we have\[
i_{c}\left(\Lambda_{g}^{c}(M,g)\right)\subseteq\mathbb{S}_{\alpha}^{c}.
\]
Indeed, for any $\eta\geq c$ one has \[
S_{L_{g},f}(q,\eta)=\frac{1}{\eta}\mathcal{E}_{g}(q)+\frac{\eta}{2},
\]
and hence \[
S_{L_{g},f}(i_{c}(q))=\frac{1}{c}\mathcal{E}_{g}(q)+\frac{c}{2}\leq\frac{1}{c}\cdot\frac{1}{2}c^{2}+\frac{c}{2}=c.
\]
Secondly we claim that \[
\mbox{pr}_{1}\left(\mathbb{S}_{\alpha}^{a}\right)\subseteq\Lambda_{\alpha}^{2a}(M,g).
\]
To see this, note that in general $S_{L_{g},f}(q,\eta)=\frac{1}{f(\eta)}\mathcal{E}_{g}(q)+\frac{f(\eta)}{2}$,
and thus if $S_{L_{g},f}(q,\eta)\leq a$ then as $\frac{1}{f(\eta)}\mathcal{E}_{g}(q)\geq0$
we have $f(\eta)\leq2a$, and hence \[
\mathcal{E}_{g}(q)=f(\eta)\left(S_{L_{g},f}(q,\eta)-\frac{f(\eta)}{2}\right)\leq2a(a-0)=2a^{2}.
\]
The result now follows from the observation that\[
\mbox{rank}\,\widetilde{Z}(a,b)\geq\mbox{rank}\,((\mbox{pr}_{1})_{*}\circ\widetilde{Z}(a,b)\circ(i_{a})_{*}).
\]

\par\end{flushleft}
\end{proof}
\bibliographystyle{amsplain}
\bibliography{C:/Users/Will/Desktop/willbibtex}

\end{document}